\documentclass[12pt, leqno]{amsart}  

\usepackage{graphicx}
\usepackage{amssymb,amsmath,amsthm,amscd,accents}
\usepackage{stackengine,scalerel}
\usepackage{mathrsfs}
\usepackage{lscape}
\usepackage{enumerate}
\usepackage{upgreek}
\usepackage{stmaryrd}
\usepackage[usenames,dvipsnames]{color}
\usepackage[colorlinks=true, pdfstartview=FitV,
 linkcolor=blue,citecolor=blue,urlcolor=blue]{hyperref}
\usepackage{tikz}
\tikzset{dynkdot/.style={circle,draw,scale=.38}}
\usetikzlibrary{arrows.meta,arrows}
\usepackage{bm}
\usepackage{marginnote}
\usepackage{verbatim}
\usepackage[normalem]{ulem}  
\usepackage{xspace}

\usepackage[all]{xy}

\allowdisplaybreaks[3]

\newcommand{\arxiv}[1]{\href{http://arxiv.org/abs/#1}{\texttt{arXiv:#1}}}

\newcommand{\nc}{\newcommand}
\numberwithin{equation}{section}

\nc{\eq}{\begin{myequation}} \nc{\eneq}{\end{myequation}}
\nc{\eqn}{\begin{myequationn}} \nc{\eneqn}{\end{myequationn}}

\nc{\hs}{\hspace*}
\nc{\ms}{\mspace}
\nc{\st}[1]{\{{#1}\}}

\nc{\qR}[1]{\ttq_{\mspace{-2mu}\raisebox{-.8ex}{${\scriptstyle{#1}}$}}}

\theoremstyle{plain}
\newtheorem{lemma}{Lemma}[section]
\newtheorem{proposition}[lemma]{Proposition}
\newtheorem{theorem}[lemma]{Theorem}

\newtheorem{corollary}[lemma]{Corollary}
\newtheorem{conjecture}[lemma]{Conjecture}

\theoremstyle{definition}
\newtheorem{remark}[lemma]{Remark}

\newtheorem{definition}[lemma]{Definition}

\newtheorem{condition}[lemma]{Condition}

\nc{\Def}{\begin{definition}}
\nc{\edf}{\end{definition}}
\nc{\mTh}{\begin{mtheorem}}
\nc{\enmth}{\end{mtheorem}}
\nc{\Th}{\begin{theorem}}
\nc{\enth}{\end{theorem}}
\nc{\Prop}{\begin{proposition}}
\nc{\enprop}{\end{proposition}}
\nc{\Lemma}{\begin{lemma}}
\nc{\enlemma}{\end{lemma}}
\nc{\Cor}{\begin{corollary}}
\nc{\encor}{\end{corollary}}
\nc{\Rem}{\begin{remark}}
\nc{\enrem}{\end{remark}}
\nc{\Conj}{\begin{conjecture}}
\nc{\enconj}{\end{conjecture}}

\nc{\ledot}{\mathrel{\le\ms{-11mu}\raisebox{.2ex}{$\cdot$}}}
\nc{\predot}{\mathrel{\preceq\ms{-9mu}\raisebox{.35ex}{$\centerdot$}}}

\newcommand{\Zq}{{\Z[q^{\pm 1/2}]}}

\renewcommand{\le}{\leqslant}
\renewcommand{\ge}{\geqslant}
\renewcommand{\preceq}{\preccurlyeq}

\newcommand{\head}{{\operatorname{hd}}}

\newcommand{\Ker}{{\operatorname{Ker}}}

\newcommand{\sotimes}{\mathop{\mbox{\normalsize$\bigotimes$}}\limits}

\newcommand{\hconv}{\mathbin{\scalebox{.9}{$\nabla$}}}
\newcommand{\sconv}{\mathbin{\scalebox{.9}{$\Delta$}}}

\newcommand{\seteq}{\mathbin{:=}}
\newcommand{\conv}{\mathop{\mathbin{\mbox{\large $\circ$}}}}
\newcommand{\soplus}{\mathop{\mbox{\normalsize$\bigoplus$}}\limits}
\newcommand{\stens}{\mathop{\mbox{\normalsize$\bigotimes$}}\limits}
\newcommand{\sbcup}{\mathop{\mbox{\normalsize$\bigcup$}}\limits}
\newcommand{\ev}{{\operatorname{ev}}}
\newcommand{\tens}{\mathop\otimes}
\newcommand{\Lto}{\longrightarrow}

\newcommand{\gmod}{\text{-}\mathrm{gmod}}

\newcommand{\ex}{\mathrm{ex}}
\newcommand{\fr}{\mathrm{fr}}

\newcommand{\g}{\mathfrak{g}}

\newcommand{\C}{\mathbb{C}}
\newcommand{\Q}{\mathbb{Q}}
\newcommand{\Z}{\mathbb{Z}\ms{1mu}}

\newcommand{\al}{{\ms{1mu}\alpha}}
\newcommand{\ep}{\epsilon}
\newcommand{\la}{\lambda}
\newcommand{\be}{{\ms{1mu}\beta}}
\newcommand{\ga}{\gamma}

\newcommand{\La}{\Lambda}
\newcommand{\Li}{\Lambda^\infty}
\newcommand{\vph}{\varphi}

\newcommand{\UpLa}{\Uplambda}

\newcommand{\upal}{\upalpha}
\newcommand{\upbe}{\upbeta}
\newcommand{\upga}{\upgamma}

\newcommand{\wt}{{\rm wt}}
\newcommand{\ch}{{\rm ch}}
\newcommand{\rev}{{\rm rev}}

\newcommand{\cl}{{\rm cl}}

\newcommand{\het}{{\rm ht}}

\newcommand{\de}{\mathfrak{d}}
\newcommand{\Dynkin}{\triangle}  

\newcommand{\im}{\imath}
\newcommand{\jm}{\jmath}
\newcommand{\oim}{{\overline{\im}}}
\newcommand{\ojm}{{\overline{\jm}}}

\newcommand{\ii}{ \textbf{\textit{i}}}

\newcommand{\Hom}{\operatorname{Hom}}
\newcommand{\Deg}{\operatorname{Deg}}

\newcommand{\tB}{\widetilde{B}}

\newcommand{\tX}{\widetilde{X}}
\newcommand{\tY}{\widetilde{Y}}
\newcommand{\tZ}{\widetilde{Z}}

\newcommand{\ta}{\widetilde{a}}
\newcommand{\tb}{\widetilde{b}}

\newcommand{\te}{\widetilde{e}}
\newcommand{\tf}{\widetilde{f}}

\newcommand{\tp}{\widetilde{p}\ms{3mu}}

\newcommand{\hA}{\widehat{A}}
\newcommand{\hB}{\widehat{B}}

\newcommand{\hd}{\mathrm{hd}}

\newcommand{\ocalD}{\overline{\calD}}
\newcommand{\obPhi}{{}^{\circ}\bPhi}

\newcommand{\uC}{\underline{C}}

\newcommand{\uL}{\underline{L}}
\newcommand{\uM}{\underline{M}}

\newcommand{\uw}{\underline{w}}

\newcommand{\frakA}{\mathfrak{a}}

\newcommand{\frakC}{\mathfrak{C}}

\newcommand{\frakH}{\mathfrak{H}}

\newcommand{\frakc}{\mathfrak{c}}

\newcommand{\tfrakc}{\widetilde{\frakc}}

\newcommand{\sfI}{\mathsf{I}}
\newcommand{\sfJ}{\mathsf{J}}
\newcommand{\sfK}{\mathsf{K}}

\newcommand{\sfM}{\mathsf{M}}

\newcommand{\sfP}{\mathsf{P}}
\newcommand{\sfQ}{\mathsf{Q}}

\newcommand{\sfS}{\mathsf{S}}
\newcommand{\sfT}{\mathsf{T}}

\newcommand{\sfV}{\mathsf{V}}
\newcommand{\sfW}{\mathsf{W}}

\newcommand{\sfg}{\mathsf{g}}

\newcommand{\sfn}{\mathsf{n}}

\newcommand{\sfs}{\mathsf{s}}

\newcommand{\bbA}{\mathbb{A}}
\newcommand{\bbB}{\mathbb{B}}

\newcommand{\bbD}{\mathbb{D}}

\newcommand{\bbF}{\mathbb{F}}
\newcommand{\bbG}{\mathbb{G}}

\newcommand{\bbK}{\mathbb{K}}

\newcommand{\bbT}{\mathbb{T}}

\newcommand{\obbA}{{}^\circ \hspace{-.1em}\bbA}

\newcommand{\obbK}{{}^\circ \hspace{-.1em}\bbK}

\newcommand{\bsa}{{\boldsymbol{a}}}
\newcommand{\bsb}{{\boldsymbol{b}}}

\newcommand{\bmu}{{\boldsymbol{\mu}}}

\newcommand{\bPhi}{{\boldsymbol{\Phi}}}

\newcommand{\bfA}{{ \Z[q^{\pm1/2}]  }}

\newcommand{\bfG}{\mathbf{G}}

\newcommand{\bfL}{\mathbf{L}}

\newcommand{\bfP}{\mathbf{P}}

\newcommand{\tbfG}{\widetilde{\bfG}}

\newcommand{\bfa}{\mathbf{a}}
\newcommand{\bfb}{\mathbf{b}}

\newcommand{\bfk}{\mathbf{k}}

\newcommand{\bfx}{\mathbf{x}}

\newcommand{\tbfx}{\widetilde{\bfx}}

\newcommand{\calA}{\mathcal{A}}

\newcommand{\calC}{\mathcal{C}}
\newcommand{\calD}{\mathcal{D}}
\newcommand{\calE}{\mathcal{E}}
\newcommand{\calF}{\mathcal{F}}
\newcommand{\calG}{\mathcal{G}}
\newcommand{\calH}{\mathcal{H}}
\newcommand{\calI}{\mathcal{I}}

\newcommand{\calK}{\mathcal{K}}
\newcommand{\calL}{\mathcal{L}}
\newcommand{\calM}{\mathcal{M}}
\newcommand{\calN}{\mathcal{N}}

\newcommand{\calQ}{\mathcal{Q}}
\newcommand{\calR}{\mathcal{R}}

\newcommand{\calU}{\mathcal{U}}

\newcommand{\calY}{\mathcal{Y}}

\newcommand{\hcalA}{\widehat{\calA}}

\newcommand{\scrA}{\mathscr{A}}

\newcommand{\scrC}{\mathscr{C}}
\newcommand{\scrD}{\mathscr{D}}

\newcommand{\scrS}{\mathscr{S}}

\newcommand{\ttB}{\mathtt{B}}

\newcommand{\ttP}{\mathtt{P}}
\newcommand{\ttQ}{\mathtt{Q}}

\newcommand{\ttS}{\mathtt{S}}

\newcommand{\ttb}{\mathtt{b}}

\newcommand{\ttl}{\mathtt{l}}

\newcommand{\ttq}{\mathtt{q}}
\newcommand{\ttr}{\mathtt{r}}
\nc{\bg}{\sigma}

\newcommand{\ttx}{\mathtt{x}}
\newcommand{\tty}{\mathtt{y}}
\newcommand{\ttz}{\mathtt{z}}

\newcommand{\rmE}{\mathrm{E}}

\newcommand{\rmG}{\mathrm{G}}

\newcommand{\rmL}{\mathrm{L}}

\newcommand{\rmP}{\mathrm{P}}
\newcommand{\rmQ}{\mathrm{Q}}

\newcommand{\hrmL}{\widehat{\rmL}}

\newlength{\mylength}
\newcommand*{\para}{%
  \rlap{\rotatebox{-30}{\rule[.05ex]{.4pt}{.77em}}}%
  \kern.04em%
  \rlap{\kern.36em\raisebox{0.649519052835em}{\rule{.6em}{.4pt}}}%
  \rule{.6em}{.4pt}\kern-.04em%
  \rotatebox{-30}{\rule[.05ex]{.4pt}{.77em}}}

\newcommand{\Rr}{\mathbf{r}}

\newcommand{\hDynkin}{{\widehat{\Dynkin}}}

\newcommand{\isoto}[1][]{\mathop{\xrightarrow%
[{\raisebox{.3ex}[0ex][.3ex]{$\scriptstyle{#1}$}}]%
{{\raisebox{-.6ex}[0ex][-.6ex]{$\mspace{2mu}\sim\mspace{2mu}$}}}}}

\newcommand{\rl}{\sfQ}
\newcommand{\weyl}{\sfW}
\newcommand{\lan}{\langle}
\newcommand{\ran}{\rangle}

\newcommand{\Aqn}{\calA_q(\sfn)}

\newcommand{\qt}[1]{\quad\text{#1}}
\newcommand{\qtq}[1][{and}]{\quad\text{{#1}}\quad}

\newcommand{\qtoq}[1][{or}]{\quad\text{{#1}}\quad}

\newcommand{\ee}{\end{enumerate}}
\newcommand{\bitem}{\begin{itemize}}
\newcommand{\eitem}{\end{itemize}}
\newcommand{\ben}{\begin{enumerate}[{\rm (1)}]}
\newcommand{\bnum}{\begin{enumerate}[{\rm (i)}]}
\newcommand{\bnump}{\begin{enumerate}[{\rm (i)$'$}]}
\newcommand{\bna}{\begin{enumerate}[{\rm (a)}]}
\newcommand{\bnA}{\begin{enumerate}[{\rm (A)}]}
\newcommand{\bc}{\begin{cases}}
\newcommand{\ec}{\end{cases}}
\newcommand{\ba}{\begin{array}}
\newcommand{\ea}{\end{array}}
\newcommand{\noi}{\noindent}
\newcommand{\snoi}{\smallskip \noindent}
\newcommand{\mnoi}{\medskip \noindent}

\newenvironment{myequation}
{\relax\setlength{\arraycolsep}{1pt}\begin{eqnarray}}
{\end{eqnarray}}
\newenvironment{myequationn}
{\relax\setlength{\arraycolsep}{1pt}\begin{eqnarray*}}
{\end{eqnarray*}}

\nc{\eqs}[1]{\underset{\raisebox{.4ex}[.7ex][0ex]{$\scriptstyle{#1}$}}{=}}

\newcommand{\Uqpg}{U_q'(\g)}
\newcommand{\Dcan}{{\bbD_\can}}
\newcommand{\DQ}{{\bbD_\calQ}}
\newcommand{\SC}{\sfS(\frakC)}

\newcommand{\rdual}{\scrD}
\newcommand{\ldual}{\scrD^{-1}}

\newcommand{\Ang}[1]{  \bigl\lan #1 \bigr\ran  }
\newcommand{\ang}[1]{\langle#1\rangle}
\newcommand{\aform}[1]{  ({ #1 })_\sfn  }

\newcommand{\bone}{\mathbf{1}}

\newcommand{\Es}{\rmE^\star}
\newcommand{\Esn}[1]{\rmE^{\star \hspace{0.1ex} #1}}

\newcommand{\hBi}{\hB(\infty)}
\newcommand{\LuphA}{\hLup}

\newcommand{\Gup}{\rmG^{{\rm up}}}
\newcommand{\Lup}{\rmL^{{\rm up}}}
\newcommand{\hLup}{\hrmL^{{\rm up}}}
\newcommand{\LupA}{\Lup\left(\Azn\right)}
\newcommand{\Azn}{\calA_\bfA(\sfn)}

\newcommand{\Iset}[1]{  \{ #1 \}_{i \in I} }
\newcommand{\Cg}{{\scrC_\g}}
\newcommand{\CgD}{{\scrC^\bbD_\g}}

\newcommand{\Cgz}{{\scrC^0_\g}}
\newcommand{\Csgz}{{\scrC^0_{\sfg^{(1)}}}}

\newcommand{\Dpair}{{\bbD,\hwc}}

\newcommand{\longepito}[1][]{\xymatrix@C=4ex{{}\ar@{->>}[r]^{#1}&{}}}
\newcommand{\Mn}{\mathbf{M}\ms{1mu}}
\newcommand{\hAform}[1]{\bl #1\br_{\hcalA}\ms{1mu}}
\newcommand{\hAm}[1]{  \hcalA[#1]  }
\newcommand{\hAmz}[1]{  \hcalA[#1]_\Zq  }
\newcommand{\hAz}{  \hcalA_\Zq  }
\newcommand{\bl}{\bigl(}
\newcommand{\br}{\bigr)}
\newcommand{\iinI}{ \im \in \sfI}
\newcommand{\iset}[1]{  \{ #1 \}_{\iinI}  }

\newcommand{\one}{{\bf{1}}}

\newcommand{\dpar}[1]{  ( \ms{-2mu} ( #1 ) \ms{-2mu} )  }
\newcommand{\dbracket}[1]{  [ \ms{-2mu}[ #1 ]\ms{-2mu}]  }

\newcommand{\Runiv}{\mathrm{R}^{\ms{1mu}\mathrm{univ}}}
\newcommand{\Rren}{\mathrm{R}^{\ms{1mu}\mathrm{ren}}} 
\newcommand{\oprod}{\prod^{\xrightarrow{}}}
\newcommand{\rprod}{\prod^{\xleftarrow{}}}
  \newcommand{\otens}{\stens\limits^{\xrightarrow{}}}
\newcommand{\Qdatum}{(\Dynkin,\sigma,\xi)} 

\newcommand{\usfwc}{{\uw_\circ}}

\newcommand{\usfwcp}{{\uw'_\circ}}
\newcommand{\hwc}{{\widehat{\uw}_\circ}}
\newcommand{\hwcp}{{\hwc'}}

\newcommand{\hwci}[1]{{\widehat{\uw}_{\circ #1}}}

\newcommand{\uii}{{\underline{\boldsymbol{\im}}}}

\newcommand{\tuii}{{\widetilde{\uii} }}
\newcommand{\tujj}{{\widetilde{\ujj} }}
\newcommand{\akete}[1][0ex]{\rule[{#1}]{0ex}{1ex}}
\nc{\ake}[1][2ex]{\rule[-.5ex]{0ex}{#1}}
\newcommand{\dtens}{\mathop{\mbox{\scalebox{1.1}{\akete[-.6ex]\normalsize$\bigotimes$}}}\limits}
\newcommand{\ujj}{{\underline{\boldsymbol{\jm}}}}
\newcommand{\TT}{ \textbf{\textit{T}}}
\newcommand{\FF}{ \ttP}

\newcommand{\Seq}{\mathrm{Seq}}
\newcommand{\PBW}{\mathrm{PBW}}

\newcommand{\kk}{{\Q(q^{1/2})}}

\newcommand{\Mseed}{{\sfS=(\{\sfM_i\}_{i\in\sfJ}, \tB;\sfJ,\sfJ_\ex)  }}

\newcommand{\Mpseed}{{\sfS^*=(\{\sfM_i\}_{i\in\sfJ^*}, \tB^*;\sfJ^*,\sfJ^*_\ex)  }}
\newcommand{\mseed}{{(\{\sfM_i\}_{i\in\sfJ}, \tB;\sfJ,\sfJ_\ex)  }}

\nc{\col}{\colon}
\nc{\ord}{\mathrm{ord}}
\nc{\catQ}{\scrC_\calQ}
\nc{\catD}{\scrC_\bbD}
\nc{\vs}{\vspace*}
\nc{\D}{\mathscr{D}}
\nc{\Proof}{\begin{proof}}
  \nc{\QED}{\end{proof}}

\nc{\tLa}{\widetilde{\Lambda}}
\nc{\ro}{{\rm(}}
\nc{\rf}{{\rm)}\xspace}
\nc{\Aut}{\mathrm{Aut}}
\nc{\can}{\mathrm{can}}
\nc{\Dc}{{\bbD_{\can}}}
\nc{\Cgo}{\scrC_{\mathfrak{g}}^{0}}
\nc{\bb}{\mathtt{b}}
\nc{\catC}{\scrC}
\nc{\braid}{\ttB}
\nc{\Ass}[1][\bbD]{\calE_{#1}}
\nc{\Ci}{C^\uii}
\nc{\Cj}{C^\ujj}
\nc{\condi}[1][K]{with $#1\cap\st{0,1}\not=\emptyset$\xspace}
\nc{\Di}[1][{\uii}]{{\bbD,{#1}}}
\nc{\Vi}[1][\uii]{V^{#1}}
\nc{\Pii}[1][\uii]{P^{#1}}
\nc{\Vdi}[1][{\Di}]{V^{#1}}
\nc{\Pdi}[1][{\Di}]{P^{#1}} 
\nc{\PBix}[1][{[l,r]}]{\Z_{\ge0}^{\oplus{#1}}}
\nc{\mcf}[1][{$[a,b]$}]{maximal commuting family of $i$-boxes in {#1}\xspace}
\nc{\nn}{\nonumber}
\nc{\bwr}{\mbox{\large$\wr$}}
\newcommand{\tch}{\widetilde{\ch}}
\nc{\tchDcan}{\ms{2mu}\tch_{\mspace{.1mu}\raisebox{-.4ex}{${\scriptstyle{\Dcan}}$}}}
\nc{\monoto}[1][]{\xymatrix@C=2ex{\ar@{>->}[r]^-{{#1}}&}\ms{-8mu}}

\usepackage[colorinlistoftodos]{todonotes}

\title[Monoidal categorification III]{Monoidal categorification and \\  quantum affine algebras III}

\author[M. Kashiwara]{Masaki Kashiwara}
\thanks{The research of M.\ Kashiwara
	was supported by Grant-in-Aid for Scientific Research (B)  23K20206,  
	Japan Society for the Promotion of Science.}
\address[M. Kashiwara]{%
Kyoto University Institute for Advanced Study, Research Institute
for Mathematical Sciences, Kyoto University, Kyoto 606-8502, Japan
}
\email[M. Kashiwara]{masaki@kurims.kyoto-u.ac.jp}

\author[M. Kim]{Myungho Kim}
\address[M. Kim]{Department of Mathematics, Kyung Hee University, Seoul 02447, Korea}
\email[M. Kim]{mkim@khu.ac.kr}
\thanks{The research of M.\ Kim was supported by the National Research Foundation of
Korea (NRF) Grant funded by the Korea government(MSIT)
(NRF-2020R1A5A1016126).}

\author[S.-j. Oh]{Se-jin Oh}
\thanks{ The research of S.-j.\ Oh was supported by the National Research Foundation of
	Korea (NRF) Grant funded by the Korea government(MSIT) (NRF-2022R1A2C1004045).}
\address[S.-j. Oh]{ Department of Mathematics, Sungkyunkwan University, Suwon, South Korea}
\email[S.-j. Oh]{sejin092@gmail.com}

\author[E. Park]{Euiyong Park}
\thanks{The research of E.\ Park was supported by the National Research Foundation of Korea (NRF) Grant funded by the Korea Government(MSIT)(RS-2023-00273425 and NRF-2020R1A5A1016126).}
\address[E. Park]{Department of Mathematics, University of Seoul, Seoul 02504, Korea}
\email[E. Park]{epark@uos.ac.kr}



\date{September 18, 2025}

\begin{document}

\begin{abstract}
Let $U_q'(\g)$ be an arbitrary quantum affine algebra of either untwisted or twisted type, and let $\Cgz$ be its Hernandez-Leclerc category. We denote by $\ttB$  the braid group determined by the simply-laced finite type Lie algebra $ \sfg$ associated with $U_q'(\g)$.
For any complete duality datum $\bbD$ and any sequence $ \uii $ of simple roots of $\sfg$, we 
construct the corresponding affine cuspidal modules and affine determinantial modules and study their key properties including T-systems.
Then, for any element $\ttb$ of the positive braid monoid  $\ttB^+$,
we introduce a distinguished subcategory $\scrC_\g^\bbD(\ttb)$ of $\Cgz$ categorifying the specialization of the bosonic extension $\hcalA (\ttb)$
at $q^{1/2}=1$ and investigate its  properties including the categorical PBW structure. We finally prove that the subcategory $\scrC_\g^\bbD(\ttb)$ provides a monoidal categorification of the (quantum) cluster algebra $\hcalA(\ttb)$, which significantly generalizes the earlier monoidal categorification developed by the authors.
\end{abstract}

\maketitle
\tableofcontents

\setcounter{section}{-1}
\section{Introduction}

This is the third paper in our series on \emph{monoidal categorifications} for cluster algebras arising from quantum affine algebras (\cite{KKOP20,KKOP24A}). 
Let $\Cgz$ be the \emph{Hernandez-Leclerc} category of a quantum affine algebra $U_q'(\g)$ which is a certain distinguished monoidal subcategory of the category $\Cg$ of finite-dimensional integrable  $U_q'(\g)$-modules (see 
\cite{HL10, HL16} and see also \cite{KKKOIV, KO19, OT19}, and
Section \ref{subsec: HL category}). 
The category $\Cgz$ possesses a rich and interesting structure including the rigidity, and  has been actively studied since its introduction. 
The category $\Cgz$ lies at the heart of the representation theory for $U_q'(\g)$ and is deeply connected with  various research areas including cluster algebras (see \cite{CP95, CH10, HL21} and references therein). 
The cluster algebraic approach to various subcategories of $\Cgz$ was introduced by Hernandez and Leclerc (\cite{HL10, HL16}). It turns out that the (quantum) Grothendieck rings of various distinguished subcategories  $\scrC_\ell$ and $\scrC_\g^-$, called the \emph{Hernandez-Leclerc subcategories}, of $\Cgz$ have (quantum) cluster algebra structures whose initial seeds arise from \emph{Kirillov-Reshetikhin} modules. Hernandez-Leclerc introduced the notion of monoidal categorification and studied subcategories  $\scrC_\ell$ and $\scrC_\g^-$ in the viewpoint of cluster algebras at the categorical level, which shed light on remarkable structural features of the Hernandez-Leclerc categories (see \cite{HL10, HL16} and see also \cite{FHOO, FHOO2, KKOP20,KKOP24A,Nak11}).

The \emph{quantum Grothendieck ring} $\calK_{\g;t} $ of $\Cgz$ defined via the \emph{$(q,t)$-characters} of modules in $\Cgz$ (\cite{Her04, Nak04, VV03}) has been studied from the ring-theoretic viewpoint.  
A ring presentation of $\calK_{\g;t} $ is discovered by Hernandez-Leclerc (\cite{HL15})  for simply-laced types and later by Fujita-Hernandez-Oh-Oya (\cite{FHOO}) 
for the remaining types.  This gives rise to the \emph{bosonic extension} $\hcalA$, which  is the associative $\Q(q^{\pm {1}/ {2}})$-algebra with infinitely many 
generators $f_{j,m}$ satisfying  the quantum Serre and the bosonic relations determined by  a \emph{generalized symmetrizable Cartan matrix} $C$ (see \cite{JLO1, KKOP24, OP24}  and see also Section \ref{Sec: BE}). 
The bosonic extensions $\hcalA$ can be understood as a vast generalization of the quantum Grothendieck rings $\calK_{\g;t} $ since it is known that  $\calK_{\g;t} $ are  isomorphic to $\hcalA$ of simply-laced finite types (\cite{FHOO, HL15}). 
For each $k\in \Z$, the subalgebra $ \hcalA[k]$  of $\hcalA$ generated by the generators $ f_{j,k}$ is isomorphic to the \emph{quantum unipotent  coordinate ring} $\calA_q(\sfn)$ associated with $C$. Thus the bosonic extension $\hcalA$ can be understood as an \emph{affinization} of  $\calA_q(\sfn)$. 

Let $ \ttB$ be the \emph{generalized braid group} (also called \emph{Artin--Tits group}) associated with $C$ and $\ttB^+$ its positive submonoid of $\ttB$. In the sequel, we simply call it the braid group. 
It was shown in \cite{JLO2, KKOP21B, KKOP24B} that there exist  the \emph{braid group actions} $\TT_j$ on $\hcalA$ which coincide with Lusztig's braid symmetries (\cite{LusztigBook, Lusztig96}) in each local pieces $\hcalA[k]$.
For any element $\ttb \in \ttB^+$,  the braid group actions $\TT_j$  lead us to the distinguished subalgebra $\hcalA (\ttb)$ of $\hcalA $  with the \emph{PBW theory} (\cite{KKOP24B,OP24}).
For each expression sequence $ \uii$ of  $\ttb$, the PBW root vectors are constructed by applying $\TT_j$ along $\uii$ and PBW monomials form a $\Zq$-linear basis of $\hcalA_\Zq(\ttb)$. 
Note that any arbitrary sequence $\uii$ can be understood as an expression of some element $\ttb \in \ttB^+$ since there is no quadratic defining relations in the braid group $\ttB$.

The \emph{global basis theory} for $\hcalA( \ttb )$ was established by the authors in \cite{KKOP24}. The global basis $\bfG$ of $\hcalA$ is a distinguished basis of the $\Zq$-lattice $\hcalA_\Zq$ of $\hcalA$. 
The global basis $\bfG$ has properties similar to  the \emph{upper global basis} (or \emph{dual canonical basis}) of $ \calA_q(\sfn)$ (see \cite{Kashiwarabook, LusztigBook} and references therein) and is parameterized by the \emph{extended crystal} $\hBi$ (\cite{KP22}). Thus the global basis of $\hcalA$ is denoted by  $\bfG = \{\rmG(\bfb) \ | \ \bfb\in\hBi\}$.  Note that the extended crystal $\hBi$ is an affinization of the infinite crystal $B(\infty)$. 
It was shown that $\bfG$ is invariant under the actions of $\TT_j$ and is compatible with the subalgebra $\hcalA (\ttb)$, i.e., the intersection $  \bfG (\ttb) := \bfG \cap \hcalA (\ttb)$ becomes a basis of the $\Q(q^{1/2})$-vector space $ \hcalA (\ttb)$ (\cite{KKOP24B}). In the case that $\hcalA \simeq \calK_{\g;t} $, the \emph{normalized} global basis $\tbfG$, which is the same as $\bfG$ up to multiples of $q^{1/2}$,  coincides with the set of the $(q,t)$-characters of simple modules in $\Cgz$ (\cite{KKOP24}), which tells us that the braid symmetries $\TT_j$ permute the set of the isomorphic classes of simple modules in $\Cgz$.

\smallskip

Meanwhile, the categorical \emph{PBW theory} for $\Cgz$ was developed by the authors (\cite{KKOP23P}) using  the \emph{quantum generalized  Schur-Weyl duality} (\cite{KKK18}). 
Let $ \sfg$ be the simply-laced finite type Lie algebra associated with the quantum affine algebra $U_q'(\g)$, and let $I$ and $ \sfI$ denote the index sets of
simple roots of $U_q'(\g)$ and $\sfg$, respectively (see Section \ref{subsec: Q-datum} for their precise definition). We denote by $\ttB = \langle  \bg_\im^{\pm1} \mid\im \in \sfI \rangle$ the braid group associated with the Lie algebra $ \sfg$. For  a \emph{complete duality datum} $\bbD =\{ L^\bbD_\im \}_{\im \in \sfI } \subset \Cgz$ and a
\emph{locally reduced {\rm (see Definition~\ref{def:seq})} sequence} 
$\uii=( \ldots, \im_{-1}, \im_{0}, \im_{1}, \ldots) $ of $\sfI$, the authors introduced the \emph{affine cuspidal modules} $C^{\bbD,\uii}_k$ and proved that 
there exist distinguished monoidal subcategories $\Cg^{[a,b], \bbD, \uii}$ for any intervals $[a,b]$ and that  the \emph{standard modules} (ordered tensor products of affine cuspidal modules) produce all simple modules   of $\Cg^{[a,b], \bbD, \uii}$   with the unitriangularity  property. Hernandez-Leclerc subcategories $\scrC_\ell$ and $\Cg^-$ appear as special cases of  the subcategories $\Cg^{[a,b], \bbD, \uii}$. 
It was conjectured in \cite{KKOP21B} that there exist  monoidal  exact autofunctors   $\mathcal{T}_\im$ ($ \im \in \sfI$) on the category $\Cgz$
which  categorify Lusztig's braid symmetries in each local piece $\scrD^k(\scrC_\bbD)$, where $\scrD$ denotes the right dual functor of $\Cgz$ and  $\scrC_\bbD$ is the subcategory of $\Cgz$ generated by $\bbD =\{ L^\bbD_\im \}_{\im \in \sfI } $. 
If the conjectural functors $\mathcal{T}_\im$ exist, then the affine cuspidal modules  $C^{\bbD,\uii}_k$ can be constructed by applying $\mathcal{T}_\im$ along the locally reduced sequence $\uii$.
Note that, in the case where $ \calK_{\g;t} \simeq \hcalA $, the quantum Grothendieck ring of $\scrD^k(\scrC_\bbD)$ is isomorphic to $\hcalA[k]$ and 
the braid group actions $\TT_\im$ on $\hcalA$ can be viewed as a ring-theoretic shadow of the conjectural functors $\mathcal{T}_\im$ on $\Cgz$.

For a complete duality datum $\bbD$ arising from a $\rmQ$-datum $\calQ $ (see Section \ref{subsec: Q-datum}) and a locally reduced sequence $\uii$, the category $\Cg^{[a,b], \bbD, \uii}$ provides a monoidal categorification of the Grothendieck ring $K(\Cg^{[a,b], \bbD, \uii})$ (see \cite{KKOP24A}).
The proof for the monoidal categorification is heavily based on the integer-valued invariants $\La$, $\de$, etc., arising from $R$-matrices (\cite{KKOP20}), which are a quantum affine counterpart of the same invariants in \emph{quiver Hecke algebras} (\cite{KKKO18}). 
The key ingredients for the monoidal categorification  are  \emph{affine determinantial modules} and \emph{$i$-boxes}. 
The affine determinantial modules $M^{\bbD,\uii}[a,b]$ are distinguished simple $U_q'(\g)$-modules determined by $\{ C^{\bbD,\uii}_k \}_{k\in \Z}$, which generalize Kirillov-Reshetikhin modules (see Section \ref{Sec: adm} for precise definition). The modules $M^{\bbD,\uii}[a,b]$ are  quantum affine analogues of the \emph{determinantial modules} (\cite{KKKO18}) over quiver Hecke algebras that categorify \emph{quantum unipotent minors}, and they have remarkable short exact sequences viewed as a vast  generalization of  \emph{T-systems} among Kirillov-Reshetikhin modules (\cite{Her06, Her10, Nakajima03II}).
These short exact sequences, which are also called $T$-systems, can be understood as the quantum affine counterpart of the quantum determinantial identities among quantum unipotent minors (\cite{GLS11,GLS13}) via generalized Schur-Weyl duality. 
The $i$-boxes are intervals that end with the same color, which provide a combinatorial skeleton for affine determinantial modules. 
An  \emph{admissible chain} $\frakC$ of $i$-boxes associated with a locally reduced sequence $\uii$ yields a monoidal seed of  $\Cg^{[a,b], \bbD, \uii}$, 
and certain combinatorial actions on $\frakC$, called \emph{box moves}, explain the mutations given by $T$-systems of affine determinantial modules.   
Thus the $i$-boxes allow us to give a monoidal seed for  $\Cg^{[a,b], \bbD, \uii}$ in a combinatorial viewpoint.

It would be natural and interesting to ask how the category $\Cg^{[a,b], \bbD, \uii}$ can be generalized to \emph{arbitrary} choices of $\bbD$ and $\uii$ without losing its categorical features. 
In the case for locally reduced sequences $\uii$,  the quantum Grothendieck ring of $\Cg^{[a,b], \bbD, \uii}$ is isomorphic to the subalgebra $\hcalA(\ttb)$ for some element $\ttb\in \ttB^+$.
This also leads us to the question: whether there exist the categories  $\scrC_\g^\bbD(\ttb)$ associated with \emph{arbitrary} elements $\ttb \in \ttB^+$ and whether they enjoy the same categorical properties such as the PBW theory and monoidal categorifications for $ \hcalA(\ttb)$. 

\smallskip 

In this paper, we answer these questions by
introducing  a distinguished subcategory $\scrC_\g^\bbD(\ttb)$ of the Hernandez-Leclerc category $\Cgz$ for an \emph{arbitrary} complete
duality datum $\bbD$ and an \emph{arbitrary} element $\ttb \in \ttB^+$. We then prove that the subcategory $\scrC_\g^\bbD(\ttb)$ provides a monoidal categorification of the algebra $\hcalA(\ttb)$, 
which significantly generalizes the earlier monoidal categorification given by the authors in \cite{KKOP24A}. The main results of the paper can be summarized as follows: let $U_q'(\g)$ be an \emph{arbitrary} quantum affine algebra of either untwisted or twisted type, and choose \emph{any} complete duality datum $\bbD$ and \emph{any} expression sequence $ \uii =(\im_1,\ldots,\im_r) $ of an element $\ttb \in \ttB^+$. 
\bnum
\item 
We introduce affine cuspidal modules $C^{\bbD,\uii}_k$ and affine determinantial modules $M^{\bbD,\uii}[a,b]$, and show that they enjoy the same categorical properties as those in the case for locally reduced expression sequences. 
Moreover the affine determinantial modules $M^{\bbD,\uii}[a,b]$ satisfy a $T$-system, in which the $i$-boxes play the same combinatorial role as in the locally reduced cases.

\item 
We introduce the subcategory $\scrC_\g^\bbD(\ttb)$ and develop its categorical PBW  theory. For each expression $\uii=(\im_1,\ldots,\im_r)$ of $\ttb$, we construct  standard modules as ordered tensor products of the affine cuspidal modules $C^{\bbD,\uii}_k$ ($k\in [1,r]$).
We then show that all simple modules can be obtained by taking the head of standard modules, which yield \emph{PBW data} parameterizing simple modules in $\scrC_\g^\bbD(\ttb)$. Moreover the unitriangularity between standard modules and simple modules holds, which generalizes the results in \cite{KKOP23P}.
The Grothendieck ring $K(\scrC_\g^\bbD(\ttb))$ of the subcategory $\scrC_\g^\bbD(\ttb)$ coincides with the commutative algebra $\obbA(\ttb)$ obtained by specializing $\hcalA(\ttb)$ at $q^{1/2}=1$. 

\item 
For each admissible chain $\frakC=(\frakc_k)_{1\le k\le r}$ of $i$-boxes associated with $\uii$, we construct the monoidal seed $\sfS^\bbD (\frakC) $  using affine determinantial modules and combinatorics of $i$-boxes. We then prove that the monoidal seed $\sfS^\bbD (\frakC)$ is completely $\Uplambda$-admissible.
It turns out that T-systems are mutations and all monoidal seeds arsing from admissible chains of $i$-boxes are connected by T-systems. We finally obtain that the category $\scrC_\g^\bbD(\ttb)$ gives a monoidal categorification of the cluster algebra $\obbA(\ttb) \simeq K(\scrC_\g^\bbD(\ttb))$ with the initial monoidal seed $\sfS^\bbD (\frakC)$.  
This implies that $\hcalA_\bfA(\ttb)$ has a quantum cluster algebra structure, and all cluster variables and monomials are contained in the normalized global basis $\tbfG(\ttb)$ of $\hcalA_\bfA(\ttb)$.
As a consequence, when $\g$ is of untwisted affine type and $\bbD$ arises from a $\rmQ$-datum $\calQ $,
the category $\scrC_\g^\bbD(\ttb)$ gives a monoidal categorification of the quantum Grothendieck ring $\calK_t(\scrC_\g^{\bbD}(\ttb)) $  and all cluster variables and monomials are the $(q,t)$-characters of simple modules.
\ee

\smallskip
One of key ingredients for the main results is the \emph{interplay} between the bosonic extension $\hcalA$ and the category $\Cgz$. Proposition \ref{prop: homomorphism from hA to Cg} says that, for any complete duality datum $\bbD$, there exists a unique $\Z$-algebra homomorphism 
$$
\bPhi_\bbD \col\hAz \Lto K(\Cgz),
$$ which is compatible with the Schur-Weyl duality functor $\calF_\bbD$ associated with $\bbD$ and the right dual functor $\rdual$. The specialization $\obbA$ of $\hAz$ at $q^{1/2}=1$ is a commutative algebra and the homomorphism $\bPhi_\bbD$ induces an isomorphism $\obPhi_\bbD \colon \obbA \isoto K(\Cgz)$ under the specialization at $q^{1/2}=1$ (see Theorem \ref{thm: obbA K(Cgz)}).
When $ \g$ is of untwisted affine type and the duality datum $\bbD_\calQ$ arises from a $\rmQ$-datum $\calQ $, there is an isomorphism
$\Psi_{\bbD_\calQ} \colon\hAz\isoto\calK_{\g;t}$  between the bosonic extension $\hAz$ and the quantum Grothendieck ring $\calK_{\g;t} $ of $\Cgz$ such that $\ev_{t=1} \circ \Psi_{\bbD_\calQ} = \bPhi_{\bbD_\calQ}$.  Under $\Psi_{\bbD_\calQ}$, the normalized global basis $\tbfG$ of $\hAz$ coincides with the $(q,t)$-characters of simple modules in $\calK_{\g;t}$ (\cite{KKOP24}). 
Hence, in the general case, i.e., $\g$ and $\bbD$ are \emph{arbitrary},
the algebra $ \hAz$ and the global basis $\bfG$ take over the roles of the quantum Grothendieck ring $\calK_{\g;t} $ and the $(q,t)$-characters of simple modules in $\Cgz$ and the homomorphism $\bPhi_\bbD$ generalizes the specialization of the $(q,t)$-characters of simple modules at $t^{1/2}=1$.  
From this perspective, we introduce the notions of \emph{$\bbD$-quantizable} and \emph{$\bbD$-categorifiable} associated with $\bPhi_\bbD$ (see Definition \ref{def: strong real}), which reflect the correspondence between the $(q,t)$-characters and the $q$-characters of simple modules under the specialization at $t^{1/2}=1$ (see Definition \ref{def: quantizable} and Lemma \ref{lem: chi quantizable}).

Under the homomorphism $\bPhi_\bbD$ together with the global basis $\bfG$, the braid group actions $\TT_\im$ on $\hcalA$ can be \emph{partially} lifted to the category $\Cgz$ for certain family of simple modules. This allows us to overcome the absence of the conjectural monoidal autofunctors  $\mathcal{T}_\im$ on  $\Cgz$ for our purpose. 
We investigate the braid group actions $\TT_\im$ with the actions $\scrS_\im $ on the set of strong duality data introduced in \cite[Section 5.3]{KKOP23P} (see Proposition \ref{prop: comm diagram} and Corollary \ref{cor: dia comm repeat}) and show that, for any simple module $M$ in $\scrD^n (\scrC_\bbD)$ and $\ttb \in \ttB$, there exists a simple module $\TT_\ttb (M)$ compatible with  $\bPhi_\bbD$ and $\bfG$, i.e., more precisely $\TT_\ttb (M)$ is \emph{$\bbD$-definable} (see Definition \ref{def: strong real} and Lemma \ref{lem:Treal}). 
This reveals the interplay between the global basis $\bfG$ and the set of simple modules under the homomorphism $\bPhi_\bbD$. 
We further study the head of tensor products and the integer-valued invariants $\La$ and $\de$ related to the braid group actions $\TT_\ttb$ on the simple modules in $\scrD^n(\scrC_\bbD)$, which provide the base for the categorical PBW theory and the monoidal categorification.

\medskip
The strategy of the proof of our main result is  
the reduction of the properties for an arbitrary sequence to those
for a locally reduced sequence.
For an arbitrary sequence $ \uii =(\im_1,\ldots,\im_r) $ of $\sfI$, we construct the affine cuspidal modules $C^{\bbD,\uii}_k$ by applying the braid group actions $\TT_\im$ along the sequence $\uii$ (see \eqref{eq: Ck}), and define the affine determinantial module $M^{\bbD,\uii}[a,b]$ by taking the head of the ordered tensor product of $C^{\bbD,\uii}_k$ along $\uii$ (see Definition \ref{def:affdet}). We then prove  that, if $\uii$ is obtained from another sequence
 $\ujj$ via a commutation move or a braid move, then affine cuspidal modules $C^{\bbD,\uii}_k$ and determinantial modules $M^{\bbD,\uii}[a,b]$ for $\uii$ have the same properties as those for $\ujj$. This yields that $C^{\bbD,\uii}_k$ and $M^{\bbD,\uii}[a,b]$ have the same categorical properties, including T-systems, as in the case of locally reduced sequence dealt in the previous work \cite{KKOP24A} by authors (see Theorem \ref{thm: every sequence good}).

\medskip
We give a closed formula for computing the $\La$-values between affine determinantial modules that commute with each other in terms of weights for $\sfg$ (Corollary \ref{cor: w-pairing}). This formula relates the $\La$-values to the exponents of $t$ between the $(q,t)$-characters of Kirillov-Reshetikhin modules computed in \cite{FHOO, FHOO2} (see Lemma \ref{lem: L up to}), which allows us to use the same formula for $\La$-matrices in the quantum torus (see Section \ref{subsec: Construction}).

Applying the same arguments given in \cite{KKOP23P}, we define the monoidal subcategory $\scrC_\g^{\bbD}(\ttb) \subset \Cgz$ categorifying $\obbA(\ttb)$ by using $C^{\bbD,\uii}_k$, where $\ttb = \bg_{\im_1} \cdots \bg_{\im_r} \in \ttB^+$
 and $\uii=(\im_1,\ldots,\im_r)$, 
and build the PBW theory for $\scrC_\g^{\bbD}(\ttb)$ (see Section \ref{subsec: subcat C(b)}). 
The PBW theory explains that the determinantial modules $M^{\bbD,\uii}[a,b]$ are contained in $\scrC_\g^{\bbD}(\ttb)$. Theorem \ref{thm: admissible seed} and Theorem \ref{thm: box and mutation} tell us that $M^{\bbD,\uii}[a,b]$ form a completely $\Uplambda$-admissible monoidal seed together with the combinatorics of $i$-boxes following the arguments developed in \cite{KKOP24A, KK24}. We finally prove that $\scrC_\g^\bbD(\ttb)$ gives a monoidal categorification of the cluster algebra $\obbA(\ttb) \simeq K(\scrC_\g^\bbD(\ttb))$ (see Theorem \ref{thm: mCat of cluster}) and $\hcalA_\bfA(\ttb)$ has a quantum cluster algebra structure (see Theorem \ref{thm: q cluster structure}). 
In the case where $\g$ is of untwisted affine type and $\bbD$ arises from a $\rmQ$-datum $\calQ $,
the category $\scrC_\g^\bbD(\ttb)$ gives a monoidal categorification of the quantum Grothendieck ring $\calK_t(\scrC_\g^{\bbD}(\ttb)) $ (see Theorem \ref{Thm: MC for Kt}).
 We remark that the quantum cluster algebra structure of $\hcalA_\bfA(\ttb)$ and its categorification are also studied by Qin in a different approach (\cite{Qin24A, Qin24B}).  
It would be interesting to ask how deeply $\hcalA(\ttb)$ and $\scrC_\g^\bbD(\ttb)$ are related to the cluster algebra structures arising from braid varieties (\cite{CGGLSS25,GLSbS22, GLSb23}). 

\smallskip

This paper is organized as follows. 
In Section \ref{Sec: preliminaries}, we briefly review the necessary backgrounds on quantum affine algebras and their representation theory.
In Section \ref{Sec: SW duality}, we recall the generalized Schur-Weyl duality and its related subjects.
In Section \ref{Sec: BE sc}, we review the bosonic extensions $\hcalA$ and investigate their key features including the notions of quantizability and categorifiability.
Section \ref{Sec: adm ac} and Section \ref{sec: generalization} are devoted to developing affine cuspidal and determinantial modules and their key properties including T-systems, and to building the PBW theory for $\scrC_\g^\bbD(\ttb)$. 
In Section \ref{Sec: quantizability}, we investigate the $\bbD$-quantizability with the quantum Grothendieck ring of $\Cgz$. Section \ref{Sec: QCA} explains the notion of quantum cluster algebras, and Section \ref{Sec: MS} and Section \ref{Sec: MC for qca} are devoted to proving that $\Cg(\ttb)$ provides a monoidal categorification.

\vskip 1em 

\textbf{Acknowledgments} The second, third and fourth authors gratefully acknowledge for the hospitality of RIMS (Kyoto University) during their visit in 2025.

\vskip 2em

\section{Preliminaries} \label{Sec: preliminaries}
In this section, we will briefly review basic stuff on the quantum affine algebras $U_q'(\g)$ and their representation theory.
Then we will recall the $\Z$-invariants related to $R$-matrices and root modules.
We refer~\cite{Kas02,KKOP20,KKOP21,KKOP22,KKOP23P,KKOP24A} for more details.

\subsection{Convention}
Throughout this paper, we use the following convention.
\bnum
\item For a statement $\ttP$, we set $\delta(\ttP)$ to be $1$ or $0$ depending on whether $\ttP$ is true or not. In particular, we set $\delta_{i,j}=\delta(i=j)$.
\item A ring is always unital.
  \item For a ring $A$, we denote by $A^\times$ the group of invertible elements.
  
\item For a totally ordered set $J = \{ \cdots < j_{-1} < j_0 < j_1 < j_2 < \cdots \}$, write
\eqn
&&\oprod_{j \in J} A_j \seteq \cdots A_{j_2}A_{j_1}A_{j_0}A_{j_{-1}}A_{j_{-2}} \cdots,\\
&&\rprod_{j \in J} A_j \seteq \cdots A_{j_{-2}}A_{j_{-1}}A_{j_0}A_{j_{1}}A_{j_{2}} \cdots,\\
&&\otens_{j \in J} A_j \seteq \cdots\tens A_{j_{2}}\tens A_{j_{1}}\tens A_{j_0}\tens
A_{j_{-1}}\tens A_{j_{-2}}\tens \cdots.
\eneqn
\item For $a,b \in \Z \sqcup \{ \pm\infty\}$, an \emph{interval} $[a,b]$ is the set of integers between $a$ and $b$:
$$
[a,b] \seteq \{ k \in \Z \ |  \ a \le k \le b\}.
$$
If $a>b$, we understand $[a,b]=\emptyset$.
\item For $k\in\Z$ let us denote by $\upsigma_k\in\Aut(\Z)$ the
transposition of $k$ and $k+1$.  
\item  For an interval $[a,b]$, we set
  $A^{[a,b]}$ to be the product of copies of a set $A$ indexed by $[a,b]$,  and for a monoid
  commutative $S$
$$
S^{\;\oplus [a,b]} \seteq \{ (c_a,\ldots,c_b) \ | \ c_k \in S\text{ and  $c_k=0$ except for finitely many $k$'s} \}. 
$$
\item For a vector space $V$ and an interval $[a,b]$
$$
V^{\tens [a,b]} \seteq V_{b} \otimes V_{b-1} \otimes \cdots \otimes V_a.
$$
where $V_k$ denotes the copy of $V$ for each $k \in \Z$.
\item For a set $S$, $|S|$ denotes the cardinality of $S$.
\item Let $\bsa=(a_j )_{ \ j \in J }$ be a family parameterized by an index set $J$. Then for any $j \in J$, we set $(\bsa)_j \seteq a_j$.
  \ee

\subsection{Quantum affine algebras} Let $q$ be an indeterminate. We take the algebraic closure of $\C(q)$ in $\bigcup_{m>0} \C\dpar{q^{1/m}}$ as a base field $\bfk$.
Let $(C,P,\Uppi ,P^\vee,\Uppi^\vee)$ be an \emph{affine Cartan datum} consisting of an \emph{affine Cartan matrix} $C=(C_{i,j})_{i,j \in I}$
with an index set $I$, a \emph{weight lattice} $P$, a set of \emph{simple roots} $\Uppi =\{\upal_i \}_{i \in I} \subset P$, 
a \emph{coweight lattice} $P^\vee \seteq \Hom_{\Z}(P,\Z)$ and a set of
\emph{simple coroots} $\{ h_i \}_{i \in I} \subset P^\vee$. The datum satisfies
$\Ang{h_i,\upal_j}=C_{i,j}$ for all $i,j\in I$, where $\Ang{ \ ,\ }\col P^\vee \times P \to \Z$
is the canonical pairing. We choose $\Iset{\UpLa_i}$ such that $\Ang{h_j,\UpLa_i}=\delta_{i,j}$ for $i,j \in I$ and call them
the \emph{fundamental weights}.

We take the \emph{imaginary root} $\updelta=\sum_{i\in I}u_i\upal_i$ and the \emph{central element} $c=\sum_{i \in I}c_i h_i$ such that
$ \{ \la \in \bigoplus_{i\in I} \Z \upal_i \ | \ \Ang{h_i,\la}=0 \text{ for every } i\in I \} =\Z\updelta$ and $ \{ h \in \bigoplus_{i\in I} \Z h_i \ | \ \Ang{h,\upal_i}=0 \text{ for every } i\in I \} =\Z c$.   We
choose $\uprho \in P$ (resp.\ $\uprho^\vee \in P^\vee$) such that $\lan
h_i,\uprho \ran=1$ (resp.\ $\lan \uprho^\vee,\upal_i\ran =1$) for all $i \in I$ and set $p^* \seteq (-1)^{\ang{\uprho^\vee,\updelta}}q^{\ang{c,\uprho}}$

Let us take a non-degenerate symmetric bilinear form $( \ , \ )$ on $P$ such that
$$
\ang{h_i,\la} = \dfrac{2(\upal_i,\la)}{(\upal_i,\upal_i)} \qtq (\updelta,\la)=\ang{c,\la}
\qt{ for any } \la \in P. 
$$
Note that $DC$ is symmetric for the diagonal matrix 
$D = {\rm diag}( d_i \seteq (\upal_i,\upal_i)/2 \ | \ i \in I )$. 
We set $q_i \seteq q^{d_i}$ and define
$$
[n]_{i} \seteq \dfrac{q_i^n-q_i^{-n} }{q_i-q_i^{-1}}, \quad   [n]_{i} ! \seteq \prod_{k=1}^n [k]_{i}
\qtq \begin{bmatrix} m \\ n     \end{bmatrix}_{q} \hspace{-1ex} \seteq
\dfrac{[m]_{i}!}{[n]_{i}!\,[m-n]_{i}!},
$$
for $i \in I$ and $m \ge n \in \Z_{\ge 0}$.

We denote by $\g$ and $U_q(\g)$ the \emph{affine Kac-Moody} algebra and the \emph{quantum group} associated with $(C,P,\Uppi,P^\vee,\Uppi^\vee)$, respectively.
Recall that $U_q(\g)$ is generated by Chevalley generators $e_i,f_i$ $(i \in I)$ and $q^h$ $(h \in P^\vee)$.

We will use the convention in~\cite[\S 2.1]{KKOP24A} to choose $0 \in I$ and set $I_0 \seteq I \setminus \{ 0 \}$. 
We define $\g_0$ to be the subalgebra of $\g$ generated by the Chevalley generators
$e_i,f_i$ and $h_i$ $(i \in I_0 )$. 
Throughout this paper,
we denote by $\Dynkin =(\Dynkin_0,\Dynkin_1)$ the Dynkin diagram of finite type $\g_0$
consisting of the set of vertices $\Dynkin_0$ and the set of edges $\Dynkin_1$
of $\Dynkin$, respectively (see Figure~\ref{Fig:unf} below for Dynkin diagrams of classical finite types). For indices $i,j \in \Dynkin_0=I_0$, we denote by $d(i,j)$ the \emph{distance} between $i$ and $j$ in $\Dynkin$.  

\smallskip 

We denote by $U_q'(\g)$ the subalgebra of $U_q(\g)$ generated by $e_i,f_i,t_i^{\pm 1}$
 $(i\in I)$, where $t_i = q_i^{h_i}$, and call it the \emph{quantum affine algebra} (see \cite[\S 2.1]{Kas02} for more details). 

\smallskip

Set $P_\cl \seteq P/\Z\updelta$ and call it \emph{the classical weight lattice}. Let
$\cl\col P \to P_\cl$ be the canonical projection. Then $P_\cl = \bigoplus_{i\in I} \cl(\UpLa_i)$. Set $P_\cl^0 \seteq \{ \la \in P_\cl \ | \  \ang{c,\la}=0 \} \subset P_\cl$.   

A $\Uqpg$-module $M$ is said to be \emph{integrable} if {\rm (a)} $M$ has a weight space
decomposition $M = \bigoplus_{\la \in P_\cl}M_\la$ where $M_\la \seteq \{ u \in M \ | \ t_iu=q_i^{\ang{h_i,\la}}u \text{ for all } i\in I \}$, and {\rm (b)} the actions of $e_i$ and $f_i$  on $M$ are locally nilpotent for any $i \in I$.  We denote by
$\scrC_\g$ the abelian monoidal category of finite-dimensional  
integrable modules over
$\Uqpg$. 

\smallskip

Let $z$ be an indeterminate.  For a $\Uqpg$-module $M$, let us denote by $M_z$ the module $\bfk[z^{\pm1}] \tens M$ with the action of $\Uqpg$ given by
$$
e_i(u_z) = z^{\delta_{i,0}}(e_iu)_z, \quad f_i(u_z) =
z^{-\delta_{i,0}}(f_iu)_z, \quad t_i(u_z) =(t_iu)_z.
$$
Here, for $u \in M$, we denote by $u_z$ the element $1 \tens u \in \bfk[z^{\pm 1}]\tens M$. 
For $x \in \bfk^\times$, we define $M_x \seteq M_z/(z-x)M_z$ and call $x$
a \emph{spectral parameter} of $M_x$. Note that $M_x \in \scrC_\g$ for $M \in \scrC_\g$.

\smallskip

For $i \in I_0$, we set
$$
\upvarpi_i \seteq \gcd(c_0,c_i)^{-1}\cl(c_0\UpLa_i-c_i\UpLa_0) \in P^0_\cl.
$$
Then there exists a unique simple module $V(\upvarpi_i)$ in $\Cg$, called the \emph{$i$-th fundamental representation} of weight $\upvarpi_i$ satisfying certain properties (see, \cite[\S 5.2]{Kas02}). We also call $V(\upvarpi_i)_a$ $(a \in\bfk^\times)$ a fundamental representation. 

\smallskip

For simple modules $M$ and $N$ in $\Cg$, we say that $M$ and $N$ \emph{commute}
if $M \tens N \simeq N \tens M$. We also say that they \emph{strongly commute}
if $M \tens N$ is simple. Note that $M$ and $N$ commutes as soon as they strongly commute. We say that a simple module $L$ is \emph{real} if $L$ strongly commutes with itself.
We say that a simple module $L$ is \emph{prime} if there exist no non-trivial modules
$M_1$ and $M_2$ such that $L \simeq M_1 \tens M_2$. 

\smallskip

Note that the category $\Cg$ is \emph{rigid}; i.e., every module $M$ has a right dual $\scrD M$ and a left dual $\ldual M$. 
Thus we have the evaluation morphisms
$$
M \otimes \rdual M \to \bone, \quad  \ldual M \otimes  M \to \bone,
$$
and the co-evaluation morphisms
$$
\bone \to \rdual M \tens M, \quad  \bone \to  M \tens \ldual M.
$$
Here $\mathbf{1}$ denotes the trivial representation.

\subsection{$R$-matrices and $\Z$-invariants} For modules $M$ and $N \in \Cg$, there exists $\bfk\dpar{z} \tens \Uqpg$-module isomorphism
$$
\Runiv_{M,N_z} \col \bfk\dpar{z} \tens_{\bfk[z^{\pm 1}]} (M \tens N_z) \to \bfk\dpar{z} \tens_{\bfk[z^{\pm 1}]} (N_z \tens M)
$$
satisfying certain properties (see \cite{Kas02} for more details). We call $\Runiv_{M,N_z}$ the \emph{universal $R$-matrix} of $M$ and $N$. 

For modules  $M$ and $N \in \Cg$, we say that $\Runiv_{M,N_z}$ is \emph{rationally renormalizable} if there exists $c_{M,N}(z) \in \bfk\dpar{z}^\times$
such that 
\bnum
\item $\Rren_{M,N_z} \seteq c_{M,N}(z)\Runiv_{M,N_z}\col M \otimes N_z \to N_z \otimes M$ and 
\item $\Rren_{M,N_z}|_{z =x}$ does not vanish for any $x \in \bfk^\times$. 
  \ee
The function $c_{M,N}(z)$ is unique up to a multiple of $\bfk[z^{\pm1}]^\times$.  

In this case, 
we write $\Rr_{M,N}\seteq \Rren_{M,N_z}|_{z =1}$
and call it the \emph{$R$-matrix}. 
Note that $\Runiv_{M,N_z}$ is  
rationally renormalizable for simple modules $M,N \in \Cg$.  

We set $\tp \seteq p^{*2} = q^{2\ang{c,\uprho}}$ and
$$
\vph(z) \seteq \prod_{s \in \Z_{\ge 0}} (1-\tp^s z) = \sum_{n=0}^\infty \dfrac{(-1)^n \tp^{n(n-1)/2}}{\prod_{k=1}^n (1-\tp^k)}z^n \in \bfk\dbracket{z}.
$$

We define the multiplicative subgroup $\calG$ in $\bfk\dpar{z}^\times$ containing $\bfk(z)^\times$ as follows:
\begin{align*}
\calG \seteq \left\{ 
cz^m \prod_{a \in \bfk^\times} \vph(az)^{\eta_a} \left| 
\begin{matrix}
c \in \bfk^\times, \ m \in \Z \\
\ \eta_a \in \Z \text{ vanishes except finitely many $a$'s}
\end{matrix}
\right.
\right\}. 
\end{align*}
Then it is proved in~\cite{KKOP20} that $c_{M,N}(z)$ is contained in $\calG$ for any rationally renormalizable $\Runiv_{M,N_z}$.

In~\cite[Section 3]{KKOP20}, the following group homomorphisms
are introduced 
$$
\Deg\col \calG \to \Z \qtq \Deg^\infty\col \calG \to \Z
$$
defined by 
$$
\Deg(f(z)) \seteq \sum_{a \in \tp^{\Z_{\le 0}}} \eta_a - \sum_{a \in \tp^{\Z_{> 0}}} \eta_a \qtq 
\Deg^\infty(f(z)) \seteq \sum_{a \in \tp^{\Z} } \eta_a
$$
for $f(z) = cz_m \prod_{a\in\bfk^\times} \vph(az)^{\eta_a} \in \calG$. Here  $\tp^S \seteq \{ \tp^k \ | \ k\in S\}$ for a subset $S$ of $\Z$.  

\begin{definition}[{\cite[Definition 3.6, 3.14]{KKOP20}}] Let $M,N \in \Cg$. 
\ben
\item If $\Runiv_{M,N_z}$ is rationally renormalizable, we define the integers $\La(M,N)$ and $\La^\infty(M,N)$ by
$$
\La(M,N)=\Deg(c_{M,N}(z)) \qtq \Li(M,N)=\Deg^\infty(c_{M,N}(z)). 
$$
\item For simple modules $M$ and $N$ in $\Cg$, we define the integer $\de(M,N)$ by 
$$
\de(M,N) = \dfrac{1}{2}\big(\La(M,N)+\La(\scrD^{-1}M,N)\big). 
$$
\ee
\end{definition}

\begin{proposition}[{\cite{KKOP20,KKOP23P}}]\label{prop: de property 1}
Let $M$ and $N$ be simple modules in $\Cg$.
\bnum
\item We have $\de(M,N) \in \Z_{\ge0}$ and $ \de(M,N)=\frac{1}{2}\big(\La(M,N)+\La(N,M)\big)=\de(N,M)$.
\item  Assume that one of $M$ and $N$ is real. Then $M$ and $N$ strongly commute if and only if $\de(M,N)=0$. 
\item \label{it: La de} $\La(M,N) = \sum_{k\in \Z}(-1)^{k+\delta(k<0)}\de(M,\rdual^kN)$ and $\Li(M,N) = \sum_{k\in \Z}(-1)^{k}\de(M,\rdual^kN)$.
\item $\La(M,N)=\La(\ldual N,M)=\La(N,\rdual M)$. 
\ee 
\end{proposition}

\begin{lemma} [{\cite[Corollary 3.18]{KKOP21}}] \label{lem: decrease}
Let $L$ be a real simple module and let $M$ be a module in $\Cg$. 
Let $n \in \Z_{\ge0}$ and assume that any simple subquotient $S$ of $M$ satisfies $\de(L,S)\le n$. 
Then any simple subquotient $K$ of $L \tens M$ satisfies $\de(L,K)<n$. In particular, any simple subquotient of $L^{\tens n}\tens M$ 
strongly commutes with $L$. 
\end{lemma}

For simple modules $M$ and $N$ in $\Cg$, $M\hconv N$ and $M \sconv N$ denote the head and the socle of $M \tens N$, respectively.

\begin{proposition}[{\cite{KKKO18,KKOP20,KKOP23P,KKOP24A}}] \label{prop: de properties}
  Let $M$ and $N$ be simple modules in $\Cg$ such that one of them is real.
  Then, we have
\bna
\item $\Hom(M \tens N,N\tens M)=\bfk \; \Rr_{M,N}$,
\item $M \hconv N$ and $N \hconv M$ are simple modules in $\Cg$. Moreover $M \hconv N \simeq {\rm Im}(\Rr_{M,N}) \simeq N \sconv M$, 
\item 
  $M \hconv N$, as well as $M \sconv N$,
  appears once in the composition series of $M \tens N$.
\item\label{it: head=socle}
  $M \tens N$ is simple if and only if $M \hconv N \simeq M \sconv N$,
  
\item \label{it: d=1 length2} If $\de(M,N)=1$, we have an exact sequence 
$$
0 \to  M \sconv N \to M \tens N \to M \hconv N \to 0.
$$
\ee

\snoi
Assume further that $M$ and $N$ are real. Then, we have
\bna
\item[{\rm (f)} ] If $\de(M,N) \le 1$, $M \hconv N$ is real,
\item[{\rm (g) }] If $M$ commutes with $M \hconv N$, then $M \hconv N$ is real. 
\ee 
\end{proposition}

The following lemma is a dual version of \cite[Lemma 2.24]{KKOP24A}.

\begin{lemma} \label{lem: Lj Mj}
Let $L_j$ and  $M_j$ be real simple modules $(j=1,2)$.
Assume that
\bnum\item
$M_j\hconv L_j$ commutes with $L_k$ for $j,k=1,2$,
\item $L_1$ and $L_2$ commute.
\ee
Then we have the followings:
\bna
\item  $M_j\hconv L_j$ is real for $j=1,2$.
\item  If $\de(\ldual L_j,M_1)=0$ for $j=1,2$, then
$$(M_1\hconv M_2) \hconv (L_1\tens L_2) \simeq   M_1 \hconv \big(M_2 \hconv (L_1\tens L_2)\big)
\simeq (M_1\hconv L_1)\hconv (M_2\hconv L_2).$$
\item 
Assume that $\de(\ldual L_j,M_k)=0$ for $j,k=1,2$.
Then $M_1$ and $M_2$ commute if and only if $M_1\hconv L_1$ and $M_2\hconv L_2$ commute.
\ee
\end{lemma}

\begin{lemma}[{\cite[Corollary 3.13]{KKKO15}}] \label{lem: recover}
Let $L$ be a real simple module and $X$ a simple module. 
\begin{align*}
(L \hconv X) \hconv \rdual L \simeq X, \quad     \ldual L \hconv (X \hconv L) \simeq X, \\
L \hconv (X \hconv \rdual L) \simeq X, \quad     (\ldual L \hconv X) \hconv L \simeq X. 
\end{align*}    
\end{lemma}

\begin{lemma} \label{lem: LX LY}
  Let $X,Y$ be simple module such that one of them is real and let
  $L$ be a real simple module. We assume that
one of $L \hconv X$ and $L \hconv Y$ is real, $\de(X,Y)=0$, $\de(L,L\hconv X)=0$ and $\de(L,L\hconv Y)=0$. Then we have
$$\de(L \hconv X,L\hconv Y)=0.$$
\end{lemma}

\begin{proof}
By the assumption, $L^{\otimes2} \tens X \tens Y \simeq L^{\otimes2} \tens Y \tens X$ and $L^{\otimes2} \tens X \tens Y$ have simple heads. 
On the other hand we have the following surjections
$$
(L \hconv X) \hconv (L \hconv Y ) \twoheadleftarrow L^{\otimes2} \tens X \tens Y \simeq L^{\otimes2} \tens Y \tens X \twoheadrightarrow (L \hconv Y) \hconv (L \hconv X ).
$$
Hence $(L \hconv X) \hconv (L \hconv Y ) \simeq (L \hconv Y) \hconv (L \hconv X)$. Then the assertion follows from Proposition~\ref{prop: de properties}~\eqref{it: head=socle}. 
\end{proof}

\begin{definition}
A sequence $\uL=(L_1,\ldots,L_r)$ of simple modules is called a \emph{normal sequence} if the composition of $R$-matrices 
\begin{align*}
    \Rr_{L_1,\ldots,L_r} & \seteq
\displaystyle\prod_{1\le i <k \le r} \Rr_{L_i,L_k}
=(\Rr_{L_{r-1},L_r})  \circ \cdots \circ (\Rr_{L_{2},L_r}\circ
\cdots \circ \Rr_{L_{2},L_{3}})  \circ (\Rr_{L_1,L_r} \circ \cdots
\circ \Rr_{L_1,L_{2}}) \\
& : L_1 \tens \cdots\tens L_r \to L_r \tens \cdots \tens L_1 \text{ does not vanish}. 
\end{align*}
\end{definition}

An ordered sequence of simple modules $\uL = (L_1,L_{2}\ldots,L_r)$ in $\Cg$ is called \emph{almost real}, if all $L_i$ $(1 \le i \le r)$ are real except for at most one. 

\begin{lemma} [{\cite{KK19,KKOP23L}}]\label{lem: normal property} 
Let $\uL=(L_1,\ldots,L_r)$ be an almost real sequence. If $\uL$ is normal, then the image of $\Rr_{\uL}$ is simple and coincides with the head of $L_1 \tens \cdots \tens L_r$    
and also with the socle of $L_r \tens \cdots \tens L_1$.
Moreover, the following conditions are equivalent. 
\bna
\item
  $\uL$ is normal,
  \item 
  $\uL'=(L_{2},\ldots,L_r)$ is a normal sequence and $\La\big(L_1,{\rm Im}(\Rr_{\uL'})\big) = \sum_{k=2}^{r} \La(L_1,L_k)$.
\item  
  $\uL''=(L_1,\ldots,L_{r-1})$ is a normal sequence and $\La\big({\rm Im}(\Rr_{\uL''}),L_r\big) = \sum_{k=1}^{r-1} \La(L_k,L_r)$.
\ee
\end{lemma}

\Prop\label{prop:normalsimple}
Let $\uL=(L_1,\ldots,L_r)$ be an almost real normal sequence.
\bnum
\item
  Any simple subquotient $S$ of \/
  $L_{2}\tens\cdots\tens L_r$ satisfies
  $\La(L_1,S)\le\sum_{k=2}^r\La(L_1,L_k)$.
  \item
Any simple subquotient $S$ of \/
  $L_{1}\tens\cdots\tens L_{r-1}$
  satisfies $\La(S,L_r)\le\sum_{k=1}^{r-1}\La(L_k,L_r)$
\item $\hd(L_{1}\tens\cdots\tens L_r)$ appears only once in the composition series of 
  $L_{1}\tens\cdots\tens L_r$.\label{itsimple}
  \ee
  \enprop
  \Proof
  (i) and (ii) are known in \cite[Corollary 4.2]{KKOP20}. 
  Let us prove (iii).
  We shall argue by induction on $r$.
  Either $L_1$ or $L_r$ is real.
  Since the other case can be proved similarly,
  we assume that $L_1$ is real.
  Set $K=L_1\tens\cdots\tens L_r$,
  $K'=L_2\tens\cdots\tens L_r$,
  and   $L=\hd(L_1\tens\cdots\tens L_r)$,
  $L'=\hd(L_2\tens\cdots\tens L_r)$,

  \mnoi
(1)\  First let us show that $L$ does not appear in the composition series of 
$L_1\tens \Ker(K'\to L')$.
If it appears, then there exists a simple subquotient
$S$ of $\Ker(K'\to L')$ such that $L$ appears as a simple subquotient of
$L_1\tens S$. 
Hence we have
$$\La(L_1, L)\le \La(L_1, S)\le \La(L_1,K')\le \La(L_1, L).$$
Thus we have 
$$\La(L_1, L)=\La(L_1,S) = \La(L_1,L_1 \hconv S)$$ 
and hence
$L\simeq L_1\hconv S$ by \cite[Theorem 4.11]{KKOP20}. Here the second equality holds by~\cite[Corollary 3.20, Lemma 4.3]{KKOP20}.  Since $L\simeq L_1\hconv L'$, we have
$L'\simeq S$ by Lemma~\ref{lem: recover}. 
By the induction hypothesis, $L'$ cannot appear as a simple subquotient of 
$\Ker(K'\to L')$. It is a contradiction.

\mnoi
(2)\ Since $L\simeq L_1\hconv L'$ appears only once in the composition series of $L_1\tens L'$,
we are done.
  \QED

\begin{lemma} [{\cite[Lemma 4.3 and 4,17]{KKOP20} and \cite[Lemma 2.24]{KKOP23P}}] \label{lem: normal seq d}
Let $L,M,N$ be simple modules in $\Cg$ that are all real except for at most one.
\bna
\item \label{it: normal 1}
Assume that one of the following conditions holds:
\bnum
\item $\de(L,M)=0$ and $L$ is real,
\item $\de(M,N)=0$ and $N$ is real,
\item $\de(L,\ldual N)=\de(\rdual L,N)=0$ and $L$ or $N$ is real,
\ee
then $(L,M,N)$ is a normal sequence.
\item Assume that $L$ is real. 
\bnum
\item $(L,M,N)$ is normal if and only if $(M,N,\rdual L)$ is normal.
\item $\de(L,M\hconv N)=\de(L,M)+\de(L,N)$ if and only if $(L,M,N)$ and $(M,N,L)$ are normal. 
\ee
\ee
\end{lemma}

\begin{lemma} [{\cite[Corollary 2.25]{KKOP23P}}]\label{lem: de=de}
Let $L,M$ be real simple modules and $X$ a simple module.
\bnum
\item \label{it: de=de 1}
If $\de(L,M)=\de(\rdual L,M)=0$, then we have $\de(L,X\hconv M)=\de(L,X)$.
\item  \label{it: de=de 2} If $\de(L,M)=\de(\ldual L,M)=0$, then we have $\de(L,M\hconv X)=\de(L,X)$.
\ee
\end{lemma}

\begin{definition} [{\cite{KKOP23P,KKOP24A}}] Let $(M,N)$ be an ordered pair of simple modules in $\Cg$. 
\ben
\item We call the pair \emph{unmixed} if 
$$ \de(\rdual M,N)=0$$
and \emph{strongly unmixed} if 
$$ \de(\rdual^k M,N)=0 \qt{for any } k \in \Z_{>0}.$$
\item An almost real sequence $\uM=(M_1,\ldots,M_r)$ is said to be \emph{$($strongly$)$ unmixed} if $(M_i,M_k)$ is (strongly) unmixed for all
  $1\le i <k \le r$.  
\ee 
\end{definition}

\begin{proposition} [{\cite{KKOP23P}}] \label{prop: Unmix normal}
  \hfill
\bnum
\item \label{it: Li La} For a strongly unmixed pair $(M,N)$ of simple modules,
  we have
$$ \Li(M,N) = \La(M,N).$$ 
\item \label{it: unmix normal} Any unmixed almost real
  sequence $\uM=(M_1,\ldots,M_r)$ is normal. 
\item For a strongly unmixed almost real sequence $\uM=(M_1,\ldots,M_r)$, 
the pair $$
\big( \head(M_1\tens\cdots\tens M_j), \head(M_{k}\tens\cdots\tens M_r) \big)
$$
is strongly unmixed for any $1< j<k \le r$. 
\ee
\end{proposition}
 
\begin{lemma}[{\cite[Lemma 6.11]{KKOP23P}}] \label{Lem: LMLN MN}
Let $L,M,N$ be simple modules in $\Cg$ and assume that $L$ is real.
\bnum
\item \label{it: left}
If $(L,M)$ and $(L,N)$ are strongly unmixed and $L \hconv N$ appears in $L \tens M$ as a subquotient, then we have $M \simeq N$.
\item \label{it: right}
If $(M,L)$ and $(N,L)$ are strongly unmixed and $N \hconv L$ appears in $M \tens L$ as a subquotient, then we have $M \simeq N$.
\ee
\end{lemma}

\subsection{Root modules} We say that a real simple module $L$ is  
a \emph{root module} if
\begin{align} \label{eq: root module}
\de(L,\rdual^{k}(L)) =\delta(k = \pm 1) \qt{ for any } k \in \Z.     
\end{align}

\begin{lemma} [{\cite[Lemma 3.4]{KKOP23P}}] \label{lem: -1 by root}
Let $L$ be a root module and let $X$ be a simple module such that 
$\de(L,X) >0$.  Then we have
\bnum
\item $\de(L,L\hconv X) = \de(L,X)-1$ and $\de(\ldual L,L\hconv X) = \de(\ldual L,X)$,
\item $\de(L,X\hconv L) = \de(L,X)-1$ and $\de(\rdual L,X\hconv L) = \de(\rdual L,X)$.
\ee
Thus we have
\begin{align} \label{eq: de L,X}
  \de(L,L^{\otimes n} \hconv Y) =  \de(L, Y \hconv L^{\otimes n}) = \max(\de(L,Y)-n,0)\end{align}
for any simple module $Y$ and $n\in\Z_{\ge0}$.
\end{lemma}

\begin{lemma}[{\cite[Lemma 3.8 and 3.9]{KKOP23P}}] \label{lem: roots and dual}
Let $L$ and $L'$ be root modules satisfying 
$$
\de(\rdual^k L,L') =\delta(k =0 ) \qt{ for } k\in \Z.
$$Then, we have
\bnum
\item $L \hconv L'$ is a root module,
\item $\de(\rdual^k L,L\hconv L')=\delta(k=1)$ and $\de(\rdual^k L,L'\hconv L)=\delta(k=-1)$.
\ee
\end{lemma}

\begin{proposition}[{\cite[Proposition 2.28]{KKOP24A}}]
Every fundamental representation is a root module.     
\end{proposition}

\section{Schur-Weyl dualities and their related subjects} \label{Sec: SW duality}
In this subsection, we recall the generalized Schur-Weyl duality functors, constructed in~\cite{KKK18},
and its related subjects including categorification of quantum unipotent coordinate rings by following~\cite{KKOP24A}.

\subsection{$\rmQ$-data} \label{subsec: Q-datum}
For each untwisted quantum affine algebra $U_q'(\g)$, we assign the finite simple Lie algebra $\sfg$ of symmetric type as follows:
\renewcommand{\arraystretch}{1.5}
\begin{align} \label{Table: root system}
\small
\begin{array}{|c||c|c|c|c|c|c|c|}
\hline
 \g  & A_n^{(1)} \ (n\ge1)  & B_n^{(1)} \ (n\ge2) & C_n^{(1)} \ (n\ge3)  & D_n^{(1)} \ (n\ge4) & E_{6,7,8}^{(1)} & F_{4}^{(1)} & G_{2}^{(1)}  \\ \hline
 \g_0  & A_n & B_n & C_n  & D_n & E_{6,7,8} & F_{4} & G_{2} \\ \hline
\sfg  & A_n & A_{2n-1}    & D_{n+1}   &  D_n & E_{6,\,7,\,8} & E_{6} & D_{4}  \\
\hline
\ord(\sigma)&1&2&2&1&1&2&3\\
\hline
\end{array}
\end{align}
Note that $\g_0 \ne \sfg$ when $\g_0$ is  not simply laced. 
Let $\sfI$ be the index set of simple roots  $\{\al_\im\}_{\im\in \sfI}$
of $\sfg$.
We denote by $\Phi^+_\sfg$ the set of positive roots of $\sfg$, by $\sfQ^\pm$ the positive (resp. negative) root lattice of $\sfg$ and
by $\sfP$ the weight lattice of $\sfg$. For any $\be = \sum_{\im \in \sfI} a_\im\al_\im \in \rl$, we set
$$
\het(\be) = \sum_{\im \in\sfI} |a_\im| \in \Z_{\ge 0}. 
$$

\smallskip

The Weyl group $\weyl$ of $\sfg$ is
generated by simple reflections $\{ s_\im \}_{\im \in \sfI}$ subject to
$$\text{(i) $s_\im^2= 1$ $(\im\in \sfI)$, (ii) $s_\im s_\jm = s_\jm s_\im$  if $d(\im,\jm)>1$,
and (iii) $s_\im s_\jm s_\im= s_\jm s_\im s_\jm$  if $d(\im,\jm)=1$.}$$
We call (ii) the \emph{commutation relations}, and 
(iii) the \emph{braid relations}.
We denote by $w_0$ the longest element of $\weyl$. Note that $w_0$ induces an involution $^*$ on $\sfI$
defined by $w_0(\al_\im)=-\al_{\im^*}$. 

\begin{remark}
We remark here that the finite simple Lie algebra  $\sfg$ corresponding to $\g$ in~\eqref{Table: root system}
can be understood as an \emph{unfolding} of $\g_0$ in the following sense: The Dynkin diagram $\Dynkin_{\g_0}$
of $\g_0$ can be obtained by folding the one of $\Dynkin_{\sfg}$ via a Dynkin diagram folding $\sigma={\rm id}$, $\vee$ or $\widetilde{\vee}$ on $\Dynkin_{\sfg}$ 
(see Figure~\ref{Fig:unf}).     
\end{remark}

\begin{figure}[ht]
\begin{center}
\begin{tikzpicture}[xscale=1.25,yscale=.7]
\node (A2n1) at (-0.2,4.5) {$(\mathrm{A}_{2n-1}, \vee)$};
\node[dynkdot,label={below:\footnotesize$n+1$}] (A6) at (4,4) {};
\node[dynkdot,label={below:\footnotesize$n+2$}] (A7) at (3,4) {};
\node[dynkdot,label={below:\footnotesize$2n-2$}] (A8) at (2,4) {};
\node[dynkdot,label={below:\footnotesize$2n-1$}] (A9) at (1,4) {};
\node[dynkdot,label={above:\footnotesize$n-1$}] (A4) at (4,5) {};
\node[dynkdot,label={above:\footnotesize$n-2$}] (A3) at (3,5) {};
\node (Au) at (2.5, 5) {$\cdots$};
\node (Al) at (2.5, 4) {$\cdots$};
\node[dynkdot,label={above:\footnotesize$2$}] (A2) at (2,5) {};
\node[dynkdot,label={above:\footnotesize$1$}] (A1) at (1,5) {};
\node[dynkdot,label={above:\footnotesize$n$}] (A5) at (5,4.5) {};
\path[-]
 (A1) edge (A2)
 (A3) edge (A4)
 (A4) edge (A5)
 (A5) edge (A6)
 (A6) edge (A7)
 (A8) edge (A9);
\path[-] (A2) edge (Au) (Au) edge (A3) (A7) edge (Al) (Al) edge (A8);
\path[<->,thick,blue] (A1) edge (A9) (A2) edge (A8) (A3) edge (A7) (A4) edge (A6);
\path[->, thick, blue] (A5) edge [loop below] (A5);
\def\Foffset{6.5}
\node (Bn) at (-0.2,\Foffset) {$\mathrm{B}_n$};
\foreach \x in {1,2}
{\node[dynkdot,label={above:\footnotesize$\x$}] (B\x) at (\x,\Foffset) {};}
\node[dynkdot,label={above:\footnotesize$n-2$}] (B3) at (3,\Foffset) {};
\node[dynkdot,label={above:\footnotesize$n-1$}] (B4) at (4,\Foffset) {};
\node[dynkdot,label={above:\footnotesize$n$}] (B5) at (5,\Foffset) {};
\node (Bm) at (2.5,\Foffset) {$\cdots$};
\path[-] (B1) edge (B2) (B2) edge (Bm) (Bm) edge (B3) (B3) edge (B4);
\draw[-] (B4.30) -- (B5.150);
\draw[-] (B4.330) -- (B5.210);
\draw[-] (4.55,\Foffset) -- (4.45,\Foffset+.2);
\draw[-] (4.55,\Foffset) -- (4.45,\Foffset-.2);
\draw[-,dotted] (A1) -- (B1);
\draw[-,dotted] (A2) -- (B2);
\draw[-,dotted] (A3) -- (B3);
\draw[-,dotted] (A4) -- (B4);
\draw[-,dotted] (A5) -- (B5);
\draw[|->] (Bn) -- (A2n1);
\node (Dn1) at (-0.2,0) {$(\mathrm{D}_{n+1}, \vee)$};
\node[dynkdot,label={above:\footnotesize$1$}] (D1) at (1,0){};
\node[dynkdot,label={above:\footnotesize$2$}] (D2) at (2,0) {};
\node (Dm) at (2.5,0) {$\cdots$};
\node[dynkdot,label={above:\footnotesize$n-2$}] (D3) at (3,0) {};
\node[dynkdot,label={above:\footnotesize$n-1$}] (D4) at (4,0) {};
\node[dynkdot,label={above:\footnotesize$n$}] (D6) at (5,.5) {};
\node[dynkdot,label={below:\footnotesize$n+1$}] (D5) at (5,-.5) {};
\path[-] (D1) edge (D2)
  (D2) edge (Dm)
  (Dm) edge (D3)
  (D3) edge (D4)
  (D4) edge (D5)
  (D4) edge (D6);
\path[<->,thick,blue] (D6) edge (D5);
\path[->,thick,blue] (D1) edge [loop below] (D1)
(D2) edge [loop below] (D2)
(D3) edge [loop below] (D3)
(D4) edge [loop below] (D4);
\def\Coffset{1.8}
\node (Cn) at (-0.2,\Coffset) {$\mathrm{C}_n$};
\foreach \x in {1,2}
{\node[dynkdot,label={above:\footnotesize$\x$}] (C\x) at (\x,\Coffset) {};}
\node (Cm) at (2.5, \Coffset) {$\cdots$};
\node[dynkdot,label={above:\footnotesize$n-2$}] (C3) at (3,\Coffset) {};
\node[dynkdot,label={above:\footnotesize$n-1$}] (C4) at (4,\Coffset) {};
\node[dynkdot,label={above:\footnotesize$n$}] (C5) at (5,\Coffset) {};
\draw[-] (C1) -- (C2);
\draw[-] (C2) -- (Cm);
\draw[-] (Cm) -- (C3);
\draw[-] (C3) -- (C4);
\draw[-] (C4.30) -- (C5.150);
\draw[-] (C4.330) -- (C5.210);
\draw[-] (4.55,\Coffset+.2) -- (4.45,\Coffset) -- (4.55,\Coffset-.2);
\draw[-,dotted] (C1) -- (D1);
\draw[-,dotted] (C2) -- (D2);
\draw[-,dotted] (C3) -- (D3);
\draw[-,dotted] (C4) -- (D4);
\draw[-,dotted] (C5) -- (D6);
\draw[|->] (Cn) -- (Dn1);
\node (E6desc) at (6.8,4.5) {$(\mathrm{E}_6, \vee)$};
\node[dynkdot,label={above:\footnotesize$2$}] (E2) at (10.8,4.5) {};
\node[dynkdot,label={above:\footnotesize$4$}] (E4) at (9.8,4.5) {};
\node[dynkdot,label={above:\footnotesize$5$}] (E5) at (8.8,5) {};
\node[dynkdot,label={above:\footnotesize$6$}] (E6) at (7.8,5) {};
\node[dynkdot,label={below:\footnotesize$3$}] (E3) at (8.8,4) {};
\node[dynkdot,label={below:\footnotesize$1$}] (E1) at (7.8,4) {};
\path[-]
 (E2) edge (E4)
 (E4) edge (E5)
 (E4) edge (E3)
 (E5) edge (E6)
 (E3) edge (E1);
\path[<->,thick,blue] (E3) edge (E5) (E1) edge (E6);
\path[->, thick, blue] (E4) edge[loop below] (E4); 
\path[->, thick, blue] (E2) edge[loop below] (E2); 
\def\Foffset{6.5}
\node (F4desc) at (6.8,\Foffset) {$\mathrm{F}_4$};
\foreach \x in {1,2,3,4}
{\node[dynkdot,label={above:\footnotesize$\x$}] (F\x) at (\x+6.8,\Foffset) {};}
\draw[-] (F1.east) -- (F2.west);
\draw[-] (F3) -- (F4);
\draw[-] (F2.30) -- (F3.150);
\draw[-] (F2.330) -- (F3.210);
\draw[-] (9.35,\Foffset) -- (9.25,\Foffset+.2);
\draw[-] (9.35,\Foffset) -- (9.25,\Foffset-.2);
\draw[|->] (F4desc) -- (E6desc);
\path[-, dotted] (F1) edge (E6)
(F2) edge (E5) (F3) edge (E4) (F4) edge (E2);

\node (D4desc) at (6.8,0) {$(\mathrm{D}_{4}, \widetilde{\vee})$};
\node[dynkdot,label={above:\footnotesize$1$}] (D1) at (7.8,.6){};
\node[dynkdot,label={above:\footnotesize$2$}] (D2) at (8.8,0) {};
\node[dynkdot,label={left:\footnotesize$3$}] (D3) at (7.8,0) {};
\node[dynkdot,label={below:\footnotesize$4$}] (D4) at (7.8,-.6) {};
\draw[-] (D1) -- (D2);
\draw[-] (D3) -- (D2);
\draw[-] (D4) -- (D2);
\path[->,blue,thick]
(D1) edge [bend left=0] (D3)
(D3) edge [bend left=0](D4)
(D4) edge[bend left=90] (D1);
\path[->, thick, blue] (D2) edge[loop below] (D2); 
\def\Goffset{1.8}
\node (G2desc) at (6.8,\Goffset) {$\mathrm{G}_2$};
\node[dynkdot,label={above:\footnotesize$1$}] (G1) at (7.8,\Goffset){};
\node[dynkdot,label={above:\footnotesize$2$}] (G2) at (8.8,\Goffset) {};
\draw[-] (G1) -- (G2);
\draw[-] (G1.40) -- (G2.140);
\draw[-] (G1.320) -- (G2.220);
\draw[-] (8.25,\Goffset+.2) -- (8.35,\Goffset) -- (8.25,\Goffset-.2);
\draw[|->] (G2desc) -- (D4desc);
\path[-, dotted] (D1) edge (G1) (D2) edge (G2);
\end{tikzpicture}
\end{center}
\caption{$(\Dynkin,\sigma)$ for non-simply-laced $\g_0$} \label{Fig:unf}
\end{figure}
Thus let us associate
$(\Dynkin,\sigma)$ for each untwisted quantum affine algebra $\Uqpg$ consisting of (i) the Dynkin diagram $\Dynkin=\Dynkin_{\sfg}$
of $\sfg$ and the Dynkin diagram automorphism $\sigma$ on $\Dynkin_{\sfg}$ yielding $\Dynkin_{\g_0}$. 
We also call the pair $(\Dynkin_\sfg,\sigma)$ the \emph{unfolding} of $\g_0$. 
Then the index set $I_0=\{i,j,\ldots\}$ of $\g_0$ can be considered as
the orbit space of $\Dynkin_0=\sfI=\{\im,\jm,\ldots\}$ under the action of $\sigma$. 
Hence we understand $I_0 \ni i = \oim$ the orbit of $\im \in \sfI$.

\begin{definition} [\cite{FO21}]
\bna
\item 
A \emph{height function} on $(\Dynkin,\sigma)$ is a function $\xi\col \Dynkin_0 \to \Z$
satisfying the following conditions (here we write $\xi_\im \seteq \xi(\im)$):
\bnum
\item Let $\im,\jm \in \Dynkin_0$ with $d(\im,\jm)=1$ and $d_\oim = d_\ojm$. Then we have
$|\xi_\im-\xi_\jm|=d_\oim=d_\ojm$. 
\item Let $i,j \in I_0$ with $d(i,j)=1$ and $d_i =1 < d_j =r \seteq\ord(\sigma)$. Then there exists a unique 
$\jm \in j$ such that $|\xi_\im-\xi_\jm| =1$ and $\xi_{\sigma^k(\jm)} = \xi_\jm +2k$ 
for any $1 \le k < r$, where $i=\{ \im \}$. 
\ee
\item Such a triple $\calQ = (\Dynkin,\sigma,\xi)$ is called a 
\emph{$\rmQ$-datum} for $\g$. 
\ee
\end{definition}

For a twisted quantum affine algebra $U_q'(\sfg^{(t)})$ $(t=2,3)$,
we take the  \emph{$\rmQ$-datum} of $U_q'(\sfg^{(t)})$ to be the same 
as the $\rmQ$-datum of $U_q'(\sfg^{(1)})$; e.g., the $\rmQ$-datum of $U_q'(A_n^{(t)})$ coincides with the $\rmQ$-datum of $U_q'(A_n^{(1)})$, and so on.  
\renewcommand{\arraystretch}{1.5}
\begin{align} \label{Table: root system 2}
\small
\begin{array}{|c||c|c|c|c|c|}
\hline
  \g=\sfg^{(t)}  & A_{2n}^{(2)} \ (n\ge1)  & A_{2n-1}^{(2)} \ (n\ge2) & D_{n+1}^{(2)} \ (n\ge3)  & E_6^{(2)} & D_{4}^{(3)}    \\ \hline
  \g_0  & B_n  & C_n \ (n\ge2) & B_{n} & F_4& G_2    \\ \hline
 \sfg  & A_{2n} & A_{2n-1}    & D_{n+1}   &  E_6 & D_{4}   \\
\hline
\ord(\sigma)&1&1&1&1&1\\
\hline
\end{array}
\end{align}
Thus we have assigned a $\rmQ$-datum to every quantum affine algebra $\Uqpg$. 

\medskip
Let $\calQ = (\Dynkin,\sigma,\xi)$ be a $\rmQ$-datum for $\g$. 
A vertex $\im \in \Dynkin_0$
is called a \emph{sink} of $\calQ$ if we have $\xi_\im < \xi_\jm$ for any $\jm \in \Dynkin_0$
with $d(\im,\jm)=1$. When $\im$ is a sink of $\calQ$, we define a new $\rmQ$-datum
$\sfs_\im \calQ = (\Dynkin,\sigma,\sfs_\im\xi)$ of $\g$ with  
\begin{align} \label{eq: siQ}
(\sfs_\im \xi)_\jm \seteq \xi_\jm + 2d_\oim\, \delta_{\im,\jm} \quad \text{ for any } \jm \in \Dynkin_0.    
\end{align}

\begin{definition}\label{def:seq}
 Let $\uii=(\im_l,\im_{l+1},\ldots,\im_r)$ $(l \le r \in \Z)$  be a sequence in $\sfI$.
\bnum
\item $\uii$ is said to be \emph{reduced} if   $w^{\uii} \seteq s_{\im_l} \cdots s_{\im_r}\in\weyl$ has length $r-l+1$.  
\item For a reduced sequence $\uii$ and $k\in[l,r]$, we set $w^\uii_{\le k} \seteq s_{\im_l} \cdots s_{\im_k}$ and $w^\uii_{< k} \seteq s_{\im_l} \cdots s_{\im_{k-1}}$.
\item $\uii$ is said to be \emph{locally reduced} if $ (\im_k,\ldots,\im_{k+s-1})$ is a reduced sequence for any
  $k\in[l,r]$
and $1 \le s \le \ell(w_0)$ such that $k+s-1 \le r$.
\item For a $\rmQ$-datum $\calQ$ of $\g$,  $\uii$ with $l=1$ is said to be \emph{$\calQ$-adapted} or \emph{adapted to $\calQ$} if $\im_k$ is a sink of
the $\calQ$-datum $\sfs_{\im_{k-1}}\cdots \sfs_{\im_2}\sfs_{\im_1}\calQ$ for all $1 \le k \le r$.  
\ee
\end{definition}

For a reduced sequence $\usfwc = (\im_1,\ldots,\im_\ell)$ of $w_0$, we can obtain a locally reduced sequence 
$$\hwc \seteq (\ldots,\im_{-1},\im_0,\im_1,\ldots) $$
defined as follows:
\begin{align} \label{eq: extension of reduced} 
\im_{m\pm \ell} = \im_m^{*} \qquad \text{for any $m\in \Z$}.
\end{align}

Then the following are known (see~\cite{FO21} for more details):
\bna
\item For a $\rmQ$-datum $\calQ=\Qdatum$, there exists a reduced sequence $\usfwc$ of $w_0$ adapted to $\calQ$. 
\item For a $\rmQ$-datum $\calQ=\Qdatum$, there exists a unique \emph{Coxeter element} $\tau_\calQ \in \weyl \rtimes \ang{\sigma} \subset\Aut(\sfP)$ satisfying certain compatibility with $\calQ$. 
\ee

Let $\xi$ be a height function on $(\Dynkin,\sigma)$. We define a quiver $\hDynkin^\sigma =(\hDynkin^\sigma_0,\hDynkin^\sigma_1)$ as follows:
\begin{align*}
\hDynkin^\sigma_0 & = \{ (\im,p) \in \Dynkin_0 \times \Z \ | \  p - \xi_\im \in 2d_\oim\Z \},  \allowdisplaybreaks\\
\hDynkin^\sigma_1 & = \{ (\im,p) \to (\jm,s) \ | \ (\im,p),(\jm,s) \in \hDynkin^\sigma_0, \ d(\im,\jm)=1, \ s-p =\min(d_\oim,d_\ojm) \}. 
\end{align*}

Each reduced sequence $\usfwc=(\im_1,\ldots,\im_\ell)
$ of $w_0$ gives a labeling of $\Phi^+_\sfg$ as follows:
$$
\Phi^+_\sfg = \{ \be^\usfwc_k  \seteq s_{\im_1}\cdots  s_{\im_{k-1}} \al_{\im_k} \ | \ 1 \le k \le \ell \}.
$$

It is well-known that the total order $<_{\usfwc}$ on $\Phi^+_\sfg$,
defined by $\be^\usfwc_a <_{\usfwc} \be^\usfwc_b$ for $a < b$, is \emph{convex} in the following sense: if $\al,\be \in \Phi^+_\sfg$ satisfy
$\al <_{\usfwc} \be$ and $\al+\be \in \Phi^+_\sfg$, then $\al <_\usfwc \al+\be <_\usfwc \be$. For a pair of positive roots $\al,\be \in \Phi^+_\sfg$ with $\al \le_\usfwc \be$ and  $\ga\seteq \al+\be \in \Phi^+_\sfg$,
the pair $(\al,\be)$ is called \emph{$\usfwc$-minimal} if there exists no pair of positive roots $\al',\be'\in \Phi^+_\sfg$ such that 
$$\al'+\be'=\ga \qtq \al <_{\usfwc} \al' <_{\usfwc} \ga <_{\usfwc} \be' <_{\usfwc} \be. $$
Note that, for each $\rmQ$-datum $\calQ=(\Dynkin,\sigma,\xi)$ of $\g$,
there exists a unique bijection $$\phi_\calQ \col \hDynkin^\sigma_0 \to \Phi^+_\sfg \times \Z,$$ which is defined 
by using $\tau_\calQ$ (see \cite{HL15,FO21} and \cite{KO23} also).

Let $\Qdatum$ be a $\rmQ$-datum for $\g$. We set $\Gamma^\calQ=(\Gamma^\calQ_0,\Gamma^\calQ_1)$ the full-subquiver of $\hDynkin^\sigma$ whose set $\Gamma^\calQ_0$ of vertices 
is given as follows:
$$
\Gamma^\calQ_0 \seteq \phi_\calQ^{-1}(\Phi^+_\sfg \times \{ 0 \}) \subset \hDynkin^\sigma_0. 
$$

\subsection{Hernandez-Leclerc subcategories} \label{subsec: HL category}
In this subsection, we briefly review several subcategories of $\Cg$.
Recall $\sfI$ and the quiver $\hDynkin^\sigma$ for each quantum affine algebra $\Uqpg$.

\smallskip

For each $(\im,p) \in \sfI \times \Z$, we assign the fundamental module $L( \im,p)$
by following~\cite[\S\,6.2]{KKOP24A}. Then it is known that the Serre
 monoidal  subcategory $\scrC_\g^0$ of $\Cg$, generated by
$\{ L( \im,p) \ | \ (\im,p) \in \hDynkin^\sigma_0 \}$, forms a \emph{skeleton} subcategory in the following sense: For every prime simple module $M$, 
there exist a $x \in \bfk^\times$ and a prime simple module $L \in \scrC_\g^0$ such that $M \simeq L_x$.

Let us take a $\rmQ$-datum $\calQ$ of $\g$.
We define for each $\be \in \Phi^+_\sfg$
\begin{align}
L^\calQ(\be) \seteq  L( \im,p) \qt{where $\phi_\calQ(\im,p)=(\be,0)$.}     
\end{align}
When $\be$ is a simple root $\al_\im$, we frequently write $L^\calQ_\im$ for $L^\calQ(\al_\im)$.  
\begin{theorem} [\cite{CM05,KKOP22,FO21}] \label{thm: blcok decom}
For a $\rmQ$-datum $\calQ$ of $\g$, the category $\Cgz$
admits a block decomposition:
$$\Cgz = \bigoplus_{\be \in \rl_\sfg} \; (\Cgz)_\be.$$
Moreover we have
\bnum
\item $L^\calQ(\be)$ belongs to $(\Cgz)_\be$ for any $\beta\in\Phi^+_\sfg$,  
\item For $\be,\be'\in \rl_\sfg$, if
  $M\in(\Cgz)_\be$ and $M'\in(\Cgz)_{\be'}$, then $M\tens M'\in
  (\Cgz)_{\be+\be'}$.
\ee
\end{theorem}

By Theorem~\ref{thm: blcok decom}, for an indecomposable module $M \in \Cgz$,
we set $\wt_\calQ(M) \seteq \be$ if $M \in (\Cgz)_\be$.

\begin{theorem} 
  [{\cite[Theorem 4.6]{KKOP22}, see also \cite[Theorem 6.16]{FO21}}]
  \label{thm: wtpairing Lainf}
For simple modules $L$ and $L'$ in $\Cgz$, we have  
$$
\La^\infty(L,L') = -(\wt_\calQ(L),\wt_\calQ(L')) \quad \text{ for any $\rmQ$-datum $\calQ$ of $\g$.}
$$   
\end{theorem}

\smallskip 

For $m\in \Z$,
we define $\scrC_\calQ[m]$
as the smallest subcategory of $\Cgz$ containing $\st{ \scrD^m L^\calQ_\im
  \mid \im \in \sfI}\sqcup 
\{ \mathbf{1} \}$ and stable by taking tensor products, subquotients and extensions.  We write $\scrC_\calQ$ for $\scrC_\calQ[0]$.
We call $\scrC_\calQ$ the \emph{heart subcategory} associated with  the $\rmQ$-datum $\calQ$. 
The subcategories  $\scrC_\g^0$
and  $\scrC_\calQ$ of $\Cg$, introduced so far, are also referred to as the \emph{Hernandez-Leclerc} subcategories. 
It is proved in \cite{Her10} that there exists an isomorphism
between the  Grothendieck rings of a twisted quantum affine algebra and the corresponding
simply-laced quantum affine algebra:
\begin{align} \label{eq: untwisted to twisted}
K(\scrC_{\sfg^{(1)}}^0) \isoto  K(\scrC_{\sfg^{(t)}}^0)  \ \ (t=2,3),   
\end{align}
where $K(\Cgz)$ denotes the Grothendieck ring of $\Cgz$. As a ring, $K(\Cgz)$ is isomorphic to the
\emph{commutative} ring of the polynomials in $\{ [L( \im,p)] \}$ (\cite{FR99}). 

\smallskip

\subsection{Duality data}

{\em In the sequel, $\sfg$ denotes always a simply-laced finite-dimensional simple Lie algebra and
  $\sfI$  the index set of simple roots of $\sfg$.
} 

Let $\bbD =\{ L^\bbD_\im \}_{\im \in \sfI } \subset \Cgz$
be a family of real root modules.

\begin{definition}
A family of real root modules $\bbD =\{ L^\bbD_\im \}_{\im \in \sfI } \subset \Cgz$ is said to be a \emph{strong duality datum in $\Cgz$} if
\begin{align*}
\de(L^\bbD_\im,\scrD^k L^\bbD_\jm)=\delta(k=0)\,\delta\bl d(\im,\jm)=1\br \quad \text{ for } \im \ne \jm.    
\end{align*} 
\end{definition}

It is well-known that the family of root modules $\bbD_\calQ\seteq\{ L^\calQ_\im \}_{\im \in \sfI}$ for a $\rmQ$-datum $\calQ$ of $\g$
is a strong duality datum.

\smallskip

Let $\bbD =\{ L_\im \}_{\im \in \sfI }$ be a strong duality datum in $\Cgz$. 
For any $\jm \in \sfI$, we set 
\begin{align} \label{Def: refl}
 \scrS_\jm (\bbD) \seteq  \{ \scrS^\bbD_\jm (L_\im)  \}_{\im\in \sfI}\qtq
 \scrS_\jm^{\star} (\bbD) \seteq  \{ \scrS^{\star}_\jm{}^\bbD (L_\im)  \}_{\im\in \sfI},
\end{align}
where 
$$
\scrS_\jm ^\bbD(L_\im) \seteq  
\begin{cases}
 \scrD L_\im  & \text{ if } \im=\jm, \\
L_\jm \hconv L_\im &  \text{ if }  d(\im,\jm)=1, \\
L_\im &  \text{ if } d(\im,\jm)>1,
\end{cases}
\qtq
\scrS^{\star}_\jm{}^\bbD (L_\im) \seteq  
\begin{cases}
 \scrD^{-1} L_\im  & \text{ if } \im=\jm, \\
L_\im \hconv L_\jm &  \text{ if }  d(\im,\jm)=1, \\
L_\im &  \text{ if } d(\im,\jm)>1.
\end{cases}
$$
It is easy see that $\scrS_\jm \circ \scrS^{\star}_\jm (\bbD)= \scrS^{\star}_\jm \circ \scrS_\jm(\bbD)=\bbD $
by using Lemma~\ref{lem: recover}.  Hence we also write $\scrS_\jm^{-1}$
for $\scrS^{\star}_\jm$.  

\begin{proposition} [{\cite[Proposition 5.9]{KKOP23P}}] \label{prop: strong to strong} 
Let $\bbD$ be a strong duality datum and $\jm\in \sfI$.
\bnum
\item $\scrS_\jm(\bbD)$ and $\scrS_\jm^{-1}(\bbD)$ are strong duality data in $\Cgz$.
\item  For any $m\in\Z$,   $\D^m\bbD\seteq\st{\D^mL_\im^\bbD}_{\im\in\sfI}$ is a strong duality datum.  
\ee
\end{proposition}

For any $\jm\in\sfI$, we can regard $\scrS_\jm$ as an automorphism of
the set of isomorphism classes of strong duality data.

\Def Let $\bbD=\{ L^\bbD_\im \}_{\im \in \sfI }$ be a strong datum in $\Cgo$.
For an interval $[a,b]$ in $\Z$,
we define $\catD[a,b]$
as the smallest subcategory of $\Cgz$ containing $\st{ \scrD^m L^\bbD_\im
  \mid m\in[a,b],\im \in \sfI}\sqcup 
\{ \mathbf{1} \}$ and stable by taking tensor products, subquotients and extensions.  We write $\catD$, $\catD[m]$, $\catD{}_{,\ge m}$, $\catD{}_{,\le m}$
for $\catD[0,0]$, $\catD[m,m]$, $\catD{[m,+\infty]}$,
$\catD{[-\infty,m]}$, respectively.
\edf

When $\bbD =\bbD_\calQ$ for some $\rmQ$-datum $\calQ$ of $\g$, $\scrC_\bbD$ coincides with the heart subcategory $\scrC_\calQ$. Thus, for each strong datum $\bbD$, we also call $\scrC_\bbD$ a heart subcategory. 
 
\Rem
Let $\calQ$ be a $\rmQ$-datum of $\g$, and $\im\in\sfI$ a sink of
$\calQ$. 
Then we have
$$\scrS_\im\bbD_{\calQ}=\bbD_{\sfs_i\calQ }.$$
\enrem

\begin{definition}
A strong duality datum $\bbD$ of $\g$ is said to be \emph{complete} 
if, for each simple module $M \in \Cgz$,
there exist simple modules $M_k \in \scrC_\bbD$ $(k \in \Z)$  such that
\bna
\item $M_k \simeq \bone$  for all but finitely many $k$,
\item $M \simeq \head( \cdots \tens \rdual^2 M_2 \tens 
\rdual M_1 \tens M_0 \tens \ldual M_{-1}\tens \cdots)$.
\ee
\end{definition}

It is also known that $\bbD_\calQ$ is a complete duality datum
for any $\rmQ$-datum $\calQ$ of $\g$.

The multiplication induces an isomorphism 
\begin{align} \label{eq: iso Kz}
\stens_{m\in\Z}K(\scrC_\bbD[m]) \isoto K(\Cgz)    
\end{align}
for any complete duality datum $\bbD$ of $\g$ (see \cite[Theorem 6.10, Theorem 6.12]{KKOP23P}).

\begin{proposition} [{\cite[Theorem 6.3]{KKOP23P}}]\label{prop: complete to  complete} Let $\bbD=\{ L^\bbD_\im \}_{\im \in \sfI }$ be a complete duality datum in $\Cgz$ and $\jm\in \sfI$.
  Then $\scrS_\jm(\bbD)$ and $\scrS_\jm^{\star}(\bbD)$ are complete duality data
  in $\Cgz$.
\end{proposition}

\subsection{Quantum unipotent coordinate ring and upper global basis}
Let $\calU_q(\sfg)$ be the quantum group of $\sfg$ over $\kk$. We denote by $\calU_q^-(\sfg)$
the negative half of $\calU_q(\sfg)$.

Let $B(\infty)$ be the \emph{infinite crystal} of $\calU_q^-(\sfg)$, and let $\tf_\im$ and $\te_\im$ be the \emph{crystal operators} for $B(\infty)$. For any $b\in B(\infty)$, $\wt(b)$ stands for the weight of $b \in B(\infty) $.

Set $\calA_q(\sfn) \seteq \bigoplus_{\be \in \sfQ^-} \calA_q(\sfn)_\be$, where $\calA_q(\sfn)_\be \seteq \Hom_{\kk}(\calU_q^-(\sfg)_\be,\kk)$. Then   
$\calA_q(\sfn)$ has an algebra structure isomorphic to $\calU_q^-(\sfg)$
and is called the \emph{quantum unipotent coordinate ring} of $\sfg$. 

Let  
$$
\Ang{ \ , \ } \colon \calA_q(\sfn) \times \calU_q^-(\sfg) \to \kk
$$
be the pairing.  For each $\im \in \sfI$, we denote by $\ang{\im} \in \calA_q(\sfn)_{-\al_\im}$ the dual element of $f_\im$ with respect to $\Ang{ \ , \ }$; i.e., 
$$
\Ang{\ang{\im},f_\jm} =\delta_{\im,\jm} \quad \text{ for any $\im,\jm\in \sfI$.}
$$
Then the set $\iset{\ang{\im} }$ generates $\calA_q(\sfn)$.

Note that there exists a $\kk$-algebra isomorphism 
\begin{align} \label{eq: iota}
\iota\colon \calU_q^-(\sfg) \isoto \calA_q(\sfn)   \qquad f_\im \longmapsto 
\zeta^{-1}\ang{\im}
\end{align}
for any $\im\in \sfI$, where
$$\zeta\seteq 1 -q^{ 2}.$$ 
We define a bilinear form $\aform{ \ , \ }$ on $\Aqn$ by 
\begin{align} \label{eq: aform}
\aform{f,g} \seteq \Ang{ f, \iota^{-1}(g)} \quad \text{ for any } f,g \in \Aqn.     
\end{align}

We denote by $\calU_\bfA^-(\sfg)$ to be the $\bfA$-subalgebra of $\calU_q^-(\sfg)$ generated by $f_\im^{(n)}\seteq f_\im^n/[n]!$ $(\im \in \sfI, \ n \in \Z_{>0})$,
and by $\calA_\bfA(\sfn)$
the $\bfA$-submodule of $\calA_q(\sfn)$ generated by $\uppsi \in \calA_q(\sfn)$ such that $\uppsi(\calU^-_\bfA(\sfn)) \subset \bfA$. 
Then, $\calA_\bfA(\sfn)$ is  a $\bfA$-subalgebra of $\calA_q(\sfn)$. 

Let
$\bbG \seteq \{ \Gup(b) \ | \ b \in B(\infty) \}$ be the \emph{upper global basis} of $\calA_\bfA(\sfn)$ (see \cite{K91, K93, K95} for its definition and properties). 
Set
$$
\LupA \seteq \sum_{b \in B(\infty)} \Z[q^{1/2}]\Gup(b) \subset \calA_\bfA(\sfn).
$$

We regard $B(\infty)$ as a basis of $\LupA/q^{1/2}\LupA$ by 
\begin{align} \label{eq: local basis}
b \equiv \Gup(b) \qquad {\rm mod} \; q^{1/2}  \LupA.    
\end{align}

We know that $\aform{\Gup(b),\Gup(b')}|_{q^{1/2}=0} =\delta_{b,b'}$ and
hence $B(\infty)$ is an orthonormal basis of
$\LupA/q^{1/2}\LupA$, which implies that
the lattice $\LupA$ is characterized by 
$$
\LupA = \{ x \in \Azn \ | \  \aform{x,x} \in \Z[[q^{1/2}]] \subset \Q(\hspace{-.4ex}(q^{1/2})\hspace{-.4ex}) \}.
$$

\subsection{Braid symmetry and dual root vectors}
Recall that $\weyl$ denotes
the Weyl group associated with  
a simply-laced finite-dimensional simple Lie algebra $\sfg$.  
Let us denote by $\ttB$ the \emph{braid} group or the \emph{Artin-Tits} group associated with $\Dynkin$. 
The  braid group is generated by $\bg_{\im}$ $(\im \in \sfI)$ subject to the commutation relations and
the braid relations. 
Let us denote by
$\pi \col \ttB \twoheadrightarrow \weyl $
the natural projection
sending $\bg_\im$ to $s_\im$ for all $\im \in \sfI$.
We denote by $\ttB^\pm$
the submonoid of $\ttB$ generated by $\st{\bg_\im^\pm\mid\im \in \sfI}$.  

Note that any sequence $\uii=(\im_1,\ldots,\im_r) \in \sfI^r$
corresponds to an element $\ttb^{\uii} \seteq \bg_{\im_1}\bg_{\im_2}\cdots\bg_{\im_r}$ in $\ttB^+$.  
We denote by $\Seq(\ttb)$ the set of all $\uii$'s giving $\ttb$. 

\smallskip

We denote by $\rev\col \ttB\isoto\ttB$ the anti-automorphism of
$\ttB$ sending $\bg_\im$ to itself.

\begin{definition} \label{def: moves}
  Let $\uii=(\im_l,\im_{l+1},\ldots,\im_r)$ and $\ujj=(\jm_l,\jm_{ l+1 },\ldots,\jm_r)$ be sequences in $\sfI$.
\bnum
\item  
We say that $\ujj$ can be obtained from $\uii$ via a \emph{commutation move} 
if there exists a $k\in\Z$ such that
$$\text{$l\le k <r$,\quad $\im_s = \jm_s$ for $s \ne k,k+1$,\quad
  $\im_k=\jm_{k+1}$, $\im_{k+1}=\jm_k$  and $d(\im_k ,\im_{k+1}) > 1$.}$$
In this case, we write $\ujj = \upga_k(\uii)$.
\item \label{it: braid move}  We say that $\ujj$ can be obtained from $\uii$ via a \emph{braid move} if
  there exists $k\in\Z$ such that
  $$\parbox{75ex}{$l\le k \le r-2$,\quad $\im_s = \jm_s$ for $s \ne  k,k+1,k+2 $,\\
    \quad  $\im_k=\im_{k+2}=\jm_{k+1}$,\quad $\im_{k+1}=\jm_k=\jm_{k+2}$ and $d(\im_k ,\im_{k+1}) = 1$.}$$
In this case, we write $\ujj = \upbe_k(\uii)$. 
\ee 
In the both cases, $\ttb^\uii=\ttb^\ujj$ as an element of $\ttB^+$. 
\end{definition}

Now we recall the braid symmetry on  $\calU_q(\sfg)$ by mainly following \cite{LusztigBook}. For $\iinI$, we set $\sfS_{\im} \seteq T_{\im,-1}'$
and $\sfS_\im^{*} \seteq T_{\im,1}''$, which are inverse to each other.
The description of $\sfS_\im$ is given as follows $(\im \ne \jm \in \sfI)$: 
\begin{align*}
&\sfS_\im(t_\im)\seteq t_\im^{-1}, \qquad  \sfS_\im(t_\jm) \seteq t_\jm t_\im^{-\ang{h_\im,\al_\jm}},  \qquad 
\sfS_\im(f_\im) \seteq -e_\im t_\im , \qquad \sfS_\im(e_\im) \seteq -t_\im^{-1}f_\im, \\
&\sfS_\im(f_\jm) \seteq \bc  f_\im f_\jm  -q   f_\jm f_\im  & \text{ if } d(\im,\jm)=1,  \\ 
f_\jm  & \text{ if } d(\im,\jm)>1, \ec \ \
\sfS_\im(e_\jm) \seteq \bc e_\jm e_\im -q^{-1} e_\im e_\jm & \text{ if } d(\im,\jm)=1,  \\  
e_\jm  & \text{ if } d(\im,\jm)>1. \ec
\end{align*}  
Note that $\{ \sfS_\im\}_{\iinI}$ satisfies the relations of $\ttB_\sfg$ and hence $\ttB_\sfg$ acts on $\calU_q(\sfg)$ via $\{ \sfS_\im\}_{\iinI}$.

Let us take an element $w$ in $\weyl$.
For a reduced sequence $\uw=(\im_1,\im_2,\ldots,\im_r)$ of $w$  and $1 \le k \le r$,  we set 
\eq&&
  E _{\uw}(\be_k) \seteq \sfS_{\im_1} \ldots \sfS_{\im_{k-1}} (f_{\im_k})\in \calU_\bfA^-(\sfg) \qtq  
  E^*_{\uw}(\be_k) \seteq   \zeta \iota\bl E_{\uw}(\be_k)\br.
  \label{eq: PBW}
  \eneq
Note that when $\be_k = \al_\im$ for some $\im \in \sfI$, $E_{\uw}(\be_k)=f_\im$ and $E^*_{\uw}(\be_k)$ is equal to $\ang{\im}$.

It is known that   $E^*_{\uw}(\be_k)$  
belongs to $\Azn$ and  is called the \emph{dual root vector} corresponding to $\be_k$ and $\uw$.

The $\bfA$-subalgebra of $\Azn$ generated by $\{ E^*_{\uw}(\be_k) \}_{1 \le k \le r}$
does not depend on the choice of a reduced expression $\uw$ of $w$,
which we denote by
$\calA_\bfA(\sfn(w))$ (see~\cite[Section 4.7.2]{Kimura12}).  We call $\calA_\bfA(\sfn(w))$ the \emph{quantum unipotent coordinate ring associated with $w$}.

\subsection{Schur-Weyl duality functor} In this subsection, we briefly review Schur-Weyl duality functors
between categories over a quiver Hecke algebra and a quantum affine algebra for our purpose (see~\cite{KKOP24A} for more detail). 

We first review the quiver Hecke algebra associated with a finite simple Lie algebra $\sfg$
of simply-laced type. 
Take a family of polynomial $(\ttQ_{\im\jm})_{\im,\jm \in \sfI}$ in $\bfk[u,v]$
such that
$$   \ttQ_{\im\jm} (u,v) = \pm\delta(\im \ne \jm)(u-v)^{-(\al_\im,\al_\jm)}
\qtq\ttQ_{\im\jm}(u,v)=\ttQ_{\jm\,\im}(v,u). $$

For each $\be \in \rl^+$ with $|\be|=n$, we set $\sfI^\be \seteq \{
\nu=( \nu_1,\ldots,\nu_n) \in \sfI^n  \ | \ \sum_{k=1}^n
\al_{\nu_k} = \be \}$. 

\smallskip

The \emph{symmetric quiver Hecke algebra} $R(\be)$ at $\be \in \rl^+$
associated to $\sfg$ and $(\ttQ_{\im\jm})_{\im,\jm\in  \sfI }$, is the $\Z$-graded
$\C$-algebra generated by the elements $\{ e(\nu) \}_{\nu \in
\sfI^\be}$, $\{ x_k \}_{1\le k\le n}$ and $\{ \tau_m \}_{1\le m\le n-1}$ satisfying the certain defining relations (see \cite[Definition 2.1.1]{KKKO18} for more details). 

\smallskip

Let us denote by $R(\beta)\gmod$ the category of finite-dimensional graded $R(\be)$-modules, and we set $R\gmod=\soplus_{\beta\in\rl^+}R(\beta)\gmod$. For an
$R(\beta)$-module $M$, we set $\wt(M)\seteq -\beta \in \rl^-$. For
the sake of simplicity, we  say that $M$ is an $R$-module instead of
saying that $M$ is a graded $R(\beta)$-module. 
For a graded $R(\be)$-module $M = \soplus_{k \in \Z} M_k$, we define $qM = \soplus_{k\in \Z} (qM)_k$,
where $(qM)_k = M_{k-1}$ $(k\in \Z)$.  
We call $q$ the \emph{grading shift functor} on the category of graded $R(\be)$-modules. Thus the Grothendieck group $K(R(\be)\gmod)$ of $R(\be)\gmod$
has a $\Z[q^{\pm 1}]$-module structure induced by the grading shift functor. 
 For an $R(\beta)$-module $M$ and an $R(\gamma)$-module $N$, we define their \emph{convolution product} $\conv$ by 
$$
M\conv N \seteq R(\beta+\gamma)e(\beta,\gamma) \otimes_{R(\beta) \otimes R(\gamma)} (M  \otimes  N),
$$
where $e(\beta,\gamma) =\displaystyle\sum_{ \nu_1 \in I^{\beta},  \nu_2 \in I^{\gamma} } e(\nu_1*\nu_2)$. Here $\nu_1*\nu_2$ is 
the concatenation of $\nu_1$ and $\nu_2$.

Note that  
$$K(R\gmod)\seteq \soplus_{\be \in \rl^+}K(R(\be)\gmod)$$
has a $\Z[q^{\pm1}]$-algebra structure by the convolution product $\conv$
and the grading shift functor $q$.

For $\im \in \sfI$,  $L(\im)$ denotes the $1$-dimensional
simple graded
$R(\al_\im)$-modules $\bfk u(\im)$ with the action
$x_1u(\im)=0$.

\begin{theorem} [{\cite{KL1,R08,VV09}}] \label{thm: KLR iso}
There exists a $\Z[q^{\pm 1/2}]$-algebra isomorphism
\begin{align} \label{eq: categorification}
{\rm ch}_q \colon \calK(R\gmod) \seteq \bfA \tens_{\Z[q^{\pm 1}]} K(R\gmod) \isoto \calA_\bfA(\sfn),    
\end{align}
sending $[L(\im)]$ to $\ang{\im}$. 
Furthermore, under the isomorphism ${\rm ch}_q$, the upper global
basis $\bbG$ of $\calA_\bfA(\sfn)$ corresponds to the
set of the isomorphism classes of self-dual simple $R$-modules.
\end{theorem}

For $k=1,\ldots,\ell$, let $\sfV^\usfwc_k$ be the \emph{cuspidal module} corresponding to $\be_k$ with respect to $\usfwc$
(see \cite[Section 2]{KKOP18} for the precise definition). 
Under the categorification in~\eqref{eq: categorification}, the cuspidal module $\sfV^{\usfwc}_k$ corresponds to the dual root vector $E^*(\beta_k)$ in $\calA_\bfA(\sfn)$.
Note the followings:
\bnum
\item For a minimal pair $(\be_a,\be_b)$ of $\be_k$, there exists an isomorphism 
$$
\sfV^\usfwc_a \hconv \sfV^\usfwc_b \simeq \sfV^\usfwc_k.
$$
\item For $1 \le a \le \ell$ with $\beta_a=\al_\im$, $\sfV^\usfwc_a \simeq L(\im)$.
\ee
See \cite[Lemma 4.2]{McNa15} and \cite[Section 4.3]{BKM12} for more details.

\begin{theorem}[{\cite{KKK18,KKOP23P}}] \label{thm:gQASW duality} For a given strong duality datum $\bbD=\{L^\bbD_\im \}_{\im \in \sfI}$ in $\Cgz$,
  there exists a functor
\begin{align} \label{eq: FD}
\calF_\bbD \colon  R\gmod \rightarrow  \scrC_\bbD    
\end{align}
satisfies the following properties$\colon$ 
\bna
\item \label{it: Lim angleim} $\calF_\bbD(L(\im))\simeq L^\bbD_\im$.
\item The functor $\calF_\bbD$ is an {\em exact} functor on $R\gmod$ such that, for any $M_1, M_2 \in  R\gmod$, we have
isomorphisms
$$\calF_\bbD(R(0)) \simeq \bfk, \quad \calF_\bbD(M_1 \conv M_2) \simeq \calF_\bbD(M_1) \tens \calF_\bbD(M_2),$$
and $\calF_\bbD$ sends simple modules to simple modules.  
\ee
\end{theorem}
We call $\calF_\bbD$ the \emph{quantum affine Schur-Weyl duality functor} associated with $\bbD$.

Let us set 
$$
\obbK (R\gmod)  \seteq  \calK(R\gmod)/(1-q^{1/2}) \calK(R\gmod).
$$

\begin{theorem}[\cite{KKOP23P}]
\label{thm: isomorphisms} Let $\bbD$ be a strong duality datum in $\Cgz$. Then
$\calF_\bbD$ induces a $\Z$-algebra isomorphism 
\begin{align} \label{eq: [FD] iso}
[\calF_\bbD] \colon  \obbK(R\gmod) \isoto K(\scrC_\bbD),    
\end{align}  
where $K(\scrC_\bbD)$ denotes the Grothendieck ring of $\scrC_\bbD$.
\end{theorem}

\section{Relation between bosonic extension and the skeleton category} \label{Sec: BE sc}
In this section, we first review the definition of bosonic extensions $\hcalA$ of quantum unipotent coordinate rings, their global bases and braid group symmetries, which are investigated in~\cite{HL15,FHOO,FHOO2,KKOP21B,JLO1,JLO2,OP24,KKOP24,KKOP24B}.  
Then we study the relation between $\hcalA$ and the skeleton category. In particular, we shall prove that the induced braid
symmetries on simple modules in the category preserve the $\Z$-invariants when those modules are contained in a heart subcategory.

\subsection{Bosonic extension} \label{Sec: BE}
In this subsection, we recall the bosonic extension $\hcalA$ 
associated with a   finite-dimensional  simple Lie algebra $\sfg$
of simply-laced type, even though  $\hcalA$ is defined for an arbitrary
symmetrizable Kac-Moody algebra \cite{KKOP24}.

\begin{definition} \label{def: Bosonic ext}
The \emph{bosonic extension} $\hcalA$ of $\calA_q(\sfn)$ is the  
$\kk$-algebra generated by $\{ f_{\im,p} \}_{(\im,p) \in \sfI  \times \Z}$ subject to the following relations:
For any $\im,\jm \in \sfI$ and $m,p \in \Z$,
\bna
\item \label{it: relation1}
$\displaystyle\sum^{1-\lan h_\im,\al_\jm \ran}_{k=0} (-1)^k  \Bigl[\begin{matrix}1-\lan h_\im,\al_\jm \ran \\ k\\ \end{matrix}\Bigr]  {f_{\im,p}}^{1-\lan h_\im,\al_\jm \ran-k} f_{\jm,p} f_{\im,p}^{k} = 0 \quad \text{ for }  \im \ne \jm \in \sfI $,
\item \label{it: relation2} 
$f_{\im,m} f_{\jm,p} =  q^{(-1)^{p-m+1}(\al_\im,\al_\jm)}  f_{\jm,p} f_{\im,m}+ \delta(\im=\jm)\,\delta(p=m+1)\;(1- q^{2}) \quad \text{ if $m<p$.}$ 
\ee
\end{definition}
With the assignment $\wt(f_{\im,m})=(-1)^{m+1}\al_{\im }$,
the relations of $\hcalA$ in~\eqref{it: relation1} and~\eqref{it: relation2} are homogeneous.  
Thus we have a $\sfQ$-weight space decomposition of $\hcalA$:
$$
\hcalA = \bigoplus_{\be \in \sfQ} \hcalA_\be.
$$

\begin{definition}
For $ - \infty \le a \le b \le \infty$, let $\hcalA[a,b]$ be the $\Q(q^{1/2})$-subalgebra
of $\hcalA$ generated by $\{ f_{\im,k} \ | \  \im \in \sfI, a\le k \le b\}$.
We simply write 
$$
\hcalA[m] \seteq \hcalA[m,m], \quad 
\hcalA_{\ge m} \seteq \hcalA[m,\infty], \quad
\hcalA_{\le m} \seteq \hcalA[-\infty,m]. 
$$
Similarly, we set $\hcalA_{>m} \seteq \hcalA_{\ge m+1}$
and $\hcalA_{<m} \seteq \hcalA_{\le m-1}$. 
\end{definition}

Note that we have the following (anti-)automorphisms on $\hcalA$:
\bnum
\item There exists a $\Q$-algebra anti-automorphism  $\calD_q$ of $\hcalA$ such that
$$\calD_q(q^{\pm 1/2}) = q^{\mp 1/2} \qtq \calD_q(f_{\im,p})= f_{\im,p+1}.$$
\item There exists a $\Q$-algebra anti-automorphism $\overline{\phantom{a}}$ of $\hcalA$, called the \emph{bar-involution}, such that
$$\overline{q^{\pm 1/2}} = q^{\mp 1/2} \qtq \overline{f_{\im,p}}= f_{\im,p}.$$
\item There exists a $\kk$-algebra automorphism 
\begin{align} \label{eq: shift}
\ocalD_q = \overline{\phantom{a}} \circ \calD_q = \calD_q \circ \overline{\phantom{a}}    
\end{align}
 on $\hcalA$ defined by  $\ocalD_q(f_{\im,p})=f_{\im,p+1}$ for all $\im \in \sfI$
and $p \in \Z$. 
\ee

We define a $\Q$-linear map $c\colon \hcalA \to \hcalA$ by 
\begin{align} \label{eq: c map}
c(x)\seteq q^{(\wt(x),\wt(x))/2} \overline{x} \quad\text{for any homogeneous element } x \in \hcalA.
\end{align}
 
\begin{theorem} [{\cite[Corollary 5.4]{KKOP24}}] \label{thm: serial decomp} 
For any $a,b \in \Z$ with $a \le b$, the $\kk$-linear map 
$$
\hcalA[b] \otimes_\kk \hcalA[b-1] \otimes_\kk \cdots \otimes_\kk \hcalA[a+1] \otimes_\kk \hcalA[a] \to \hcalA[a,b]
$$
defined by $x_b \tens x_{b-1} \tens \cdots \tens x_{a+1} \tens x_a \longmapsto x_bx_{b-1}\cdots x_{a+1}x_a$ is an isomorphism. 
\end{theorem}

For homogeneous elements $x,y \in \hcalA$, we set
\begin{align*}
[x,y]_q \seteq xy - q^{-(\wt(x),\wt(y))}yx.     
\end{align*}

For any $\im \in \sfI$ and $m\in \Z$, let $\rmE_{\im,m}$ and $\Es_{\im,m}$ to be the endomorphisms 
of $\hcalA$ defined by
\begin{align} \label{eq: Ei Esi}
\rmE_{\im,m}(x) \seteq [x,f_{\im,m+1}]_q \qtq  \Es_{\im,m}(x) \seteq [f_{\im,m-1},x]_q     
\end{align}
for any homogeneous element $x \in \hcalA$. For any $n\in \Z_{\ge0}$, we set
$$
\rmE_{\im,m}^{(n)} \seteq  \dfrac{1}{[n]!}  \rmE_{\im,m}^n, \qtq \Esn{(n)}_{\im,m} \seteq \dfrac{1}{[n]!} \Esn{n}_{\im,m}.
$$

For any homogeneous $x,y \in \hcalA$, one can easily check that
\begin{align*}
\rmE_{\im,m}(xy) &= x \rmE_{\im,m}(y) + q^{-(\al_{\im,m},\wt\, y)} \rmE_{\im,m}(x)y, \\    
\Es_{\im,m}(xy) &=  \Es_{\im,m}(x)y + q^{-(\al_{\im,m},\wt\, x)} x\Es_{\im,m}(y).    
\end{align*}

From Theorem~\ref{thm: serial decomp},  we have the decomposition
\begin{align} \label{eq: hcalA decomposition}
\hcalA = \bigoplus_{(\be_k)_{k \in \Z} \in \rl^{\oplus \Z}} \oprod_{k \in \Z} \hcalA[k]_{\be_k}.     
\end{align}
Define 
\begin{align} \label{eq: def of Mn}
\Mn\colon \hcalA\longepito \kk
\end{align}
to be the natural projection
$ \hcalA\longepito \displaystyle\oprod_{k\in\Z}\hcalA[k]_{0}\simeq\kk$
arising from the decomposition \eqref{eq: hcalA decomposition}.

\begin{definition}
We define a bilinear form on $\hcalA$ as follows:
\begin{align} \label{eq: hA form}
\hAform{x,y} \seteq \Mn(x \ocalD_q (y)) \in \kk \quad \text{ for any } x,y \in \hcalA,   
\end{align}
where $\ocalD_q$ is the automorphism of $\hcalA$ given in~\eqref{eq: shift}.
\end{definition}

\begin{theorem}[{\cite[Lemma 6.3, Theorem 6.4]{KKOP24}}] \label{thm: hAform} \hfill  
\bnum
\item \label{it: hAform sym nond}
The bilinear form $\hAform{ \ , \ }$ is symmetric and non-degenerate.
\item \label{it: hAform 2}  If $x$ and $y$ are homogeneous elements such that $\wt(x)\not=\wt(y)$, then $\hAform{x,y}=0$.
  \item \label{it Est}  For any $m\in\Z$ and $\im\in\sfI$, we have
  $\rmE_{\im,m}\hA_{\le m}\subset \hA_{\le m}$ and
  $\Es_{\im,m}\hA_{\ge m}\subset \hA_{\ge m}$. 
\item  For any $x,y\in\hcalA_{\le m}$ and $u, v \in\hcalA_{\ge m}$, we have
$$
\hAform{f_{\im,m}x,y}=\hAform{ x,\rmE_{\im,m}(y)} \qtq \hAform{u , vf_{\im,m} }=\hAform{\Es_{\im,m}(u) ,v}.
$$
\ee   
\end{theorem}
Note that the first statement in \eqref{it Est} easily follows from $\rmE_{\im,m}(1)=0$ and
$\rmE_{\im,m}(f_{\jm,k})=\delta(\im=\jm)\delta(m=k)(1-q^2)$
for any $\jm\in \sfI$ and $k\in\Z$ such that $k\le m$. 

\subsection{Bosonic extension at $q=1$} Note that we have the $\kk$-algebra isomorphism 
\begin{align} \label{eq: vph}
\varphi_{k} \colon \Aqn   \isoto  \hcalA[k] \qquad \text{ by } \varphi_k(\ang{\im}) = q^{1/2} f_{\im,k}.      
\end{align}
For $k \in \Z$ and $\iinI$, we define
\begin{align*}
\hcalA[k]_\Zq  &\seteq \vph_{k}(\Azn) \subset \hcalA 
\end{align*}
and set
$$
\hcalA[a,b]_\Zq \seteq \oprod_{k \in [a,b]} \hcalA[k]_\Zq\subset \hcalA,  \qquad 
\hcalA_\Zq \seteq \sbcup_{a \le b} \hcalA[a,b]_{\Zq}\subset \hcalA. 
$$

\begin{proposition} [{\cite[Proposition 7.2]{KKOP24} }]\label{prop: hcalA integral form} 
$\hcalA_\Zq$ is a $\Zq$-subalgebra of $\hcalA$, and 
$$\kk \otimes_\Zq \hcalA_\Zq \isoto \hcalA.$$ 
In particular, we have
\begin{equation} \label{eq: adjacent}
\begin{aligned}
& \hAmz{m}\hAmz{m-1} \\
& \qquad   = \left\{  x \in \hAm{m-1,m} \ | \  \hAform{x,uv} \in \Zq  \right. \allowdisplaybreaks \\
& \hspace{15ex} \left. \text{ for any } u \in \vph_m \circ \iota\big(\calU^-_\bfA(\g)\big) \text{ and } v \in \vph_{m-1} \circ \iota\big(\calU^-_\bfA(\g)\big) \right\}.
\end{aligned}
\end{equation}
\end{proposition}

\begin{proposition} \label{prop: comm}
The $\Z$-algebra 
\begin{align} \label{eq: commutative algebra bbA}
\obbA \seteq \text{$\hcalA_\Zq/(q^{1/2}-1)\hcalA_\Zq$  is commutative.}
\end{align}
\end{proposition}

\begin{proof}
It is known that $\hAmz{m}/(q^{1/2}-1)\hAmz{m} \simeq \calA_\Zq(\sfn)/(q^{1/2}-1)\calA_\Zq(\sfn)$ is commutative for each $m$. Since     
$\hAmz{n}$ and $\hAmz{m}$ $q$-commutes if $n>m+1$ by Definition~\ref{def: Bosonic ext}~\eqref{it: relation2}, it is enough to show that
\begin{align} \label{eq: commutative integral form}
 xy-yx \in (q^{1/2}-1)\hcalA_\Zq    \text{ if } x \in \hAmz{m+1}\text{ and } y \in \hAmz{m}.
\end{align}

In order to see~\eqref{eq: commutative integral form}, it is enough to show
$$
\hAform{xy-yx,uv} \in \Zq(q^{1/2}-1) 
$$
for any homogeneous $u \in \varphi_{m+1}\circ \iota (\calU^-_\bfA(\g))$ and $v \in \varphi_{m} \circ \iota (\calU^-_\bfA(\g))$ by~\eqref{eq: adjacent}. 
We shall prove this by induction on $\het(\wt (u))+\het(\wt (v))$. 

Assume that $v= v'f_{\im,m}^{(k)}\zeta_{\im}^{-k}$ with $v' \in \varphi_{m}\circ\iota (\calU^-_\bfA(\g))$ and $k>0$. Then we have
$$
\hAform{xy-yx,uv} =\hAform{xy-yx,uv'f_{\im,m}^{(k)}\zeta_{\im}^{-k}} = \hAform{ \zeta_{\im}^{-k}\rmE^{\star(k)}_{\im,m}(xy-yx),uv' }.
$$

Recall that $\hAmz{m}$ is stable by $\zeta_{\im}^{-k}\rmE^{(k)}_{\im,m}$ and $\zeta_{\im}^{-k}\rmE^{\star(k)}_{\im,m}$. Since $\rmE^{\star}_{\im,m}(x)=0$, we have
$$
\zeta_{\im}^{-k}\rmE^{\star(k)}_{\im,m} (xy-yx) \equiv \left(x \zeta_{\im}^{-k}\rmE^{\star(k)}_{\im,m}(y) - \zeta_{\im}^{-k}\rmE^{\star(k)}_{\im,m}(y)x \right) \ \ {\rm mod} \ (q^{1/2}-1)\hcalA_\Zq,
$$
 and hence we obtain
$$
\hAform{xy-yx,uv} = \hAform{x \zeta_{\im}^{-k}\rmE^{\star(k)}_{\im,m}(y) - \zeta_{\im}^{-k}\rmE^{\star(k)}_{\im,m}(y)x,uv'} \equiv 0 \ \ {\rm mod} \ (q^{1/2}-1)\Zq, 
$$
by the induction. 

Similarly, if $u = \zeta_{\im}^{-k}f^{(k)}_{\im,m+1}u'$ with $u' \in  \varphi_{m+1}\circ\iota (\calU^-_\bfA(\g))$ and $k>0$, we have
$$
\hAform{xy-yx,uv} =\hAform{xy-yx,\zeta_{\im}^{-k}f_{\im,m+1}^{(k)}u'v} = \hAform{ \zeta_{\im}^{-k}\rmE^{(k)}_{\im,m}(xy-yx),u'v }.
$$
Since $\rmE_{\im,m+1}(y)=0$, we have
$$
\zeta_{\im}^{-k}\rmE^{(k)}_{\im,m} (xy-yx) \equiv \left(\zeta_{\im}^{-k}\rmE^{(k)}_{\im,m}(x)y - y\zeta_{\im}^{-k}\rmE^{(k)}_{\im,m+1}(x) \right), \ \ {\rm mod} \ (q^{1/2}-1)\hcalA_\Zq,
$$
and hence we obtain
\[
\hAform{xy-yx,uv} = \hAform{\zeta_{\im}^{-k}\rmE^{(k)}_{\im,m}(x)y - y\zeta_{\im}^{-k}\rmE^{(k)}_{\im,m+1}(x)} \equiv 0 \ \ {\rm mod} \ (q^{1/2}-1)\Zq. \qedhere
\]
\end{proof}

\begin{remark}
Even though we consider $\hcalA$ associated with a finite simple Lie algebra $\sfg$
of simply-laced type, Proposition~\ref{prop: comm} still holds
for an arbitrary symmetrizable Kac-Moody algebra. 
\end{remark}

We denote by the canonical map
\begin{align} \label{eq: overline convention}
\ev_{q=1}\col \hAz \twoheadrightarrow \obbA   
\end{align}
For each interval $[a,b]$, $[a,\infty]$ and $[-\infty,b]$, we define $\obbA[a,b]$, $\obbA_{\ge a}$ and $\obbA_{\le b}$ respectively, in an obvious way.

\subsection{Global basis}
We now define $\Z[q^{1/2}]$-lattices as follows:
\begin{equation} \label{Eq: def of Zq lattices of hA}
\begin{aligned}
\hLup[k] \seteq \vph_k\left(\Lup\big(\Azn\big)\right), \ \ \hLup[a,b]  \seteq  \oprod_{k \in [a,b]} \hLup[k],  \ \
\hLup  \seteq \sbcup_{a \le b} \hLup[a,b]. 
\end{aligned}
\end{equation}

The notion of \emph{extended crystal} of $\hB(\infty)$ is introduced in~\cite{KP22} and  defined as 
\begin{align} \label{Eq: extended crystal}
	\hB(\infty) \seteq   \Bigl\{  (b_k)_{k\in \Z } \in \prod_{k\in \Z} B(\infty)  \bigm | b_k =\mathsf{1} \text{ for all but finitely many $k$}  \Bigr\}.
\end{align}
Here $\mathsf{1}$ is the highest weight element of $B(\infty)$.

For any $\bfb =(b_k)_{k\in\Z} \in \hB(\infty)$, we set
$$
\rmP(\bfb) \seteq \oprod_{k \in \Z} \vph_k(\Gup(b_k)) \in \hLup. 
$$
Then, $\{\rmP(\bfb)\mid\bfb\in \hBi\}$ forms a $\Z[q^{1/2}]$-basis of $\LuphA$.

We regard $\hB(\infty)$ as a $\Z$-basis of $\hLup/q^{1/2}\hLup$ by 
$$
\bfb  \equiv \rmP(\bfb) \ {\rm mod} \ q^{1/2}\hLup. 
$$

\begin{theorem}[{\cite[Theorem 7.6]{KKOP24}}] \label{Thm: global basis} \hfill 
\bnum
\item \label{it: global (i)}
For each $\bfb = (b_k)_{k \in \Z} \in\hBi$, there exists a unique
$\rmG(\bfb)\in \LuphA$ such that
$$\rmG(\bfb)-\rmP(\bfb)\in \displaystyle\sum_{\bfb'\boldsymbol{\prec}^*\bfb}q\Z[q]\rmP(\bfb') \qtq  c(\rmG(\bfb))=\rmG(\bfb),$$
where $\boldsymbol{\prec}^*$ is a certain order on $\hBi$ \ro see \cite[(7.4)]{KKOP24} for the definition of $\boldsymbol{\prec}^*$\rf. 
\item \label{it: global (ii)}
  The set $\{\rmG(\bfb) \ | \ \bfb\in\hBi\}$ forms a $\Z[q^{1/2}]$-basis of $\LuphA$, and
 a $\Z$-basis of $\LuphA\cap c\big(\LuphA\big)$. 
\item \label{it: global (iii)}
For any $\bfb \in\hBi$, we have 
$$
\rmP(\bfb) = \rmG(\bfb) + \sum_{\bfb'  \boldsymbol{\prec}^* \bfb}  f_{\bfb,\bfb'}(q) \rmG(\bfb')\qquad \text{ for some $f_{\bfb,\bfb'}(q) \in q \Z[q]$.}
$$
 \ee
\end{theorem}

We call 
$$\text{$\bfG \seteq \{\rmG(\bfb) \ | \ \bfb\in\hBi\}$ the \emph{global basis} of $\hcalA$.}$$ 

For each $u \in \Z$, we set 
$$\bfG[u] \seteq \{\rmG(\bfb) \ | \ \bfb=(b_k)_{k \in \Z}\in\hBi  \text{ with } b_k =\mathsf{1} \text{ for } k \ne u \}.$$ 
Obviously, $\bfG[u]$ is a $\Z[q^{1/2}]$-basis of $\LuphA[u]$.

\subsection{Braid symmetries,
  Noetherian property of $\hcalA(\ttb)$ and strong duality data} \label{subsec: Braid symmetry}

\begin{proposition} [{\cite{KKOP21B} (see also \cite{JLO2,KKOP24B})}] For each $\im \in \sfI$, there exist $\kk$-algebra
automorphisms $\TT_\im$ and $\TT^{\star}_\im$ on $\hcalA$ defined as follows$:$
\begin{align}
\TT_\im(f_{\jm,m}) & = \bc
f_{\im,p+\delta_{\im,\jm}} & \text{ if } d(\im,\jm) \ne 1, \\
\dfrac{ q^{1/2} f_{\jm,m}f_{\im,m} - q^{-1/2}f_{\im,m}f_{\jm,m}  }{q - q^{-1}},  & \text{ if } d(\im,\jm) =1,
\ec \label{eq: T_i}\allowdisplaybreaks\\ 
\TT^{\star}_\im(f_{\jm,m}) & = \bc
f_{\im,p-\delta_{\im,\jm}} & \text{ if } d(\im,\jm) \ne 1, \\
\dfrac{ q^{1/2} f_{\im,m}f_{\jm,m} - q^{-1/2}f_{\jm,m}f_{\im,m}  }{q - q^{-1}},& \text{ if } d(\im,\jm) =1.
\ec \label{eq: T_i inverse}
\end{align} 
Furthermore, $\{\TT_\im\}_{\im \in \sfI}$ $($resp. $\{\TT^{\star}_\im\}_{\im \in \sfI})$ satisfies the commutation relations and the braid relations of $\sfg$ and
$\TT^{\star}_\im \circ \TT_\im = \TT_\im \circ \TT^{\star}_\im = {\rm id}$. 
\end{proposition}

From the above proposition, for each $\ttb \in \ttB$ with $\ttb= \bg_{\im_1}^{\ep_1}\bg_{\im_2}^{\ep_2} \cdots
\bg_{\im_r}^{\ep_r}$, 
$$ \TT_\ttb \seteq \TT^{\ep_1}_{\im_1}\TT^{\ep_2}_{\im_2} \cdots \TT^{\ep_r}_{\im_r} \text{ is well-defined}.$$
Note that, for any homogeneous element $x$, we have $\wt(\TT_\im(x)) = s_\im \wt(x)$.

\begin{proposition}[{\cite[Proposition 4.3, Lemma 4.4]{OP24}}]\label{prop: T_i in CQ}
Let $\uii=(\im_1,\ldots,\im_r)$ be a reduced sequence. Then, for any $1 \le k \le r$ and $m \in \Z$, we have
$$
\TT_{\im_1}\cdots \TT_{\im_{k-1}}(f_{\im_k,m}) \in \hcalA[m] 
$$
Furthermore, if  $\usfwc=(\im_1,\ldots,\im_\ell)$ is a reduced sequence of $w_0$, we have
$$
\TT_{\im_1}\cdots \TT_{\im_{\ell}}(f_{\im,m})  = f_{\im^*,m+1}. 
$$
\end{proposition}

\smallskip

Let $\ttb$ be an element in $\ttB^+$. For $\uii=(\im_1,\ldots,\im_r) \in \Seq(\ttb)$, we set 
\begin{align} \label{eq: PBW vector for hA}
\FF_k^\uii \seteq \TT_{\im_1}\cdots \TT_{\im_{k-1}}(f_{\im_k,0}) \quad \text{ for } 1\le k \le r.     
\end{align}

Let $\hcalA^\uii$ be a subalgebra of $\hcalA$ generated by $\{ \FF_k^\uii \}_{1 \le k \le r}$. 
When there is no danger of confusion, we drop $\uii$ in $\FF_k^\uii$.

For $\bsa =(a_1,\ldots,a_r) \in \Z_{\ge 0}^r$, we set   
\begin{align} \label{eq: PBW monomial hA}
\FF^\uii(\bsa) \seteq  \displaystyle   \oprod_{k\in [1,r]}   q^{a_k(a_k-1)/2 } \FF_k^{a_k}.    
\end{align}

\begin{theorem} [\cite{OP24,KKOP24B}] \label{thm: hatA(b)}
For any $\ttb\in \ttB^+$, we set $\hcalA(\ttb)=\hcalA_{\ge0}\cap \TT_\ttb(\hcalA_{<0})$.
Then,  
\begin{align} \label{eq: PBW hcal(b)}
\text{$\bfP_\uii \seteq \{ \FF^\uii(\bsa)  \ | \ \bsa \in \Z_{\ge 0}^{\ell(\ttb)} \}$ forms a basis of $\hcalA(\ttb)$ for any $\uii \in \Seq(\ttb)$.}
\end{align}
\end{theorem}

We call $\bfP_\uii$ in~\eqref{eq: PBW hcal(b)} the \emph{PBW-basis} of $\hcalA(\ttb)$ associated with $\uii \in \Seq(\ttb)$.

\begin{theorem} [\cite{KKOP24,KKOP24B}] \label{thm: P and G for ttb}
  \hfill
\bnum
\item
  $\TT_\im$ induces an $\bfA$-algebra 
automorphism of $\hcalA_\bfA$ and the global basis $\bfG$ of $\hcalA$ is invariant under this automorphism.
\item \label{it: P and G 2nd}  The global basis $\bfG$ is compatible with 
  $\hcalA(\ttb)$. Namely, $\bfG(\ttb) \seteq \bfG \cap \hcalA(\ttb)$ is a $\bfA$-basis of the $\bfA$-module $\hcalA_\bfA(\ttb) \seteq \hcalA(\ttb) \cap \hcalA_\bfA$.
\item \label{it: P and G 3rd} For each $\uii \in \Seq(\ttb)$, $\bfP_\uii$ is indeed a $\bfA$-basis of $\hcalA_\bfA(\ttb)$
and there exists a uni-triangular transition map between $\bfA$-bases $\bfP_\uii$ and $\bfG(\ttb):$
\begin{align} \label{eq: uni map}
 \FF^\uii(\bsa) =  \bb^\uii(\bsa) +  \sum_{\bsb\prec\bsa} c_{\bsa,\bsb}(q)\bb^\uii(\bsb) \quad \text{ for $c_{\bsa,\bsb}(q) \in q\Z_{\ge0}[q]$},
\end{align}
where $\bb^\uii(\bsa),\bb^\uii(\bsb) \in \bfG$ and $\prec$ is the bi-lexicographic order \ro see {\rm Definition~\ref{def: bi-lexico}} below\rf. 
\ee
\end{theorem}

\begin{remark}\label{rmk: comp PBW hwc}
    Recall $\hwc$ in~\eqref{eq: extension of reduced}.
  We extend the definition of $\FF_k^\hwc$ for $1 \le k \le \ell$ in~\eqref{eq: PBW vector for hA} by
$$
\FF_{k+n\ell}^\hwc \seteq \ocalD^{n}_q(\FF_k^\hwc)   \quad \text{ for $n \in \Z$.} 
$$
Then we have the followings:
\bna
\item $\FF_{k}^\hwc$ coincides with $\FF_{k}^\uii$ in \eqref{eq: PBW vector for hA} with $\uii=\hwc$. 
\item \label{it: PBW generators} The set $\{ \FF_{k}^\hwc \ | \  k \in \Z \}$ generates $\hAz$ as a $\Zq$-algebra.
\item The set $\bfP_\hwc \seteq \{ \FF^\hwc(\bsa)  \ | \ \bsa \in \Z_{\ge 0}^{ \oplus \Z } \}$ forms a $\Zq$-basis of $\hcalA_\Zq$.
\item The set $\bfP_\hwc[m] \seteq  \bigl\{ \FF^\hwc(\bsa)
   \bigm| \ \bsa \in \Z_{\ge0}^{[m\ell+1 ,(m+1)\ell]}\subset\Z_{\ge0}^{\oplus\Z} \bigr\}$  
  coincides with $\bfP_\hwc \cap\hcalA[m]_\Zq $ and forms a $\Zq$-basis of $\hcalA[m]_\Zq$. 
\item
There exists a unique family $\st{\bb^\hwc(\bsa)}_{\bsa\in\Z_{\ge0}^{\oplus\Z}}$
  of elements in $\bfG$ such that
  $$\bb^\hwc(\bsa)\equiv\FF^\hwc(\bsa)\mod\sum_{\bsa'\prec\bsa}q\Z[q] \FF^\hwc(\bsa').$$
  \ee
\end{remark}

Let us define  
$$
\obbA(\ttb) \seteq \hcalA_\bfA(\ttb)/(q^{1/2}-1)\hcalA_\bfA(\ttb)\subset \obbA.
$$
Then Proposition~\ref{prop: comm},  Theorem~\ref{thm: hatA(b)} and Theorem~\ref{thm: P and G for ttb} say that $\obbA(\ttb)$ is also a commutative ring.

The following lemma immediately follows from Theorem~\ref{thm: hatA(b)}.
\Lemma\label{lem:obbAcom}
Let $\ttb\in\ttB^+$ and $\uii=(\im_1, \ldots, \im_r)\in\Seq(\ttb)$. Then the 
commutative $\Z$-algebra $\obbA(\ttb)$
is the polynomial algebra generated by
$\st{\ev_{q=1}(\FF_k^\uii)\mid k\in[1,r]}$.
\enlemma

\smallskip

For a while, we shall prove that the algebra $\hcalA_\bfA(\ttb)$ is a Noetherian domain. In order to do that, we need a preparation.

\begin{proposition}
Let $B$ be rings and  $A \subset B$  its subring and $x \in B$. Assume that
\bnum
\item $A$ is left $($resp.\ right$)$ Noetherian,
\item $Ax + A = xA +A$,
\item $B = \sum_{k \in \Z_{\ge 0}} A x^k$.
\ee
Then $B$ is left $($resp.\ right$)$ Noetherian. 
\end{proposition}

\begin{proof} 
Since the proof for right Noetherian is similar to the one for left Noetherian, we only give the proof for left Noetherian.
For $n \in \Z_{\ge 0}$, we set $B_n\seteq \sum_{k \le n} A x^k = \sum_{k \le n} x^k A$.
Let $\calI \subset B$ be a left ideal. Let us show that $\calI$ is finitely generated.

For $n\in\Z_{>0}$, set 
$$
\frakA_n = \{a \in A \ | \  x^n a \in \calI + B_{n-1}  \}.
$$
We claim that $\frakA_n$ is a left ideal of $A$. For $a \in \frakA_n$, we have
\begin{align*}
  x^n A a \subset B_n a \subset (Ax^n + B_{n-1}) a \subset A(\calI + B_{n-1}) +
  B_{n-1}a \subset \calI+B_{n-1}, 
\end{align*}
which implies the claim. 

Note that $\{  \frakA_n \}_{n \in \Z_{>0}}$ is increasing. Hence there exists $n_0 \in \Z_{> 0}$ such that
$\frakA_n = \frakA_{n_0}$ for all $n \ge n_0$. 

Since $\frakA_{n_0}$ is finitely generated, we can write
as $\frakA_{n_0} = \sum_{j=1}^r A a_j$ for some $a_j\in \frakA_{n_0}$.
We write
$$x^{n_0}a_j = q_j + p_j \quad \text{ with } q_j \in \calI \text{ and } p_j \in B_{n_0-1}.$$
Then for $n \ge n_0$, we have
\begin{align*}
\calI \cap B_n & \subset x^n \frakA_n + B_{n-1} \\
& \subset \sum_{j=1}^r A x^n a_j + B_{n-1} 
 = \sum_{j=1}^r A x^{n-n_0} (q_j+p_j) + B_{n-1} \\
& \subset  \sum_{j=1}^r A x^{n-n_0} q_j+ B_{n-1} 
 \subset  \sum_{j=1}^r B q_j+ B_{n-1},
\end{align*}
which implies
$$
\calI \cap B_n \subset \sum_{j=1}^r B q_j+ \calI \cap B_{n-1}
\ \
\text{ and hence } \ \ 
\calI \subset \sum_{j=1}^r B q_j+ \calI \cap B_{n_0-1}.
$$
Since $\calI \cap B_{n_0-1}$ is finitely generated as a left $A$-module,
 we can conclude that
$$
\calI = \sum_{j=1}^r B q_j + B(\calI \cap B_{n_0-1} )
$$
is finitely generated as a $B$-module, which implies the assertion. 
\end{proof}
Recall that a ring $A$ is called a domain if $ab\not=0$
for any non-zero $a,b\in A$. 
\begin{proposition} \label{prop: Noetherian}
For $\ttb \in \ttB^+$, $\hcalA_\bfA(\ttb)$ is a Noetherian domain.
\end{proposition}

\begin{proof} 
  First note that  $\hcalA_\bfA(\ttb)$ is a free $\Z[q^{\pm1/2}]$-module and $\obbA(\ttb)=\hcalA_\bfA(\ttb)/(q^{1/2}-1)\hcalA_\bfA(\ttb)$ is a polynomial ring
(Lemma~\ref{lem:obbAcom}). It follows that $\hcalA_\bfA(\ttb)$ is a domain.

We prove that $\hcalA_\bfA(\ttb)$ is Noetherian
by induction on $\ell(\ttb)$.
  Let us write
$\ttb = \sigma_{\im_1}\ldots \sigma_{\im_r}$ for $\uii =(\im_1,\ldots,\im_r) \in \Seq(\ttb)$, $\ttb' = \sigma_{\im_1}\ldots \sigma_{\im_{r-1}}$ and
$$
B \seteq \hcalA_\bfA(\ttb) \ \ \text{ and }  \ \  A \seteq \hcalA_\bfA(\ttb')
\quad \text{which obviously satisfy $A \subset B$.}
$$
Then $A$ is Noetherian by the induction hypothesis. Set $x \seteq \FF^\uii_{r} = \TT_{\im_1} \cdots \TT_{\im_{r-1}} f_{\im_r,0}$. Note that 
\eqn
[A,x]_q&&\subset [\TT_{\ttb'}\hA_{<0},\TT_{\ttb'}f_{\im_r,0}]_q
=\TT_{\ttb'}\bl[\hA_{<0},f_{\im_r,0}]_q\br=
\TT_{\ttb'}\bl\rmE_{\im_r,-1}(\hA_{<0})\br\underset{*}{\subset} \TT_{\ttb'}\hA_{<0}\qtq\\
{}[A,x]_q&&\subset [\hA_{\ge0},\hA_{\ge0}]_q\subset\hA_{\ge0},
\eneqn
where $\underset{*}{\subset}$ follows from Theorem~\ref{thm: hAform}\;\eqref{it Est}.
Hence we obtain
$$[A,x]_q\subset A,$$
which implies $xA+A = Ax +A$.
Since
\eqn
A &&= \sum_{n_j \in \Z_{\ge 0}} \bfA (\FF^\uii_{r-1})^{n_{r-1}} \cdots (\FF^\uii_{1})^{n_{1}}\qtq\\
B &&= \sum_{n_j \in \Z_{\ge 0}} \bfA x^{n_r}(\FF^\uii_{r-1})^{n_{r-1}} \cdots (\FF^\uii_{1})^{n_{1}},
\eneqn
we have $B = \sum_{n \in \Z_{\ge 0}}x^nA$. Hence
the assertion follows from the previous proposition.  
\end{proof}

Note that
$\hAmz{m} \simeq \calA_\Zq(\sfn)$ for any $m\in \Z$. 
 
\begin{proposition} \label{prop: homomorphism from hA to Cg}
Let $\bbD = \{ L_\im \}_{\im \in \sfI}$ be a strong duality datum in $\Cgz$.
Then there exists a unique $\Z$-algebra  homomorphism 
\begin{align} \label{eq: bPhi}
\bPhi_\bbD \col\hAz \Lto K(\Cgz)    
\end{align}
satisfying the followings$\colon$
\bna
\item The homomorphism induced by the Schur-Weyl functor $\calF_\bbD$
$$
\hAmz{0} \isoto \calK(R\gmod) \twoheadrightarrow K(\scrC_\bbD) \hookrightarrow  K(\Cgz)
$$
coincides with $\bPhi_\bbD|_{\hAmz{0}}$.
\item \label{it: commutes with D} $\bPhi_\bbD \circ \calD_q =[\scrD] \circ \bPhi_\bbD$, where $[\scrD]$ denotes 
the automorphism of $K(\Cg^0)$ induced by $\scrD$.
\ee
\end{proposition} 

\begin{proof} Note that $\obbA$ and $K(\Cgz)$ are commutative algebras.
Since $$\obbA = \hAz/(q^{1/2}-1)\hAz \simeq \underset{m \in \Z}{\overset{\Lto}{\stens}} \hAmz{m}/(q^{1/2}-1)\hAmz{m},$$
it is enough to construct a homomorphism
$$\bPhi_\bbD[m]\col \hAmz{m} \to K(\scrD^m(\scrC_\bbD)) \subset K(\Cgz).$$ 

For $m=0$, we set $\bPhi_\bbD[0]\col \hAmz{0} \to K(\scrC_\bbD)$ induced from the functor $\calF_\bbD$ in~\eqref{eq: FD} yielding
the isomorphism 
$$ \obPhi_\bbD[0] \col   \obbK (R\gmod) \simeq \hAmz{0}/(q^{1/2}-1)\hAmz{0} \isoto K(\scrC_\bbD)$$
in~\eqref{eq: [FD] iso}. 
For a general $m$, we define $\bPhi_\bbD[m]$ by the commutative diagram:
$$\xymatrix@C=8ex{
   \hAmz{0}\ar[r]^-{\bPhi[0]}\ar[d]_{\ocalD_q^m}^\bwr& K(\scrC_\bbD)\ar[d]_{\D^m}^\bwr\\
 \hAmz{m}\ar[r]^-{\bPhi[m]}& K(\scrD^m(\scrC_\bbD)).
  }$$  
Hence we obtain a $\Z$-algebra homomorphism 
\eqn
\xymatrix@R=1ex@C=9ex{  
\hAz  \ar[dr]_{\ev_{q=1}} \ar[rr]^{\bPhi_\bbD} && K(\Cgz) \\
  & \obbA \ar[ur]_{\obPhi_\bbD}
 } 
\eneqn
with the desired properties.
\end{proof}

\begin{theorem} \label{thm: obbA K(Cgz)}
If $\bbD$ is a complete duality datum, then  $\bPhi_\bbD$ induces an isomorphism 
\begin{align} \label{eq: obPhi}
  \obPhi_\bbD \colon \obbA \isoto K(\Cgz).    
\end{align} 
\end{theorem}

\begin{proof}
It follows from the isomorphism $K(\scrC_\bbD)^{\otimes \Z} \isoto K(\Cgz)$ in~\eqref{eq: iso Kz}. 
\end{proof}

For $\im \in \sfI$,
let us take a reduced expression $\usfwc=(\im_1,\im_2,\ldots,\im_\ell)$ of $w_0$ with $\im_1=\im$ and consider its extension $\hwc$ in~\eqref{eq: extension of reduced}. Note that $\usfwcp=(\im_2,\ldots,\im_\ell,\im_1^*)$ is also a reduced expression of $w_0$.
Recall the cuspidal module $\sfV^\usfwc_k$ for $1 \le k \le \ell$. 
For a complete duality datum $\bbD$, define 
\begin{align} \label{eq: affine cuspidal module}
C_k^{\bbD,\hwc} \seteq \calF_\bbD(\sfV^\usfwc_k) \quad \text{ for $1\le k \le \ell$,}
\qtq 
C_{k + n\ell}^{\bbD,\hwc}   \seteq \scrD^{n} C_{k}^{\bbD,\hwc} \qquad\text{for $n \in \Z$.}
\end{align}%
We call $(\Dpair)$ a \emph{PBW-pair} and  $C^\Dpair_{m}$ $(m\in \Z)$ the \emph{affine cuspidal module} associated with $(\Dpair)$.

Using the homomorphism $\bPhi_\bbD$, \cite[Proposition 5.10]{KKOP23P} can be expressed as follows:

\begin{proposition} [{\cite[Proposition 5.10]{KKOP23P}}] \label{prop: reflection on duality datum}
Let $\bbD$ be a complete duality datum and set $\bbD' =\scrS_\im \bbD$. Then we have
$$
\bPhi_{\bbD'}( \FF^{\hwcp}_k) = [C^{\bbD',\hwcp}_{k}] = [C^{\bbD,\hwc}_{k+1}] \quad \text{ for $k \in \Z$}. 
$$    
\end{proposition}

\begin{proposition} \label{prop: comm diagram}
For a strong duality datum $\bbD$ and $\im \in \sfI$, we have 
the following commutative diagram:
\begin{align*} 
\raisebox{4.5ex}{\xymatrix@R=1ex@C=7ex{  
  \hcalA_\bfA \ar[dd]^{\TT_\im^{\pm1}} \ar[dr]^{\bPhi_{\scrS_\im^{\pm1}\bbD} }  \\
 & K(\Cgz) \\
 \hcalA_\bfA \ar[ur]_{\bPhi_\bbD}
 } }
\end{align*}
\end{proposition}

\begin{proof}  
Note that
$$
\TT_\im(\FF_k^\hwcp) = \TT_{\im_1}\TT_{\im_2}\ldots \TT_{\im_k}(f_{\im_{k+1},0}) = \FF^\hwc_{k+1} \quad \text{ for $1 \le k <\ell$.}
$$
Thus we have
$$
\bPhi_\bbD( \TT_\im(\FF_k^\hwcp)) = \bPhi_\bbD(\FF_{k+1}^\hwc) = [C^{\bbD,\hwc}_{k+1}]  \quad \text{ for $1 \le k <\ell$.}
$$
For $k =\ell$, we have
\begin{align*}
\TT_\im(\FF_\ell^\hwcp)  =    \TT_{\im_1}\TT_{\im_2}\ldots \TT_{\im_{\ell}}(f_{\im^*,0})  = f_{\im,1} = \ocalD_q(f_{\im,0}) = \FF_{\ell+1}^\hwc
\end{align*}
so that
$$
\bPhi_\bbD( \TT_\im(\FF_\ell^\hwcp)) = \bPhi_\bbD(\ocalD_q(f_{\im,0}) ) = [\scrD C^{\bbD,\hwc}_{1}] = [C^{\bbD,\hwc}_{\ell+1}].
$$
By Proposition~\ref{prop: homomorphism from hA to Cg}~\eqref{it: commutes with D}, we can conclude that
$$
\bPhi_\bbD( \TT_\im(\FF_k^\hwcp)) = [C^{\bbD,\hwc}_{k+1}] = \bPhi_{\bbD'}(\FF_k^\hwcp) \quad \text{ for all $k \in \Z$.}
$$
Then our assertion for $\TT_\im$ follows from the fact that  $\{ \FF_{k}^\hwcp \ | \  k \in \Z \}$ generates $\hAz$. 
The assertion for $\TT^{-1}_\im$ can be obtained in a similar way. 
\end{proof}

\begin{corollary} \label{cor: braid for scrS}
  The family of operators $\st{\scrS_\im}_{\im\in\sfI}$ acting on the set of
 \ro the isomorphism classes of\/\rf complete duality data  satisfies the  commutation relations and the braid relations.  
\end{corollary}
\Proof
Let us show that
$\scrS_\im\scrS_\jm\scrS_\im\bbD=\scrS_\jm\scrS_\im\scrS_\jm\bbD$ if
$d(\im,\jm)=1$.
Then we have
\eqn
[L^{\scrS_\im\scrS_\jm\scrS_\im\bbD}_k]
&&=\bPhi_{\scrS_\im\scrS_\jm\scrS_\im\bbD}(f_{k,0})\\
&&=\bPhi_{\scrS_\jm\scrS_\im\bbD}\TT_{\im}(f_{k,0})\\
&&=\bPhi_{\scrS_\im\bbD}\TT_{\jm}\TT_{\im}(f_{k,0})\\
&&=\bPhi_{\bbD}\TT_{\im}\TT_{\jm}\TT_{\im}(f_{k,0}).
\eneqn
Hence we have
$$[L^{\scrS_\im\scrS_\jm\scrS_\im\bbD}_k]=[L^{\scrS_\jm\scrS_\im\scrS_\jm\bbD}_k].$$

We can prove similarly that
$\scrS_\im\scrS_\jm\bbD=\scrS_\jm\scrS_\im\bbD$ if
$d(\im,\jm)>1$.
\QED

By Corollary~\ref{cor: braid for scrS},
the braid group $\braid$ acts on the set of
complete duality data.
In particular $\scrS_{\ttb} \bbD$ is well-defined
for $\ttb \in \ttB$ and a complete duality datum $\bbD$:
$$
\scrS_{\ttb} \bbD \seteq \scrS_{\im_1}^{\ep_1}\cdots  \scrS_{\im_r}^{\ep_r} \bbD
$$
where $\ttb=\bg_{\im_1}^{\ep_1}\cdots  \bg_{\im_r}^{\ep_r}$.

\smallskip
Recall that $\rev$ is the anti-automorphism of the group $\ttB$
which sends $\sigma_\im$ to itself.
\begin{corollary} \label{cor: dia comm repeat}
For any $\ttb \in \ttB$, we have the following commutative diagram 
\begin{align} \label{eq: T_i S_i}
\raisebox{5.5ex}{\xymatrix@R=1ex@C=7ex{  
  \hcalA_\bfA \ar[dd]^{\TT_\ttb} \ar[dr]^{\bPhi_{\bbD'} }  \\
 & K(\Cgz) \\
 \hcalA_\bfA \ar[ur]_{\bPhi_\bbD}
 } } \qquad \text{where $\bbD' =  \scrS_{\ttb^\rev} \bbD$.}
\end{align}
Namely, we have
\eq
\bPhi_{\scrS_\ttb\bbD}=\bPhi_\bbD\circ \TT_{\ttb^\rev}.
\label{eq:TTPhi}
\eneq
\end{corollary}

\begin{proof}
It is enough to show that if \eqref{eq:TTPhi} holds for $\ttb_1$ and $\ttb_2$, then it holds for $\ttb_1\ttb_2$.  We have
\eqn
\bPhi_{\scrS_{\ttb_1\ttb_2}\bbD}
=\bPhi_{\scrS_{\ttb_1}\scrS_{\ttb_2}\bbD}
=\bPhi_{\scrS_{\ttb_2}\bbD}\circ\TT_{\ttb_1^\rev}
=\bPhi_\bbD\circ \TT_{\ttb_2^\rev}\circ\TT_{\ttb_1^\rev}
=\bPhi_\bbD\circ \TT_{(\ttb_1\ttb_2)^\rev}. \qedhere
\eneqn
\QED

\subsection{Quantizability and Categorifiability} 
In this subsection, we fix a complete duality datum $\bbD$ in $\Cgz$.
Recall that we denote by $\bfG$ the global basis of $\hcalA$.

\begin{definition} \label{def: strong real} 
Let $\bbD$ be a complete duality datum in $\Cgz$.
\bnum
\item A simple module $M \in \Cgz$ is \emph{$\bbD$-quantizable} 
if there exists $\bfx \in \bfG$ such that $\bPhi_\bbD(\bfx)=[M]$. In this case, we write
$$
\ch_\bbD(M)=\bfx.
$$
\item An element $\tbfx \in q^{\Z/2}\bfG$ is \emph{$\bbD$-categorifiable}
  if there exists a simple module $M \in \Cgz$ such that $\bPhi_\bbD(\tbfx)=[M]$.
\item  
  Let $\ttb \in \ttB$ and let $M$ be a $\bbD$-quantizable simple module in $\Cgo$ with $\ch_\bbD(M)=\bfx \in \bfG$.
If $\TT_\ttb(\bfx)$ is $\bbD$-categorifiable,
then we say that $\TT^\bbD_\ttb (M)$ is {\em $\bbD$-definable} and set 
$$
\TT^\bbD_\ttb(M) \seteq N,
$$
where $N\in \Cgz$ is given by $\bPhi_\bbD\bl\TT_\ttb(\bfx)\br=[N]$.
If there is no danger of confusion, we write simply
$\TT_\ttb(M)$ for $\TT^\bbD_\ttb(M)$.
\ee
\end{definition}

\Lemma \label{lem: any sequence}
For any $\ttb \in \ttB$, $(\im,m) \in \sfI \times \Z$ and
a positive integer $n$, 
$\TT_\ttb f_{\im,m}^n$ is $\bbD$-categorifiable.
\enlemma

\begin{proof}
Set  $\bbD' \seteq \scrS_{\ttb^\rev} \bbD = \{ L'_\im \}_{\im \in \sfI}$. By~\eqref{eq: T_i S_i}, we have
\[\bPhi_\bbD(\TT_\ttb f_{\im,m}^n)=\bPhi_{\bbD'}(f_{\im,m}^n)=[(\scrD^m L'_\im)^{\otimes n}]. \qedhere\]
\end{proof}

\Prop\label{prop: Gm dequatiazable} Let $u \in \Z$.
\bnum    
\item \label{it: Gm i} Any element in $\bfG[u]$  is $\bbD$-categorifiable.
\item \label{it: Gm ii} For any $\ttb \in\ttB$ and $\bfx \in \bfG[u]$, $\TT_\ttb(\bfx)$
is $\bbD$-categorifiable.
\ee
\enprop

\begin{proof}
\eqref{it: Gm i} follows from Theorem~\ref{thm: KLR iso} and Theorem~\ref{thm:gQASW duality}. 

\noindent
\eqref{it: Gm ii}
By \eqref{eq:TTPhi}, we have $\bPhi_\bbD\bl\TT_\ttb(\bfx)\br=
\bPhi_{\scrS_{\ttb^\rev}\bbD}(\bfx)$, which is
represented by a simple module in $\Cgo$ by (i).
\end{proof}
 
\Lemma
Let $M$ be a simple module in $\scrD^u(\scrC_\bbD)$. Then there exists a simple $X\in R\gmod$ such that
$M\simeq\D ^u\calF_\bbD(X)$.
\enlemma
\Proof
Take $\bfx\in\bfG[u]$ such that $\bPhi_\bbD(\bfx)=[M]$.
Then there exists a simple $X\in R\gmod$ such that
$\calD_q^{-u}\bfx=[X]\in\hcalA[0]$.
Then we have
$\calF_\bbD(X)\simeq\D^{-u}M$.
\QED

\Lemma\label{lem:Treal}
Let $M$ be a simple module in $\scrD^u(\scrC_\bbD)$. Then $\TT^{\hspace{0.1em}\bbD}_\ttb (M)$ is $\bbD$-definable.
Moreover if $M$ is real, then $\TT^{\hspace{0.1em}\bbD}_\ttb (M)$ is real.
\enlemma
\Proof
Let $X\in R\gmod$ be a simple module such tat $M\simeq\D^u\calF_\bbD(X)$. Then 
$\TT^\bbD_\ttb (M)\simeq\D^u\calF_{\scrS_{\ttb^\rev}\bbD}(X)$ is $\bbD$-definable.
If $M$ is real, then $X$ is real and hence $\TT^\bbD_\ttb (M)$ is real.
\QED

Based on \cite[Theorem 4.12]{KKOP23P}, we obtain the following proposition:

\Prop\label{prop: ell de} Let $\ttb\in\ttB$ and $u \in \Z$.
Let  $M$ and $N$ be simple modules in $\scrD^u(\scrC_\bbD)$ such that one of them is real. Then
\bnum
\item $M\hconv N \in \scrD^u(\scrC_\bbD)$ and $\TT_\ttb(M\hconv N)\simeq(\TT_\ttb M)\hconv(\TT_\ttb N)$,
\item $\La(\TT_\ttb M,\TT_\ttb N)=\La(M,N)$ and $\de(\TT_\ttb M,\TT_\ttb N)=\de(M,N)$,
\item $\de(\D ^kM,N)=0$ if $|k|>1$,
\item $\de(\D ^k\TT_\ttb M,\TT_\ttb N)=\de(\D ^kM,N)$ for any $k\in\Z$.
\ee 
\enprop

\begin{proof}
 We may assume that $u=0$. There exist simple $X,Y\in R\gmod$ such that $M\simeq\calF_\bbD(X)$ and   $N\simeq\calF_\bbD(Y)$.
Then one of $X$ and $Y$ is real and we have $M\hconv N\simeq\calF_\bbD(X\hconv Y)$.
Hence we have $M\tens N \in \scrC_\bbD$ and
$$\TT_\ttb(M\hconv N)\simeq\calF_{\scrS_{\ttb^\rev}\bbD}(X\hconv Y)
\simeq\calF_{\scrS_{\ttb^\rev}\bbD}(X)\hconv\calF_{\scrS_{\ttb^\rev}\bbD}(Y)\simeq \TT_\ttb(M)
\hconv\TT_\ttb (N).$$
Moreover, we have (iii),
$\La(M,N)=\La(X,Y)=\La(\TT_\ttb M,\TT_\ttb N)$ and
$\de(\D M,N)=\tLa(X,Y)=\de(\D\TT_\ttb M,\TT_\ttb N)$.
\end{proof}

\begin{lemma} \label{lem: preserving root}
  Let  $L \in \scrD^u(\scrC_\bbD)$  be a root module
  for some $u \in \Z$ and $\ttb\in \ttB$. Then $\TT_\ttb  (L)$ is also a root module.
\end{lemma}

\begin{proof}
By Lemma~\ref{lem:Treal}, $\TT_\ttb (L)$ is real. Then the assertion follows from 
$$\de(\D ^k\TT_\ttb (L),\TT_\ttb (L))=\de(\D ^kL,L)$$
in Proposition~\ref{prop: ell de}.
\end{proof}

\section{Affine determinantial modules  and admissible chains of $i$-boxes} \label{Sec: adm ac}

In this section, we shall review the notions of affine cuspidal modules,  affine determinantial
modules, admissible chains of $i$-boxes and their properties
associated with a \emph{not} necessarily locally reduced sequence, which are
studied in \cite{KKOP23P,KKOP24A} mainly for locally reduced sequences.   
 
\smallskip

Throughout this section, we fix a complete duality datum $\bbD =  \{ L^\bbD_{\im} \}_{\im\in\sfI}$ associated with the simply laced root system of $\sfg$. 
We sometimes drop $^\bbD$ for simplicity of notation.

\subsection{Combinatorics of $i$-boxes}
In this subsection, we fix a sequence $\uii=(\im_k)_{k\in K}$ in $\sfI$ where $K$ is a possibly infinite interval in
$\Z$. For $k\in K$, we define
\eqn
&&k_\uii(\jm)^+ \seteq  \min(\{ t \in K\mid t\ge k,\; \im_t=\jm \} \sqcup \{ +\infty \}), \\
&&k_\uii(\jm)^- \seteq  \max(\{t \in K \mid t\le k,\;\im_t=\jm \} \sqcup \{-\infty\}), \\
&& k_\uii^+\seteq \min\bl\{ t \in K\mid t> k,\; \im_t= \im_k \}
\sqcup \st{+\infty}\br,\\
 && k_\uii^-\seteq \max(\{t \in K \mid t< k,\;\im_t=\im_k \} \sqcup \{-\infty\}). 
\eneqn
We will frequently drop $_\uii$ in the above notations for simplicity when there is no danger of confusion. 

\begin{definition}\label{def:ibox}
\hfill
\bnum
\item For an interval $[a,b] \subset K$ and $c \in K$, we set 
$$\uii_{[a,b]} \seteq (\im_k)_{k\in [a,b]}, \quad \uii_{\le c} \seteq \uii_{K\cap[-\infty,c]},\qtq \uii_{\ge c} \seteq \uii_{K\cap[c,+\infty]}.$$
\item We say that a finite interval $\frakc = [a,b]$ contained in $K$ is an \emph{$i$-box} if $ a \le b $ and $\im_a=\im_b$. We sometimes write
it as $[a,b]^{\uii}$ to emphasize that it is associated with $\uii$. 
\item For an $i$-box $[a,b]$, we set
$$ [a,b]_\phi \seteq  \{s \ | \  s \in [a,b] \text{ and } \im_a = \im_s =\im_b \}.$$
\item For a finite interval $[a,b]$ in $K$, we define the $i$-boxes
\begin{align} \label{eq: p i-box }
  [a,b \} \seteq [a,b(\im_a)^-] \qtq    \{a,b] \seteq [a(\im_b)^+,b]. 
\end{align}
\item We say that $i$-boxes $[a_1,b_1]$ and $[a_2,b_2]$ \emph{commute} if we have either
$$
a_1^- < a_2 \le b_2 < b_1^+ \qtoq a_2^- < a_1 \le b_1 < b_2^+.
$$

\item A chain $\frakC$ of $i$-boxes $(\frakc_k=[a_k,b_k])_{1 \le k \le l}$  for $l \in \Z_{>0} \sqcup \{ \infty \}$
is called \emph{admissible} if 
$$
\tfrakc_k  =[\ta_k,\tb_k] \seteq \sbcup_{1 \le j \le k} [a_j,b_j] \text{ is an interval with $|\tfrakc_k|=k$ for $k=1,\ldots,l$} 
$$
and either $[a_k,b_k]=[\ta_k,\tb_k \}$ or $[a_k,b_k]=\{ \ta_k,\tb_k ]$ for $k=1,\ldots,l$. 
\item The interval $\tfrakc_k$ is called the \emph{envelope} of $\frakc_k$, and $\tfrakc_l$ is the \emph{range} of $\frakC$. 
\ee
\end{definition}

\begin{lemma}[\cite{KKOP24A}] \label{lem: frack commute} 
Let $\frakC=(\frakc_k)_{1\le k\le l}$ be an admissible chain of $i$-boxes.
\bna
\item For all $1\le j,k\le l$, $\frakc_j$ and $\frakc_k$ commute.
\item If an $i$-box $\frakc\subset \tfrakc_l$
  commutes with all $\frakc_j$ $(1 \le j\le l)$,
  then $\frakc$ is a member of $\frakC$.
\ee
\end{lemma}

Note that the admissible chain $\frakC=(\frakc_k)_{1\le k\le l}$ is uniquely determined by its envelopes and \emph{horizontal moves} at steps:
\begin{align} \label{eq: horizontal moves}
 \frakc_{k}=[a_k,b_k] = \calH_{k-1}[\ta_k,\tb_k]  \seteq \bc
 [\ta_k,\tb_k\} & \text{ {\rm (i)} } \ta_k=\ta_{k-1}-1 \text{ and } \tb_k=\tb_{k-1}, \\
  \{ \ta_k,\tb_k] & \text{ {\rm (ii)} } \tb_k=\tb_{k-1}+1 \text{ and } \ta_k=\ta_{k-1},  
 \ec
\end{align}
for $1 < k \le l$. In case {\rm (i)} in~\eqref{eq: horizontal moves}, we write $\calH_{k-1}=\calL$, while $\calH_{k-1}=\calR$ in case {\rm (ii)} \footnote{In \cite{KKOP24A}, $T_{k-1}$ have used
instead of $\calH_{k-1}$.}. Hence, for each chain $\frakC$ of length $l$, we can associate a pair $(c,\frakH)$ consisting of 
\begin{align}
\text{$c =a_1=b_1$ and $\frakH=(\calH_1,\ldots,\calH_{l-1})$
such that $\calH_i \in \{ \calL,\calR\}$ $(1 \le i <l)$.}    
\end{align}

\begin{definition}   \label{def: movable and B-move}
Let $\frakC=(\frakc_k)_{1\le k\le l}$ be an admissible chain of $i$-boxes  associated with $(c,\frakH)$.
\bna
\item
  For $1 \le m <l$,
  we call the $i$-box $\frakc_m$ \emph{movable}
if  
$$
 m=1  \qtoq  \calH_{m-1} \ne \calH_m  \text{ for } m \ge 2.   
$$
Note that the latter condition is equivalent to $\ta_{m+1}=\ta_{m-1}-1$ and $\tb_{m+1}=\tb_{m-1}+1$.  
\item \label{it: B-move}
For a movable $i$-box $\frakc_m$ in $\frakC$, we define a new admissible chain of $i$-boxes $\bbB_m(\frakC)$ whose associated pair $(c',\frakH')$
is given as follows:
\bnum 
\item $\bc c'=c\pm1 & \text{ if $m=1$ and $\calH_1=\calR$ (resp. $\calL$)}, \\
      c'=c  & \text{ otherwise}, \ec$ 
\item $\calH_k'=\calH_k$ for $k\not\in\{m-1,m\}$ and $\calH_k'\ne \calH_k$ for $k\in\{m-1,m\}$. 
\ee
We call $\bbB_m(\frakC)$ the \emph{box move} of $\frakC$ at $m$. 
\ee
\end{definition}

\begin{proposition} [\cite{KKOP24A}] Let $\frakC=(\frakc_k)_{1\le k \le l}=(c,\frakH)$ be an admissible chain of $i$-boxes and $\frakc_m$ a movable $i$-box in $\frakC$. 
Set $\bbB_m(\frakC)=(\frakc'_k)_{1 \le k \le l}$.
\bna
\item Assume that $\tfrakc_{m+1}$ is \emph{not} an $i$-box. Then we have $\frakc'_k =\frakc_{\upsigma_m(k)}$ for all $1 \le k \le l$. 
\item Assume that $\tfrakc_{m+1}=[a,b]$ is an $i$-box. Then we have
$$
\bc
\frakc_{m} = [a^+,b] \text{ and }  \frakc'_{m}=[a,b^-] & \text{ if } \calH_{s-1}=\calR, \\
\frakc_{m} = [a,b^-] \text{ and }  \frakc'_{m}=[a^+,b] & \text{ if } \calH_{s-1}=\calL, 
\ec
$$
and $\frakc_k = \frakc'_k$ for all $k\in[1,l]\setminus\st{m}$. 
\ee
\end{proposition}

\Lemma[{\cite[Lemma 5.10]{KKOP24A}}]
\label{lem: finite sequence are T-equi} 
Let $\frakC$ be an admissible chain of $i$-boxes.
Then any admissible chain $\frakC'$ 
with the same range as $\frakC$ can be obtained from $\frakC$ by successive box moves.
\enlemma

\subsection{Affine determinantial modules} \label{Sec: adm}

Let $\uii=(\im_k)_{k\in K}$ be a sequence in $\sfI$ such that
$K$ is a possibly infinite interval \condi.
For $k\in K$ and 
  a strong duality datum $\bbD$ in $\Cgz$, we set
\begin{align} \label{eq: Ck}
  C^{\bbD,\uii}_k\seteq
  \bc
  \TT^\bbD_{\im_1}\cdots\TT^\bbD_{\im _{k-1}}L^{\bbD}_{\im_k}&\text{if $k>0$,}\\
  (\TT^\bbD_{\im_0}){}^{-1}\cdots(\TT^\bbD_{\im_{k}}){}^{-1}L^\bbD_{\im_k}&\text{if $k\le 0$.}
        \ec    
\end{align}
We have
$$C^{\bbD,\uii}_k\simeq (\TT^\bbD_{\im_l}\cdots\TT^\bbD_{\im _{0}})^{-1}
\TT^\bbD_{\im_l}\cdots\TT^\bbD_{\im _{k-1}}L^{\bbD}_{\im_k}$$
for any $l\in K$ such that $l\le 1, k$.
{}From Lemma~\ref{lem: any sequence} and Lemma~\ref{lem: preserving root}, 
for any sequence $\uii=(\im_1,\ldots,\im_r)$,
$C^{\bbD,\uii}_k$ is a well-defined root module. 

\smallskip
 
The theorem below is an interpretation of
results in \cite[\S 5]{KKOP23P} in terms of $\bbD$ and $\TT_\im$ $(\im \in \sfI)$. 

\begin{theorem} [\cite{KKOP23P}] \label{thm: affine cuspidal}
Let $\usfwc=(\im_1,\im_2,\ldots,\im_\ell)$ be a reduced sequence of $w_0$ of\; $\weyl$. 
\bnum
\item For each $k \in  \Z$, $C_k^{\bbD,\hwc}$ in \eqref{eq: affine cuspidal module}  coincides with the definition ~\eqref{eq: Ck}. 
\item Let $1 \le k \le \ell$. If $\be^\usfwc_k =\al_\jm$ for $\jm\in \sfI$, then $C_{k}^{\bbD,\hwc} \simeq L_\jm$.

\item For $1 \le k < m <l \le \ell$, if $(\be^\usfwc_k,\be^\usfwc_l)$
is a $\usfwc$-minimal pair of $\be^\usfwc_m$, then 
$C_{m}^{\bbD,\hwc} \simeq C_{k}^{\bbD,\hwc} \hconv C_{l}^{\bbD,\hwc}$.
\item The infinite sequence of root modules 
\begin{align}\label{eq: affine cuspidal module1}
\text{$\uC^\Dpair \seteq(\ldots,C^\Dpair_{1},C^\Dpair_{0}, C^\Dpair_{-1}, \ldots)$ is strongly unmixed. }
\end{align}
\ee
\end{theorem}
We frequently drop $^\Dpair$ in notations  if there is no danger of confusion. 

\smallskip

Since $\uC^{\Dpair}$ is strongly unmixed and hence normal by Proposition~\ref{prop: Unmix normal}~\eqref{it: unmix normal}, 
\begin{align} \label{eq: head of S}
\text{$\head(C^{\bsa})$
  is a simple module in $\Cgz$ for each $\bsa =(\ldots,a_1,a_0,a_{-1},\ldots) \in  \Z_{\ge 0}^{\oplus \Z}$,}
\end{align}
  where $C^{\bsa}\seteq\cdots \tens (C^\Dpair_1)^{\tens a_1}
  \tens (C^\Dpair_0)^{\tens a_0}\tens  (C^\Dpair_{-1})^{\tens a_{-1}} \tens  \cdots$.

\begin{theorem} [{\cite[Theorem 6.1]{KKOP23P}}] \label{thm: unique a}
For a PBW-pair $(\Dpair)$ and  any simple module $M \in \Cgz$, there exists a unique $\bsa \in \Z_{\ge 0}^{\oplus \Z}$ such that
$$
\head(C^{\bsa})  \simeq M.
$$
\end{theorem}

By Theorem~\ref{thm: affine cuspidal} and Theorem~\ref{thm: unique a},
any simple module $M \in \Cgz$ can be obtained by
a simple subquotient of a tensor product of root modules $\{ \rdual^k L^\bbD_\im \}_{\im \in \sfI,\, k \in \Z}$.  

\medskip

\Def \label{def:affdet}
\hfill
\bnum
\item
 For an $i$-box $[a,b]$, we define
\begin{align} \label{eq: MiD}
  M^{\bbD,\uii}[a,b] \seteq \head\left(\otens_{s\in[a,b]_\phi}C^{\bbD,\uii}_{s}\right) 
  =\head\left(C^{\bbD,\uii}_{b} \tens C^{\bbD,\uii}_{b^-} \tens
\cdots \tens C^{\bbD,\uii}_{a^+} \tens C^{\bbD,\uii}_a\right).
\end{align}
We call $M^{\bbD,\uii}[a,b]$ the \emph{affine determinantial module} associated with $(\bbD,\uii)$ and $[a,b]$.
\item 
For an interval $[a,b]\subset K$, we write $\Cg^{[a,b],\bbD,\uii}$ the smallest full subcategory of
$\Cgo$ which is stable by taking tensor products, subquotients, extensions and contains $\one$ and $C^{\bbD,\uii}_k$ for any $k\in[a,b]$ (see also \cite[\S 6.3]{KKOP23P}).
We write for any $m\in K$, $\Cg^{[m],\bbD,\uii}$, $\Cg^{\le m,\bbD,\uii}$ and $\Cg^{\ge m,\bbD,\uii }$ for $\Cg^{[m,m],\bbD,\uii}$, $\Cg^{K\cap[-\infty,m],\bbD,\uii}$ and  $\Cg^{K\cap[m,\infty],\bbD,\uii}$, respectively.  
\ee
\end{definition}
We frequently drop $^{\bbD,\uii}$ or $^{\bbD}$ in the above notations for simplicity  when there is no danger of confusion.

\Lemma\label{lem:lneg}
Let $\uii = (\im_m)_{k\in K}$ be a sequence in $\sfI$ \condi,
and $l\in K$ such that $l\le 1$. We 
set $\bbD'=\scrS_{l}^{-1}\cdots\scrS_{0}^{-1}\bbD$.
Let $\uii'=(\im'_{k})_{k\in K'}$ be the sequence defined by
$K'=K-l+1$ and $\im'_{k}=\im_{k+l-1}$ for $k\in K$.
Then we have
\begin{align*}
C^{\bbD,\uii}_{k} & =C^{\bbD',\uii'}_{k-l+1}  && \text{for $k\in K$ and } \\
  M^{\bbD,\uii}[a,b] & =M^{\bbD',\uii'}[a-l+1,b-l+1] && \text{ for any $i$-box
                                                        $[a,b]\subset K'$.}
\end{align*}
\enlemma
\Proof
For $k\in K$, let us take $l'\in\Z$ such that $l'\le l, k$.
By Corollary~\ref{cor: dia comm repeat}, we have
\eqn
[C^{\bbD,\uii}_k]
&&=\bPhi_{\bbD}\bl
(\TT_{\im_{l'}}\cdots\TT_{\im_{0}})^{-1} (\TT_{\im_{l'}}\cdots\TT_{\im_{k-1}})f_{\im_k,0}\br\\
&&=\bPhi_{\bbD'}\bl
(\TT_{\im_{ l'}}\cdots\TT_{\im_{l-1}})^{-1}
(\TT_{\im_{l'}}\cdots\TT_{\im_{k-1}})f_{\im_{k},0}\br\\
  &&=\bPhi_{\bbD'}\bl
(\TT_{(\im'_{l' -l+1}}\cdots\TT_{\im'_{0}})^{-1}
(\TT_{\im' _{l'-l+1}}\cdots\TT_{\im'_{k-l}})f_{\im'_{k-l+1},0}\br
=[C^{\bbD',\uii'}_{k-l+1}]. \qedhere     
\eneqn
\QED

\smallskip

The following theorem is one of the main results in \cite{KKOP24A}. 

\begin{theorem} \label{thm: deter main}
  The affine determinantial modules associated with $(\Dpair)$
  satisfy the following properties.
\bnum
\item For any $a \in\Z$ $\de(C_{a^+},C_a)=1$. 
\item $M[a,b]$ is a real simple module in $\Cgz$.
\item 
If two $i$-boxes $[a_1,b_1]$ and $[a_2,b_2]$ commute, then $M[a_1,b_1]$ and $M[a_2,b_2]$ commute. 
\item $\de(C_{a^-},M[a,b])=\de(C_{b^+},M[a,b])=\de(\ldual C_{a},M[a,b])=\de(\rdual C_{b},M[a,b])=1$.
\item \label{it: de=1} $\de(M[a^-,b^-],M[a,b]) = 1$. 
\item For any $i$-box $[a,b]$ such that $a<b$,  
  we have a short exact sequence in $\Cgz$
\begin{align}  \label{eq: T-system in terms of M[a,b]}
0 \to  \dtens_{ \substack{ \jm \in \hspace{0.1ex} \sfI; \\ d(\im_a,\jm)=1}} 
M[a(\jm)^+,b(\jm)^-]  \to   M[a^+,b]  \tens M[a,b^-]    \to   M[a,b]
\tens M[a^+,b^-] \to 0 
\end{align}
such that the left term and right term in~\eqref{eq: T-system in terms of M[a,b]} are simple. 
\ee
\end{theorem}
We call~\eqref{eq: T-system in terms of M[a,b]} a \emph{$T$-system}.

\begin{definition} Let $\uii =(\ldots,\im_{-1},\im_{0},\im_{1},\ldots) \in \sfI^\Z$.
  We define an anti-symmetric $\Z$-valued map $\la^{\uii}: \Z \times \Z \to \Z$  by 
\begin{align} \label{eq: anti-symmetric pairing original}
\la_{a,b}^{\uii} \seteq \bc
-  (s_{\im_b}s_{\im_{b+1}}\cdots s_{\im_{a-1}}(\al_{\im_a}), \al_{\im_b} ) & \text{ if } a>b, \\
  (\al_{\im_a}, s_{\im_a}s_{\im_{a+1}}\cdots s_{\im_{b-1}} (\al_{\im_b}) ) & \text{ if } a<b, \\
 \qquad  0 & \text{ if } a=b, 
\ec 
\quad \text{ for $a,b \in \Z$.}
\end{align}    
\end{definition}

Using the same argument in~\cite[\S 5.2]{KKOP23F} and \cite[Theorem 4.12]{KKOP23P}, we have the following (see also Proposition~\ref{prop: new pairing La} below):

\begin{proposition} 
Let $\usfwc$ be a reduced expression of $w_0$ and 
$[a_k,b_k]^{\hwc}$ $(k=1,2)$ $i$-boxes.    
If $a_1 > a_2^-$ or $b_1^+ > b_2$, then we have
$$\La(M^{\Dpair}[a_1,b_1],M^{\Dpair}[a_2,b_2]) = \displaystyle\sum_{u \in [a_1,b_1]_\phi,v\in [a_2,b_2]_\phi} \la^\hwc_{u,v}.$$
\end{proposition}

\section{Generalization of affine determinantial modules, $T$-systems and category $\Cg(\ttb)$} \label{sec: generalization}
 
The aim of this section is to prove that, for an \emph{arbitrary} sequence $\uii=(\im_k)_{k\in K}$ \condi,
the set of root modules $\st{ C^{\bbD,\uii}_{k}}_{k\in K}$ in \eqref{eq: Ck}
satisfies the same properties of $\st{ C^{\Dpair}_{k}}_{k\in K}$ in
\eqref{eq: affine cuspidal module}. 
We also introduce the subcategory $\Cg(\ttb)$ of $\Cgz$
for an element $\ttb \in \ttB^+$, standard modules associated with $\uii \in \Seq(\ttb)$
and prove the uni-triangularity between bases of $K(\Cg(\ttb))$, arising from standard modules and simple modules in   
$\Cg(\ttb)$.

\subsection{Garside normal form}\label{subsec:Gar}

Recall that $\ttB$ is the braid group and $\pi\col \ttB\to\weyl$ is the canonical group homomorphism. 
We define $\Updelta$ to be the element in $\ttB^+$ such that
$\ell(\Updelta)=\ell(w_0)$ and $\pi(\Updelta)=w_0$.

Remark that $\Updelta^2$ is contained in the center of $\ttB$. 

The following lemma easily follows from the fact that
$\sigma_\im^{-1}\Updelta^2=\Updelta^2 \sigma_\im^{-1}\in\ttB^+$
for any $\im\in\sfI$.
\Lemma [{see \cite[Corollary 7.3]{OP24}}] \label{lem: loc red braid}
For any $\ttx \in \ttB$, there exist $\tty \in \ttB^+$ and $m \in \Z_{\ge 0}$ such that $\ttx\tty =\Updelta^m$.     
\enlemma

For $\ttx,\ttz \in \ttB$, we write $\ttx \ledot \ttz$ if there exists $\tty \in \ttB^+$ such that $\ttx\tty =\ttz$, or equivalently $\ttx^{-1}\ttz\in\ttB^+$. 
When $\ttx\in \ttB^+$ and $\ttx \ledot \ttz$, we call $\ttx$ a \emph{prefix} of $\ttz$,  
and a prefix $\ttx$ of $\Updelta$ a \emph{permutation braid}.

\begin{proposition} [{\cite{Gar69} and see also \cite[Chapter 6.6]{KT08}}]\label{prop: gcd} 
The partial ordered set  $\ttB$ with the partial order $\ledot$ is a lattice;
i.e., every pair of elements of $\ttB$ has an infimum and a supremum. 
\end{proposition}

The infimum of $\ttx$ and $\ttz$ in $\ttB$ is denoted by $\ttx \wedge \ttz$
and the supremum is denoted by $\ttx \vee \ttz$.

\begin{theorem} [{\cite{WP,EM94}}] \label{thm: Garside} 
$($Garside left normal form$)$ \
Each element $\ttb \in \ttB$ can be presented
as 
$$
\Updelta^r \ttx_1 \cdots \ttx_k,
$$
where $r \in \Z$, $k \in \Z_{\ge0}$, $1 \lessdot \ttx_s \lessdot \Delta$
, and $\ttx_s=\Updelta\wedge(\ttx_s\ttx_{s+1})$ for $1 \le s <k$.
\end{theorem}
Note that the condition for the Garside normal form of $\ttb$
is that $r$ is the largest integer such that $\Updelta^{-r}\ttb\in\ttB^{+}$,
and $k$ is the largest integer such that $\ttx_k\not=1$, where
$\ttx_j\seteq\bl (\ttx_1\cdots \ttx_{j-1})^{-1}\Updelta^{-r}\ttb\br\wedge\Updelta$ for any
$j\in\Z_{>0}$.

\begin{remark} \label{rem: conlcusion loc}
Lemma~\ref{lem: loc red braid} implies the following: 
Any finite sequence $\uii=(\im_1,\ldots,\im_r)$ can be identified with $\tuii_{[1,r]}$, 
where $\tuii$ is a sequence in $\Seq(\Updelta^m)$ obtained from a locally reduced sequence $\widetilde{\ujj}\in \Seq(\Updelta^m)$ 
by applying finitely many commutation moves and braid moves.
We can choose a $\calQ$-adapted one for some $\rmQ$-datum $\calQ$ as  $\tujj$.  
\end{remark}

\subsection{An arbitrary sequence and its related simple modules} \label{subsec: Arbitrary seq and simples}
In the rest of this section, we fix a complete duality datum $\bbD = \{ L^\bbD_{\im} \}_{\in\sfI}$ in $\Cgz$.
We frequently drop $^\bbD$ in the notations throughout this section if there is no afraid of confusion.

Now, let $\uii=(\im_k)_{k\in K}$ be an \emph{arbitrary} sequence in $\sfI$
\condi.
Recall $C^{\uii}_k$ and $M^{\uii}[a,b]$ in Definition~\ref{def:affdet}.

\smallskip

Let us consider the following condition on $\uii$:
\begin{condition} \label{assu: the assumption}
  \hfill
\bnA
\item \label{it: unmixed} $(C^\uii_r,C^\uii_l)$ is strongly unmixed for $r,l\in K$ such that $r>l$. 
\item \label{it: a a^+} We have $\de(C^\uii_a,C^\uii_{a^+})=1$ for
  any $a\in K$ such that $a^+ \in K$.
\item \label{it: real simple}   $M^\uii[a,b]$ is a real simple module
  for any $i$-box $[a,b]\subset K$.
\item \label{it: commuting} For any $i$-box $[a,b]\subset K$,
  $\de(C_s^\uii,M^\uii[a,b])=0$ if $s\in K$ satisfies
  $a^- < s < b^+$.    
\item \label{it: d=1}
  For any $i$-box $[a,b]\subset K$, we have
  $\de(C^\uii_{a^-},M^\uii[a,b])=1$ if $a^-\in K$ and $\de(C^\uii_{b^+},M^\uii[a,b])=1$ if $b^+ \in K$.
\item \label{it: T-system} For any $i$-box $[a,b]\subset K$ such that $a < b$,
  we have a short exact sequence in $\Cgz$
\begin{align} \label{eq: T-system in terms of M[a,b] again}
0 \to  \dtens_{ \substack{ \jm \in \hspace{0.1ex} \sfI; \\ d(\im_a,\jm)=1}} 
M^\uii[a(\jm)^+,b(\jm)^-]  \to   M^\uii[a^+,b]  \tens M^\uii[a,b^-]    \to   M^\uii[a,b]
\tens M^\uii[a^+,b^-] \to 0,
\end{align}
and the left term and right term in~\eqref{eq: T-system in terms of M[a,b] again} are simple. 
\ee
\end{condition}
Recall that any locally reduced sequence $\uii$ satisfies the condition above as stated in 
Theorem~\ref{thm: affine cuspidal} and Theorem~\ref{thm: deter main}.
The purpose of this subsection is to prove that Condition~\ref{assu: the assumption}
holds for an arbitrary sequence $\uii$. 

\smallskip

The following proposition is easy to prove.

\begin{proposition}  \label{prop: still good sequence under ga}
  Let $\uii = (\im_{s})_{s\in K}$ be a sequence of $\sfI$ such that $\uii =\upga_k(\ujj)$ \ro see {\rm Definition~\ref{def: moves}}\rf\ for $k\in K$ such that $k+1\in K$.
  If $\ujj$ satisfies {\rm Condition~\ref{assu: the assumption}},
then so does $\uii$. 
\end{proposition}

Now let us focus on a sequence $\uii = (\im_{s})_{s\in K}$ of $\sfI$ such that $\uii =\upbe_k(\ujj)$ for $k\in K$ such that $k+2\in K$.
For simplicity of notation, let us write $\im_k=\jm_{k+1}=\im_{k+2}  = \im$ and   $\jm_k=\im_{k+1}=\jm_{k+2}  = \jm$. 

\Prop  \label{prop: still good sequence}
If $\ujj$ satisfies {\rm Condition~\ref{assu: the assumption}}, then so does $\uii$.
\enprop

We shall prove this proposition in the rest of this subsection.

\begin{remark} \label{rem: local change}
By applying 
  $\bbT\seteq (\TT_{\jm_{l}}\cdots \TT_{\jm_0})^{-1}\TT_{\jm_l} \cdots
  \TT_{\jm_{k-1}}$, we have
\begin{align*}
  &C^\ujj_k\simeq\bbT(L_{\jm}) ,  \quad C^\ujj_{k+2}\simeq \bbT(L_{\im}),
    \quad C^\ujj_{k+1}\simeq \bbT(L_{\jm} \hconv L_{\im}), \allowdisplaybreaks\\ 
  &C^\uii_k \simeq \bbT(L_{\im}) ,  \quad C^\uii_{k+2}\simeq \bbT(L_{\jm}),
    \quad C^\uii_{k+1} \simeq \bbT(L_{\im} \hconv L_{\jm}), 
\end{align*} 
since 
$$\TT_{\im}\TT_{\jm}(L_{\im}) \simeq \TT_\im (L_{\jm} \hconv L_{\im}) \simeq (L_{\im} \hconv L_{\jm}) \hconv \rdual L_{\im} \simeq L_{\jm} \qtq \TT_{\jm}\TT_{\im}(L_{\jm}) \simeq L_{\jm}.$$
  
Here we use the facts that $L_{\im} \hconv L_{\jm} \simeq \TT_{\im} (L_{\jm})$, $L_{\jm} \hconv L_{\im} \simeq \TT_{\jm} (L_{\im})$, and Lemma~\ref{lem: recover}.  
Hence we have
\begin{align} \label{eq: Sk+1 and de}
C^\uii_{k} \simeq C^\ujj_{k+2},\ C^\uii_{k+2}\simeq C^\ujj_{k}, \quad C^\uii_{k+1} \simeq C^\ujj_{k+2} \hconv C^\ujj_{k}  \qtq
\de(\rdual^n C^\ujj_k, C^\ujj_{k+2}) = \delta(n=0)
\end{align}
by Proposition~\ref{prop: ell de}. 
Furthermore, since $\{\TT_\im\}_{\im \in \sfI}$ satisfies the braid relations, we have $C^\ujj_{s} = C^\uii_{s}$ for $s \not\in [k,k+2]$.  
\end{remark}

\begin{lemma} \label{lem: still unmixed}
The property~\eqref{it: unmixed} holds for $\{ C^\uii_m\}_{l \le m \le r}$; i.e,
$(C^\uii_r,C^\uii_{l})$ is strongly unmixed for any $l,r\in K$ such that $l<r$. 
\end{lemma}

\begin{proof}
It is enough to show that the sequence $(C_b^\uii,C_a^\uii)$ is strongly unmixed when either $a$ or $b$ belongs to $\{k,k+1,k+2\}$. 
It easily follows from Proposition~\ref{prop: ell de} and Remark~\ref{rem: local change}.
\end{proof}
The following lemma is a consequence of Proposition~\ref{prop: ell de} and Remark~\ref{rem: local change}.
\Lemma We have
$$
\de(\rdual^n C_{k}^\uii,C_{k+2}^\uii ) = \delta(n=0). 
$$
\enlemma

\begin{lemma}  \label{lem: still a a+ =1}
The property~\eqref{it: a a^+} holds for $\{ C^\uii_m\}_{ m \in K}$; i.e., for any  $a\in K$ with $a^+\in K$, we have $\de(C^\uii_a,C^\uii_{a^+})=1$.  
\end{lemma}

\begin{proof}
It is enough to prove it when $a=k+2$, $a^+_{\uii}=k$, $a = k+1$ or $a^+_{\uii} = k+1$. 

\snoi
(1) $a=k+2$. First set $r=(k+2)^+_\uii=(k+1)^+_\ujj > k+2$. By Condition~\eqref{it: a a^+} for $\ujj$, we have 
\begin{align*}
1=\de(  C^\ujj_{k+1} ,C^\ujj_{r}) = \de(  C^\ujj_{k} \hconv C^\ujj_{k+2} ,C^\ujj_{r}).    
\end{align*} 
We also have $\de(  C^\ujj_{k+2},C^\ujj_{r}) =0$ by~\eqref{it: commuting} for $\ujj$, 
and $\de(  C^\ujj_{k+2},\scrD C^\ujj_{r}) =0$ by~\eqref{it: unmixed} for $\ujj$. Hence Lemma~\ref{lem: de=de}~\eqref{it: de=de 1},
we have
$$1=\de(  C^\ujj_{k} \hconv C^\ujj_{k+2} ,C^\ujj_{r}) = \de(  C^\ujj_{k} ,C^\ujj_{r})= \de(  C^\uii_{k+2} ,C^\uii_{r}).$$

\noi
(2)\ The assertion for $a^+=k$ can be proved in a similar way.

\mnoi
(3) $a^+=k+1$.  \ First set $r=(k+1)^-_\uii=(k)^-_\ujj<k$. Then we have 
we have
$$
\de(C_r^{\uii},C_{k+1}^{\uii}) = \de(C_{k^-}^{\ujj},C_{k+2}^{\ujj}\hconv C_{k}^\ujj) =  \de(C_{k^-}^{\ujj},M^{\ujj}[k,k^+])=1,
$$
which follows from~\eqref{it: d=1} for $\ujj$.

\snoi
(4) The assertion for $a=k+1$ can be proved in a similar way. 
\end{proof}

\begin{lemma}  \label{lem: still commuting}
The property~\eqref{it: commuting} holds for $\{ C^\uii_m\}_{m\in K}$; i.e., for any $i$-box $[a,b]^\uii$, we have $\de(C_s^\uii,M^\uii[a,b])=0$ if $a^- < s < b^+$. 
\end{lemma}

\begin{proof}
  The following cases are obvious.
  \bna
\item $\im_a\not\in\st{\im,\jm}$.
\item $a>k+2$.
\item $b<k$.
\item$ \im_a=\jm$.
  Indeed, we have
  $M^{\uii}[a,b]=M^{\ujj}[a,b]$ if $a,b\not=k+1$,
  $M^{\uii}[k+1,b]=M^{\ujj}[k,b]$
  and  $M^{\uii}[a,k+1]=M^{\ujj}[a,k+2]$.
\item
  $b=k+2$.
In this case, $M^{\uii}[a,k+2]=M^{\ujj}[a,k+1]$.
  \item
  $a=k$.
In this case, $M^{\uii}[k,b]=M^{\ujj }[k+1,b]$.
  \ee
Hence the remaining cases are $a=k+2<b$, and $a<b=k$.

\medskip
Since the case $a=k+2<b$ is similar, let us focus to the case $a<b=k$.
Note that
$$C_k^\uii =  C_{k+2}^\ujj, \quad M^\uii[a,k] = C_k^\uii \hconv M^\ujj[a,(k+1)^-] \qtq C^{\uii}_{k+1} = C^\ujj_{k+2} \hconv C^\ujj_{k}.$$

Let us first prove that
\begin{align} \label{eq: claim1}
\de(C_{k}^\uii, M^\uii[a,k^-]) = \de(C_k^\uii, M^\ujj[a,(k+1)^-])=\de(C_{k+2}^\ujj, M^\ujj[a,(k+1)^-])=1.     
\end{align}
By Condition~\eqref{it: d=1} for $\ujj$, we have
$$
1=\de(C_{k+1}^\ujj, M^\ujj[a,(k+1)^-]).
$$
Then we have
\begin{align*}
1 & =\de(C_{k+1}^\ujj, M^\ujj[a,(k+1)^-]) = \de(C_{k}^\ujj \hconv C_{k+2}^\ujj, M^\ujj[a,(k+1)^-]) \allowdisplaybreaks\\
& \le \de(C_{k}^\ujj , M^\ujj[a,(k+1)^-]) + \de(C_{k+2}^\ujj, M^\ujj[a,(k+1)^-]) \allowdisplaybreaks\\
& \quad \underset{*}{=} \de(C_{k+2}^\ujj , M^\ujj[a,(k+1)^-]) \underset{\sharp}{=} \de(C_{k+1}^\ujj \hconv \scrD C_{k}^\ujj \ , M^\ujj[a,(k+1)^-]) \allowdisplaybreaks\\
& \quad \quad \le \de(C_{k+1}^\ujj , M^\ujj[a,(k+1)^-]) + \de( \scrD C_{k}^\ujj \ , M^\ujj[a,(k+1)^-]) \allowdisplaybreaks\\
& \quad \quad \quad \underset{\dagger}{=} \de(C_{k+1}^\ujj , M^\ujj[a,(k+1)^-])=1,
\end{align*}
which implies~\eqref{eq: claim1}. Here 
\bnum
\item $\underset{*}{=}$ follows from~\eqref{it: commuting}  for $\ujj$, 
\item $\underset{\sharp}{=}$ follows from the fact that
$C_{k+2}^\ujj \simeq C_{k+1}^\ujj \hconv \scrD C_{k}^\ujj$ by Lemma~\ref{lem: recover}, and 
\item $\underset{\dagger}{=}$ follows from~\eqref{it: unmixed} for $\ujj$. 
\ee 
Then Lemma~\ref{lem: -1 by root} says that 
\begin{align} \label{eq: k+@ j M a =0}
\de(C_{k+2}^\ujj,C_{k+2}^\ujj \hconv M^\ujj[a,(k+1)^-])=0.    
\end{align}

Now we have
$$
\de(M^\uii[a,k],C^\uii_{k+1}) = \de(C_{k+2}^\ujj \hconv M^\ujj[a,(k+1)^-],C^\ujj_{k+2} \hconv C^\ujj_{k}).
$$
Since $\de(M^\ujj[a,(k+1)^-],C^\ujj_{k})=0$ by~\eqref{it: commuting} for $\ujj$, $C^\ujj_{k+2} \hconv C^\ujj_{k} \simeq C^\uii_{k+1}$ is real,  \eqref{eq: k+@ j M a =0} and
Lemma~\ref{lem: LX LY} imply the assertion. 
\end{proof}

\begin{lemma}  \label{lem: still d=1}
The property~\eqref{it: d=1} holds for  $\{ C^\uii_m\}_{  m\in K}$; i.e., for any $i$-box $[a,b]^\uii$, we have
$$\text{$\de(C^\uii_{a^-},M^\uii[a,b])=1$ if $a^-\in K$ and $\de(C^\uii_{b^+},M^\uii[a,b])=1$ if $b^+ \in K$.}  $$
\end{lemma}

\begin{proof}
Except the cases (i) $a=k+2$ and (ii) $b=k$,
the assertion is easy to check.  The assertion for $b=k$ is already covered
by~\eqref{eq: claim1}.

Let us consider the case $a = k+2$. Then we have
\begin{align*}
M^\uii[k+2,b] &\simeq M^\uii[(k+2)^+,b] \hconv C^\uii_{k+2} \simeq M^\ujj[(k+1)^+,b] \hconv C^\ujj_{k}.
\end{align*}
Similarly to~\eqref{eq: claim1}, we can prove
\begin{align*}
\de(M^\ujj[(k+1)^+,b], C^\ujj_{k})=1.    
\end{align*}
Thus 
\begin{align*}
  \de(C^\uii_{(k+2)^-},M^\uii[k+2,b]) &= \de(C^\uii_{k},M^\uii[k+2,b]) \\
  & = \de(C^\ujj_{k+2},M^\ujj[(k+1)^+,b] \hconv C^\ujj_{k})    \underset{*}{=}\de(C^\ujj_{k+2},C^\ujj_{k}) = \de(L_\im,L_\jm) = 1,
\end{align*}
which implies the assertion.  Here $\underset{*}{=}$ follows from Lemma~\ref{lem: de=de}~\eqref{it: de=de 2} and~\eqref{it: unmixed},~\eqref{it: commuting} for $\ujj$.
\end{proof}
 
\begin{proposition} \label{prop: Muii properties 1} The sequence $\uii=\upbe_k(\ujj)$  satisfies the following properties. 
\bnum
\item For any $i$-box $[a,b]$ associated with $\uii$, $M^\uii[a,b]$ is a real simple module.
\item If two $i$-box $[a_1,b_1]$ and $[a_2,b_2]$ commute, then $M^\uii[a_1,b_1]$ and $M^\uii[a_2,b_2]$ commutes. 
\item For any $i$-box $[a,b]$, $\de(M^\uii[a,b],M^\uii[a^-,b^-]) \le 1$.
\item \label{it: de D =1} For any $i$-box $[a,b]$, we have
$$
\de(\scrD C^\uii_b,M^\uii[a,b]) =1 \qtq \de(\scrD^{-1} C^\uii_a,M^\uii[a,b]) =1. 
$$
\ee
\end{proposition}

\begin{proof}
  The proof is the same as the ones of \cite[Theorem 4.21, Lemma 4.22, Lemma 4.23, Lemma 4.24]{KKOP24A}  based on~\eqref{it: unmixed}$\sim$\eqref{it: d=1}
  in Condition~\ref{assu: the assumption} for $\ujj$.
\end{proof}

\begin{theorem}\label{thm: Tsystem}
  The property~\eqref{it: T-system} holds for  $\{ C^\uii_m\}_{m \in K};$
  i.e,  for any $i$-box $[a,b]\subset K$ such that $a<b$,  we have an exact sequence
\begin{align*} 
0 \to  \hspace{-1ex} \tens_{ \substack{ d(\im_a,\jm)=1}}   \hspace{-1ex}  
M^\uii[a(\jm)^+,b(\jm)^-]  \to   M^\uii[a^+,b]  \tens M^\uii[a,b^-]    \to   M^\uii[a,b]
\tens M^\uii[a^+,b^-] \to 0.
\end{align*}
\end{theorem}

\begin{proof}
For simplicity of notation, we write  $C_u$ for  $C^\ujj_u$, $C'_u$ for $C^\uii_u$, 
$M[a,b]$ for  $M^\ujj[a,b]$ and $M'[a,b]$ for  $M^\uii[a,b]$.
We also write $M[a^+,b^-]$, $M'[a^+,b^-]$, etc.\ for
$M^\ujj[a_\ujj^+,b_\ujj^-]$, $M^\uii[a_\uii^+,b_\uii^-]$, etc.
For $\kappa$, $\kappa'\in\sfI$, we write $\kappa\sim\kappa'$ if $d(\kappa,\kappa')=1$.

\smallskip
We shall prove this theorem  by induction on $|[a,b]_\phi|=|\st{k\in[a,b]\mid \im_k=\im_a}| \ge 2$.

For the start of induction, it is enough to consider the cases when 
$ b= a^+_\uii$ and $\{a,b\} \cap \{k,k+1,k+2\} \ne \emptyset$.
Since~\eqref{it: unmixed} and~\eqref{it: d=1} hold for $\uii$,
it is enough to check that $C'_a \hconv C'_b \simeq\otimes_{\substack{\jm\sim\im_a}}     
M'[a(\jm)^+,b(\jm)^-]$ by Proposition~\ref{prop: de properties}~\eqref{it: d=1 length2}.

\snoi
(1: $\; a=k$) We have $b=k+2$,
$ \tens_{\jm\sim\im_a} M'[a(\jm)^+,b(\jm)^-] = C'_{k+1}$ 
and hence the assertion is obvious. 

\mnoi
(2:$\; a=k+1$) We have $b=(k+1)^+_\uii=(k+2)^+_\ujj$ and $C'_{k+1}=M[k,k+2]$. 
Note that 
\begin{align*}
& C'_{k+1} \hconv C'_{(k+1)^+} \simeq  (C_{k+2} \hconv C_k)  \hconv C_{(k+2)^+} = (C_{k+2} \hconv C'_{k+2})  \hconv C'_{(k+1)^+}.
\end{align*}

Since $C'_{k+2}$ commutes with $C'_{(k+1)^+}$ by~\eqref{it: commuting} for $\uii$, the sequence $(C_{k+2},C'_{k+2},C'_{(k+1)^+})$ is normal and
\begin{align*}
\head( C_{k+2} \tens  & C'_{k+2}  \tens  C'_{(k+1)^+})\simeq \head( C_{k+2} \tens C'_{(k+1)^+}  \tens C'_{k+2}) \simeq \head( C_{k+2} \tens C_{(k+2)^+}  \tens C_{k})   \allowdisplaybreaks    \\
&\simeq  (C_{k+2} \hconv C_{b} ) \hconv  C_{k}  \allowdisplaybreaks \\
& \simeq  \left(\left(\tens_{\kappa \sim \jm; \kappa\ne \im} M[(k+2)(\kappa)^+,b(\kappa)^-]\right)
\tens M[(k+2)(\im)^+,b(\im)^-] \right) \hconv  C_{k} \allowdisplaybreaks \\
& \simeq  \left(\left(\tens_{\kappa \sim \jm; \kappa\ne \im} M'[(k+1)(\kappa)^+,b 
(\kappa)^-]\right) \tens M'[(k+2)^+,b(\im)^-] \right) \hconv  C'_{k+2}. 
\end{align*}
Since 
\bnum
\item $C'_{k+2}$ commutes with $M'[(k+1)(\kappa)^+,b(\kappa)^-]$ for $\kappa \sim \jm$ and  $\kappa \ne \im$, 
\item $M'[(k+2)^{+},b(\im)^-] \hconv  C'_{k+2} \simeq M'[k+2,b(\im)^-]$,
\ee
our assertion follows. 

\mnoi
(3: $\; a=k+2$)  We have $b=(k+2)^+_\uii = (k+1)^+_\ujj$. Then we have
\begin{align*}
& C'_{k+2} \hconv C'_{b}\simeq C_{k} \hconv C_{(k+1)^+} \simeq\bl\scrD^{-1}C_{k+2} \hconv C_{k+1}\br \hconv C_{(k+1)^+}.
\end{align*}

Since $C_{(k+1)^+}$ and $C_{k+2}$ commutes, the sequence $(\scrD^{-1}C_{k+2}, C_{k+1},C_{(k+1)^+})$ is normal. 
Hence
\begin{align*}
&\head(\scrD^{-1}C_{k+2} \tens C_{k+1} \tens C_{(k+1)^+}) \simeq \scrD^{-1}C_{k+2} \hconv (C_{k+1} \hconv C_{(k+1)^+}) \allowdisplaybreaks\\
& \simeq \scrD^{-1}C_{k+2} \hconv \left( \left(\tens_{\kappa \sim \im; \kappa \ne \jm} M[(k+1)(\kappa)^+,b(\kappa)^-]\right) \tens M[k+2,b(\jm)^-] \right) \allowdisplaybreaks \\
  & \simeq \scrD^{-1}C_{k+2} \hconv \left( \left(\tens_{\kappa \sim \im; \kappa \ne \jm} M'[(k+1)(\kappa)^+,b(\kappa)^-]\right) \tens
    \bl M[(k+2)^+,b(\jm)] \hconv C_{k+2}\br\right)
\end{align*}
Since 
\bnum
\item $\scrD^{-1}C_{k+2}$ commutes with $M[(k+1)(\kappa)^+,b(\kappa)^-]$ if $\kappa\sim \im$ and $\kappa\ne \jm$,
\item $\scrD^{-1}C_{k+2} \hconv (M[(k+2)^+,b(\jm)^-] \hconv C_{k+2}) \simeq  M[(k+2)^+,b(\jm)^-]\simeq   M'[(k+2)(\jm)^+,b(\jm)^-]$,
\ee
our assertion follows for this case.

In a similar way, one can prove when $b=k,k+1$, which completes the assertion when  $|[a,b]_\phi|=2$. 

\smallskip
The assertion for $|[a,b]_\phi| \ge 3$ follows from the same argument of \cite[Theorem 4.25]{KKOP24A}.
\end{proof}

\begin{proof} [{\bf End of the proof of {\rm \bf Proposition~\ref{prop: still good sequence}}}]
By Lemma~\ref{lem: still unmixed},~\ref{lem: still a a+ =1},
~\ref{lem: still commuting},~\ref{lem: still d=1} and Theorem~\ref{thm: Tsystem}, we conclude that $\uii$
satisfies Condition~\ref{assu: the assumption}.      
\end{proof}

As a corollary of Proposition~\ref{prop: still good sequence}, we obtain the following main result of this subsection.

\begin{theorem} \label{thm: every sequence good}
An arbitrary sequence $\uii=(\im_k)_{k\in K}$  in $\sfI$ \condi
satisfies {\rm Condition~\ref{assu: the assumption}}. Namely, we have
\bnum
\item $(\ldots, C^\uii_1,C^\uii_{0},\ldots)$ is strongly unmixed.
\item If $a\in K$ satisfies $a^+\in K$, then we have $\de(C^\uii_a,C^\uii_{a^+})=1$.
 \item  $M^\uii[a,b]$ is a real simple module for any $i$-box $[a,b]\subset K$.
\item \label{it: commutingv} $\de(C_s^\uii,M^\uii[a,b])=0$ if $a^- < s < b^+$.    
\item \label{it: d=1 M} $\de(C^\uii_{a^-},M^\uii[a,b])=1$ if $a^-\in K$ and $\de(C^\uii_{b^+},M^\uii[a,b])=1$ if $b^+ \in K$.
\item  For any $i$-box $[a,b]$ such that $a < b$, we have a short exact sequence in $\Cgz$
\begin{align} 
0 \to  \dtens_{ \substack{ \jm \in \hspace{0.1ex} \sfI; \\ d(\im_a,\jm)=1}} 
M^\uii[a(\jm)^+,b(\jm)^-]  \to   M^\uii[a^+,b]  \tens M^\uii[a,b^-]    \to   M^\uii[a,b]
\tens M^\uii[a^+,b^-] \to 0.
\end{align}
\ee
\end{theorem}

\begin{proof} Assume first $l=1$.
Let us choose $\tuii,\tujj  \in \Seq(\Updelta^m)$ as in Remark~\ref{rem: conlcusion loc}.
Since $\tujj$ satisfies Condition~\ref{assu: the assumption},
so does $\tuii$ as well as $\tuii$ by Proposition~\ref{prop: still good sequence under ga} and Proposition~\ref{prop: still good sequence}.

\medskip
The general case follows from $l=1$ case and Lemma~\ref{lem:lneg}.
\end{proof}

During the proof of Proposition~\ref{prop: still good sequence}, we can conclude the following corollary as in \cite[Theorem 4.25]{KKOP24A} (see also Theorem~\ref{thm: deter main}~\eqref{it: de=1}).
\begin{corollary}
For any $i$-box $[a,b]^\uii$, we have
$$
\de(M^\uii[a^-,b^-],M^\uii[a,b])=1. 
$$
\end{corollary}

The following proposition can be proved by using the results in this subsection
and the argument in the proof of \cite[Proposition 5.7]{KKOP24A}.

\begin{proposition}
Let $\frakC=(\frakc_k)_{1 \le k \le r-l+1}$ be an admissible chain of $i$-boxes of range $[l,r]$ which is associated with $\uii$. 
For an movable $i$-box $\frakc_{k_0}$ assume $\tfrakc_{k_0+1}$ is an $i$-box. Set $\tfrakc_{k_0+1}=\frakc_{k_0+1}=[a,b]$ and set $B_{k_0}(\frakC)=(\frakc'_k)_{1 \le k \le r-l+1}$. 
Then we have
\bnum
\item \label{it: T-R} $\frakc_{k_0}=[a^+,b]$ and $\frakc'_{k_0}=[a,b^-]$ if $\calH_{k_0-1}=\calR$,
\item \label{it: T-L} $\frakc_{k_0}=[a,b^-]$ and $\frakc'_{k_0}=[a^+,b]$ if $\calH_{k_0-1}=\calL$.
\ee
In particular, we have an exact sequence
\begin{align} \label{eq: Tsystem2}  
0 \to  \tens_{
d(\im_a,\jm)=1} M^\uii[a(\jm)^+,b(\jm)^-]  \to  X\tens Y
  \to   M^\uii(\frakc_{k_0+1}) \tens M^\uii[a^+,b^-]   \to 0.    
\end{align}  
 where $(X,Y)=\big( M^\uii(\frakc_{k_0}),M^\uii(\frakc'_{k_0})\big)$ in case \eqref{it: T-R} and 
 $(X,Y)=\big(M^\uii(\frakc'_{k_0}),M^\uii(\frakc_{k_0})\big)$ in case \eqref{it: T-L}.
\end{proposition}

For an arbitrary finite sequence $\uii=(\im_{k})_{k\in K}$, the anti-symmetric pairing defined in~\eqref{eq: anti-symmetric pairing original}
can be written as follows: For $a , b \in K$, 
\begin{align} \label{eq: la i ab}
\la^{\uii}_{a,b} = (-1)^{\delta(a>b)} \delta(a \ne b) (\be^\uii_a,\be^\uii_b)    
\end{align}
where 
$$\beta^\uii_k \seteq s_{\im_l}\cdots s_{\im_{k-1}}(\al_{\im_k})$$
which is a (not necessarily positive) root. Here 
we take $l\in K$ such that $l\le a,b$.

\begin{proposition} \label{prop: new pairing La}
Let $\uii=(\im_k)_{k\in K}$ be an arbitrary finite sequence of $\sfI$. Let $[a,b]$ and $[a',b']$ be $i$-boxes in $K$ and assume that
\begin{align} \label{eq: pairing assumption}
{\rm (a)} \quad a > (a')^- \qtoq {\rm (b)}  \quad b^+ > b'.      
\end{align}
Then we have 
\begin{align*}
 \La( M^\uii[a,b],M^\uii[a',b'] ) =    \sum_{u \in [a,b]_\phi, v \in [a',b]_\phi} \la^\uii_{u,v}.
\end{align*}
\end{proposition}

\begin{proof}
Since the proofs for {\rm (a)} and  {\rm (b)} are similar, we shall give only the proof of {\rm (a)}. For simplicity of notation, we drop $^\uii$
throughout the proof.

\noindent
(i) Assume that $a=b > (a')^-$. If $a > a'$, then 
\begin{align*}
 \La(C_a,M[a',b']) & =  \La(C_a,M[(a')^+,b'] \hconv C_{a'} ) \\
 & \underset{*}{=} \La(C_a,M[(a')^+,b'] ) +\La(C_a,C_{a'} ) \underset{\dagger}{=}  \La(C_a,M[(a')^+,b'] ) + \la_{a,a'}
\end{align*}
Here $\underset{*}{=}$ holds by Lemma~\ref{lem: normal property}, Lemma~\ref{lem: normal seq d} and the property~\eqref{it: unmixed} for $\uii$, and
$\underset{\dagger}{=}$ holds by Proposition~\ref{prop: Unmix normal}~\eqref{it: Li La} and Theorem~\ref{thm: wtpairing Lainf}. Then, by the induction hypothesis on $|[a',b']_\phi|$, we have
$$
 \La(C_a,M[a',b']) = \sum_{v \in [a',b']_\phi} \la^\uii_{a,v}.
$$

Now, let us consider the remaining case of (i) which can be described as follows:
$$
(a')^- < a = b \le a' \le b'. 
$$
Since $C_a$ commutes with $M[a',b']$ and  $M[a',(b')^-]$ by~\eqref{it: commutingv} for $\uii$, 
\begin{align*}
\La(C_a,M[a',b']) & = - \La(M[a',b'],C_a) = - \La(C_{b'} \hconv M[a',(b')^-],C_a)     \allowdisplaybreaks \\
& = - \La(C_{b'} ,C_a) -  \La( M[a',(b')^-],C_a)  \allowdisplaybreaks\\
& = (\be_{b'},\be_{a}) + \La(C_a, M[a',(b')^-]) = \la_{a,b'} +  \La(C_a, M[a',(b')^-]). 
\end{align*}
Then our assertion follows from the induction hypothesis on $|[a',b']_\phi|$ and the previous case. 

\noindent
(ii) Assume $a<b$. If $b>b'$, then 
\begin{align*}
\La(M[a,b],M[a',b']) & = \La(C_b \hconv M[a,b^-],M[a',b']) \allowdisplaybreaks\\
& = \La(C_b ,M[a',b']) + \La(M[a,b^-],M[a',b']),
\end{align*}
since $(C_b, M[a,b^-],M[a',b'])$ is a normal sequence by~\eqref{it: unmixed} for $\uii$. Then by the induction $|[a,b]_\phi|$, our assertion for $b>b'$ follows.

Now, let us assume $b \le b'$ which completes this assertion. Then we have
$(a')^- < a < b \le b'$. Then~\eqref{it: commutingv}  for $\uii$ says that $C_u$ commutes with $M[a',b']$ for any $u \in [a,b]_\phi$. Then we have
$$
\La(M[a,b],M[a',b']) = \sum_{u \in [a,b]_\phi} \La(C_u,M[a',b'])
$$
by \cite[Proposition 4.2]{KKOP20}. Then our assertion follows from (i). 
\end{proof}

\begin{proposition}\label{prop:detLambda}
Let $\uii=(\im_k)_{k\in K}$ be an arbitrary finite sequence of $\sfI$. For $i$-boxes $[a,b]$ and $[a',b']$ in $K$, if
$\de(M^\uii[a,b],M^\uii[a',b'])=0$, then 
$$
\La(M^\uii[a,b],M^\uii[a',b']) = \sum_{u \in [a,b]_\phi, v \in [a',b']_\phi} \la^\uii_{u,v}.
$$
\end{proposition}

\begin{proof} 
By Proposition~\ref{prop: new pairing La},
it is enough to consider the case $a \le a'$ and $b^+ \le b'$.  
 Since  $\de(M^\uii[a,b],M^\uii[a',b'])=0$,
we have 
$$\La(M^\uii[a,b],M^\uii[a',b'])=-\La(M^\uii[a',b'],M^\uii[a,b]).$$

If $a'>a^-$ or $(b')^+>b$, Proposition~\ref{prop: new pairing La} says that
$$\La(M^\uii[a,b],M^\uii[a',b']) = - \sum_{u \in [a,b]_\phi, v \in [a',b']_\phi} \la^\uii_{v,u}=\sum_{u \in [a,b]_\phi, v \in [a',b']_\phi} \la^\uii_{u,v},$$
which implies the assertion. Thus we may assume that $a'\le a^-$. However, this assumption implies 
$$a' \le a^- < a \le a',$$
which yields a contradiction. 
\end{proof}

\begin{lemma}
Let $\uii=(\im_k)_{k\in K}$ be an arbitrary sequence. 
Then, for $a,k\in K$ with $a^-< k < a^+$, we have
$$
(\be^\uii_k, w^{\uii}_{\le a^-}\varpi_{\im_a}+w^{\uii}_{\le a}\varpi_{\im_a})
= \bc
-(\be^\uii_k,\be^\uii_a) & \text{if {\rm (i)} $a^- < k  < a$,} \\
(\be^\uii_k,\be^\uii_a) & \text{if {\rm (ii)} $a< k < a^+$,}\\
0&\text{if {\rm (iii)} $k=a$,} 
\ec
$$
where $w_{\le k}^\uii=s_{\im_l}\cdots s_{\im_k}$.
\enlemma
\begin{proof}
{\rm (i)} Note that
$$
(\be^\uii_k, w^{\uii}_{\le a^-}\varpi_{\im_a}+w^{\uii}_{\le a}\varpi_{\im_a}) 
= (\be^\uii_k, 2w^{\uii}_{\le a^-}\varpi_{\im_a}-\be^\uii_a).
$$
Then it suffices to show that $(\be^\uii_k,  w^{\uii}_{\le a^-}\varpi_{\im_a})=0$. 
Since 
$$
(s_{\im_l}\ldots s_{\im_{k-1}}\al_{\im_k},   s_{\im_l}\ldots s_{\im_{a^-}} \varpi_{\im_a})
=  (s_{\im_{a^-+1}}\ldots s_{\im_{k-1}}\al_{\im_k},    \varpi_{\im_a}) = 0,
$$
the first case follows.

\mnoi
{\rm (ii)} Note that
$$
(\be^\uii_k, w^{\uii}_{\le a^-}\varpi_{\im_a}+w^{\uii}_{\le a}\varpi_{\im_a}) 
= (\be^\uii_k, 2w^{\uii}_{\le a}\varpi_{\im_a}+\be^\uii_a).
$$
As in the previous case, we have
$$
(s_{\im_1}\ldots s_{\im_{k-1}}\al_{\im_k},   s_{\im_1}\ldots s_{\im_{a}} \varpi_{\im_a})
=   (s_{\im_{a+1}}\ldots s_{\im_{k-1}}\al_{\im_k},   \varpi_{\im_a}) = 0,
$$
which completes this case.

\snoi
{\rm (iii)} follows from the fact that $\be_a = w^{\uii}_{\le a^-}\varpi_{\im_a}-w^{\uii}_{\le a}\varpi_{\im_a}$.  
\end{proof}

\begin{lemma} 
Let $\uii=(\im_k)_{k\in K}$ be an arbitrary sequence. For an $i$-box $[a,b] \subset K$ and $k \in K$ with $\im=\im_a=\im_b$ and
$a^-<k<b^+$, we have
\begin{align} \label{eq: the pairing with w}
\La(C^\uii_k,M^\uii[a,b]) = -(\be^\uii_k,w^\uii_{\le a^-}\varpi_{\im} + w^\uii_{\le b}\varpi_{\im})    
\end{align}
\end{lemma}

\begin{proof}
Let $n,t,u \in \Z_{>0}$ be integers such that $u=a^{+(n-1)} \le  k <a^{+n}$ and $a^{+t}=b$.
Then the right hand side of~\eqref{eq: the pairing with w} becomes
\begin{align*}
 (\be_k, \sum_{i=0}^{n-2} \be_{a^{+i}}) +
&(\be_k, w_{\le u^-}\varpi_\im+w_{\le u}\varpi_\im) -
(\be_k, \sum_{i=n}^{t} \be_{a^{+i}})\allowdisplaybreaks \\
& = (\be_k, \sum_{i=0}^{n-2} \be_{a^{+i}}) - \la^\uii_{k,u} - 
(\be_k, \sum_{i=n}^{t} \be_{a^{+i}}) = -\sum_{v \in [a,b]_\phi} \la^\uii_{k,v}. \qedhere
\end{align*}
\end{proof}

\begin{corollary} \label{cor: w-pairing}
  Let $\uii=(\im_k)_{k\in K}$ be an arbitrary sequence \condi
  and let $[a_1,b_1], [a_2,b_2] \subset K$ be $i$-boxes such that $a_2^- < a_1 \le b_1 < b_2^+$. Then we have
$$
\La(M^\uii[a_1,b_1],M^\uii[a_2,b_2]) = -(w^\uii_{\le a_1^-}\varpi_{\im_{a_1}} - w^\uii_{\le b_1}\varpi_{\im_{b_1}},w^\uii_{\le a^-_2}\varpi_{\im_{a_2}} + w^\uii_{\le b_2}\varpi_{\im_{b_2}}).    
$$
In particular,
$\La(M^\uii\{l,b_1],M^\uii\{l,b_2]) =  -(\varpi_{\im_{b_1}}- w^\uii_{\le b_1}\varpi_{\im_{b_1}},\varpi_{\im_{b_2}}+ w^\uii_{\le b_2}\varpi_{\im_{b_2}})$. 
\end{corollary}

\bigskip
\subsection{Category $\Cg(\ttb)$ and unitriangularity} \label{subsec: subcat C(b)}
Let us take an arbitrary sequence $\uii=(\im_k)_{k\in K}$ \condi.
For $\bsa = (a_k)_{k\in K}\in \Z^{\oplus K}_{\ge 0}$, we set 
\begin{align} \label{eq: Pi standard module}
  P^{\bbD,\uii}(\bsa) \seteq \otens C_k^{\tens a_k}=
  \cdots\tens C_{1}^{\tens a_{1}}\tens C_0^{\tens a_0}\tens \cdots ,
\end{align}
and call it the \emph{standard module} associated with $\uii$ and $\bsa$.  

\smallskip

\begin{definition} \label{def: bi-lexico}
For $\bsa=(a_k)_{k\in K}$, $\bsa'=(a'_k)_{k\in K}\in \Z_{\ge0}^{\oplus K}$, let us consider the following conditions:
\bna
\item there exists $s \in K$ such that $a_k = a_k'$ for any $k <  s$ and $ a_s < a_s'$,
\item there exists $u \in K$ such that $ a_k = a_k' $ for any $k > u$ and $ a_u < a_u'$.
\ee
We write $\bsa\prec_\ttr\bsa'$ (resp.\ $\bsa\prec_\ttl\bsa'$)
if (a) (resp.\ (b)) is satisfied, and 
$\bsa\prec\bsa'$ if the both conditions are satisfied.  
\end{definition} 

Since $(\ldots,C_1,C_0,\ldots)$ is strongly unmixed, $V^{\Di}(\bsa) \seteq \head(P^{\Di}(\bsa))$ is simple. 

The following two lemmas are proved for
$C_k=C^{\Dpair}_k$ in \cite[Lemma 6.9, Theorem 6.12]{KKOP23P}.
However its proof only uses the fact that $(\ldots,C_{1},C_0,\ldots)$
is a strongly unmixed sequence of root modules.
Hence it also holds for an arbitrary sequence $\uii$.
\begin{lemma} [{\cite[Lemma 6.9]{KKOP23P}}]\label{lem: filter}
  Let $\uii=(\im_k)_{k\in K}$ be an arbitrary sequence in $\sfI$
  \condi.
Set $\ttS_k =C_k^{\bbD,\uii}$. 
For a finite interval $[n,m]\subset K$ and $a_{m},a_{m+1},\ldots,a_n \in \Z_{\ge 0}$, set
$$ M \seteq \head( \ttS_{m}^{\otimes a_m} \tens \ttS_{m-1}^{\otimes a_{m-1}} \tens \cdots \tens \ttS_{n}^{\otimes a_n} ).$$
\bnum
\item $\de(\rdual \ttS_k,M)=0$ for any $k >m$.
\item Set $M_m \seteq M$ and define inductively
$$d_k \seteq \de(\rdual \ttS_k,M_k) \qtq M_{k-1} \seteq M_k \hconv \rdual(\ttS_k^{\otimes d_k})$$
for $k\in [l,m]$. Then 
$$
d_k =a_k \qtq M_k \simeq \head( \ttS_{k}^{\otimes a_k} \tens \ttS_{k-1}^{\otimes a_{k-1}} \tens \cdots \tens \ttS_{l}^{\otimes a_{l}}) \quad \text{ for $k \in [n,m]$.}
$$
\item $\de(\ldual \ttS_k,M)=0$ for any $k <n$.
\item Set $N_n \seteq M$ and define inductively
$$e_k \seteq \de(\ldual \ttS_k,N_k) \qtq N_{k+1} \seteq \ldual(\ttS_k^{\otimes e_k}) \hconv N_k$$
for $k\in [n,m]$. Then 
$$
e_k =a_k \qtq M_k \simeq \head( \ttS_{m}^{\otimes a_m} \tens \cdots \tens \ttS_{k+1}^{\otimes a_{k+1}} \tens \ttS_{k}^{\otimes a_{k}}) \quad \text{ for $k \in [n,m]$.}
$$
\ee
\end{lemma}

 \Lemma \label{lem:subq}
  Let $\bsa$, $\bsb\in \Z_{\ge 0}^{\oplus K}$.
  \bnum
  \item $\Vdi(\bsa)$ appears only once in the composition series of $\Pdi(\bsa)$.
\item If $\Vdi(\bsa)$ appears in the composition series of $\Pdi(\bsb)$, then we have
  $\bsa\preceq \bsb$.
\ee
\end{lemma}

\begin{proof}
  The assertion follows from Proposition~\ref{prop:normalsimple} \eqref{itsimple} and Proposition~\ref{prop: Unmix normal} \eqref{it: unmix normal}. \QED

\begin{definition} \label{def: CgDi}
Let $\uii=(\im_k)_{k\in K}$ be an arbitrary sequence in $\sfI$
  \condi. We denote by $\scrC_\g^{\bbD,\uii}$ the  smallest full subcategory of $\Cgz$ satisfying the following properties:
\ben
\item it is stable under taking tensor products,  subquotients and extensions,
\item it contains $\{ C_m^{\bbD,\uii}\}_{m\in K}$ and $\mathbf{1}$. 
  \ee
  \end{definition}
Note that  we have
\eq
\parbox{75ex}{Any simple $S$ in $\scrC_\g^{\bbD,\uii}$
  is isomorphic to a subquotient of
  $\Pdi(\bsa)$ for some $\bsa\in\Z_{\ge0}^{\oplus K}$.}
\label{eq:subP}
\eneq
Let us consider the following condition on $\uii$, 
\eq
&&\hs{5ex}\parbox{\textwidth-10ex}{
  For a simple module $M$ in $\scrC_\g^{\bbD,\uii}$, there exists  $\bsa =(a_k)_{k\in K}\in \Z_{\ge0}^{\oplus K}$ such that\\[.5ex]
  \hs{20ex}$M \simeq  \Vdi(\bsa)\seteq\head(\Pdi(\bsa))$.}
\label{eq: complete}
\eneq
Later we prove that an arbitrary sequence $\uii$ satisfies \eqref{eq: complete}.
Before proving this, we discuss consequences of
\eqref{eq: complete}.

\Th \label{thm: unitri beta}
Let $\uii=(\im_k)_{k\in K}$ be a sequence in $\sfI$ \condi.
Assume that $\uii$ satisfies \eqref{eq: complete}.
\bnum
\item \label{it: unitri 1}
Let $\bsa\in \Z_{\ge0}^{\oplus K}$. If $V$ is a simple subquotient of $P^\uii(\bsa)$ which is not isomorphic to $V^\uii(\bsa)$, then 
there exists $\bsb \in \Z_{\ge0}^{\oplus K}$ such that
$$
V \simeq V^\uii(\bsb) \qtq \bsb \prec \bsa. 
$$
\item \label{it: unitri 2} In the Grothendieck ring, we have 
\begin{align} \label{eq: uni map2}
[P^\uii(\bsa)] = [V^\uii(\bsa)] + \sum_{\bsb  \prec \bsa} c^{\bbD}_{\bsa,\bsb} [V^\uii(\bsb)] \quad\text{for some $ c_{\bsa,\bsb} \in \Z_{\ge 0}$.}    
\end{align}
\item \label{it: unitri 3} $\st{[P^\uii(\bsa)]}_{\bsa\in\Z_{\ge0}^{\oplus K}}$, as well as 
  $\st{[V^\uii(\bsa)]}_{\bsa\in\Z_{\ge0}^{\oplus K}}$, is a $\Z$-basis of $K(\scrC_\g^{\bbD,\uii})$.
\ee
\enth

\begin{proof}
\eqref{it: unitri 1} is an immediate consequence of Lemma~\ref{lem:subq},
and \eqref{it: unitri 2} and \eqref{it: unitri 3} are consequences of~\eqref{it: unitri 1}.
\QED

{}From~\eqref{eq: complete}, for a simple module $X$ in $\scrC_\g^{\bbD,\uii}$ with $X \simeq V^{\uii}(\bsa)$,  we set
\begin{align}
\PBW_\uii(X) \seteq \bsa \in \Z_{\ge0}^{\oplus K}.     
\end{align}

Using $\PBW_\uii(X)$ and~\eqref{eq: la i ab}, we can define the anti-symmetric pairing $\rmL_\uii$ on the set of pairs of
simple modules in $\scrC_\g^{\bbD}(\ttb)$, a generalization of $L_\ii$ in \cite[(5.7)]{KKOP23F}, as follows: For simple modules $X,Y$ in $\scrC_\g^{\bbD}(\ttb)$ and $\uii \in \Seq(\ttb)$, 
\begin{align} \label{eq: def Li}
\rmL_\uii(X,Y) \seteq \sum_{a,b \in K} (\PBW_\uii(X))_a (\PBW_\uii(Y))_b \; \la^\uii_{a,b}.
\end{align}
Then we can interpret Corollary~\ref{cor: w-pairing} as
$$
\La(M^\uii[a_1,b_1],M^\uii[a_2,b_2])  = \rmL_\uii(M^\uii[a_1,b_1],M^\uii[a_2,b_2])
\quad \text{if $a_2^- < a_1 \le b_1 < b_2^+$.} 
$$

Now we will show that an arbitrary sequence $\uii$ always satisfies this condition~\eqref{eq: complete}.

\Lemma[{\cite[Theorem 6.10]{KKOP23P}}]\label{lem:hwccomp} 
For any interval $[a,b]$ in $\Z$, $\uii=(\hwc)_{[a,b]}$ satisfies condition \eqref{eq: complete}.
\enlemma
 
\begin{lemma} \label{lem: untri gamma}
Assume that a sequence $\uii=\st{\im_k}_{k\in K}$ \condi satisfies \eqref{eq: complete}.
Let $\ujj$ be a sequence in $\sfI$ with $\ujj= \upga_k(\uii)$ $(k,k+1\in K)$. Then $\ujj$ also satisfies  \eqref{eq: complete}.
\end{lemma}

\begin{proof} Note that $C^{\ujj}_{m} \simeq C^{\uii}_{\upsigma_k(m)}$ for all $1 \le m \le r$. By \eqref{it: commutingv} in
  Theorem~\ref{thm: every sequence good},  
$\de(C^{\uii}_k,C^{\uii}_{k+1}) =0$. Hence $P^\uii(\bsa) \simeq P^\ujj(\upsigma_k(\bsa))$, which implies the assertion.
\end{proof}

\begin{lemma} \label{lem: exsitence of bsa for ujj}
Assume that a sequence $\uii=\st{\im_k}_{k\in K}$ \condi satisfies \eqref{eq: complete}. 
Let $\ujj$ be a sequence with $\ujj= \upbe_k(\uii)$ $(k,k+1,k+2\in K)$.
For a simple module $M \in \scrC_\g^{\bbD,\uii}$,
there exists $\bsa' \in \Z_{\ge}^r$ such that
$$M \simeq  V^\ujj(\bsa') = \head(P^\ujj(\bsa')).$$
\end{lemma}
\begin{proof}
As seen in \S~\ref{subsec: Arbitrary seq and simples}, $C_{s}^\uii \simeq C_{s}^\ujj$ for $s \not\in[k,k+2]$, $C_{k}^\uii \simeq C_{k+2}^\ujj$,
$C_{k+2}^\uii \simeq C_{k}^\ujj$, $C_{k+1}^\uii \simeq C_{k}^\uii \hconv C_{k+2}^\uii$ and $C_{k+1}^\ujj \simeq C_{k}^\ujj \hconv C_{k+2}^\ujj$. 
Since $\Ci_{k+1} \simeq \Ci_{k} \hconv \Ci_{k+2}$, $\Cj_{k+1} \simeq \Cj_{k} \hconv \Cj_{k+2}$  and 
\begin{align*}
\de(\Ci_k,\Ci_{k+1})=\de(\Ci_{k+2},\Ci_{k+1})=\de(\Ci_{k},\Ci_{k+2} \nabla \Ci_{k})= \de(\Ci_{k+2},\Ci_{k+2} \nabla \Ci_{k})=0,    
\end{align*}
we have
\eqn
&&\head( \Ci_{k+2}{}^{\otimes a_{k+2}} \tens \Ci_{k+1}{}^{\otimes a_{k+1}}
\tens \Ci_{k}{}^{\otimes a_{k}} ) 
\simeq \head(  \Ci_{k+1}{}^{\otimes a_{k+1}} \tens \Ci_{k+2}{}^{\otimes a_{k+2}} \tens \Ci_{k}{}^{\otimes a_{k}} ) \\
&&\hs{7ex}\simeq  \bc
\head( (\Cj_{k+2} \hconv \Cj_{k})^{\otimes a_{k+1}} \tens \Cj_{k}{}^{\otimes a_{k+2}-a_k}   \tens \Cj_{k+1}{}^{\otimes a_{k}} )  & \text{ if } \min(a_{k},a_{k+2})=a_{k}, \\
\head( (\Cj_{k+2} \hconv \Cj_{k})^{\otimes a_{k+1}} \tens  \Cj_{k+1}{}^{\otimes a_{k+2}} \tens \Cj_{k+2}{}^{\otimes a_{k}-a_{k+2}} ) & \text{ if }
\min(a_{k},a_{k+2})=a_{k+2},
\ec  \\
&& \hspace{7ex} \simeq  \bc
\head( \Cj_{k+2}{}^{\otimes a_{k+1}} \tens \Cj_{k+1}{}^{\otimes a_{k}} \tens \Cj_{k}{}^{\otimes a_{k+2}+a_{k+1}-a_k}    )  & \qquad  \text{ if } \min(a_{k},a_{k+2})=a_{k}, \\
\head( \Cj_{k+2}{}^{\otimes a_{k}+a_{k+1}-a_{k+2}} \tens  \Cj_{k+1}{}^{\otimes a_{k+2}} \tens \Cj_{k}{}^{\otimes a_{k+1}} ) & \qquad  \text{ if }
\min(a_{k},a_{k+2})=a_{k+2},
\ec 
\eneqn
for each $(a_{k},a_{k+1},a_{k+2}) \in \Z^{3}_{\ge 0}$.
Hence we have 
\begin{align} \label{eq: braid move for head}
\head( \Ci_{k+2}{}^{\otimes a_{k+2}} \tens \Ci_{k+1}{}^{\otimes a_{k+1}} \tens \Ci_{k}{}^{\otimes a_{k}} ) 
\simeq \head( \Cj_{k+2}{}^{\otimes a'_{k+2}} \tens \Cj_{k+1}{}^{\otimes a'_{k+1}} \tens \Cj_{k}{}^{\otimes a'_{k}} ) 
\end{align}
where 
\begin{align} \label{eq: braid mutation rule for heads}
\bc
a_k'= a_{k+1}+a_{k+2} - \min(a_{k},a_{k+2}),  \\
a_{k+1}'= \min(a_{k},a_{k+2}),  \\
a_{k+2}'= a_{k+1}+a_{k}- \min(a_{k},a_{k+2}).
\ec
\end{align}
For a simple module $M$ in $\scrC_\g^{\bbD,\uii}$ with $M \simeq V^\uii(\bsa)$, we have
$$
M \simeq \head\big( V^\uii(\bsa_{>k+2}) \tens V^\uii(\bsa_{[k,k+2]}) \tens V^\uii(\bsa_{<k})\big),
$$
since $(\Ci_r,\Ci_{r-1},\ldots,\Ci_1)$ is strongly unmixed. 
Here $\bsa_{>k+2} \seteq (0,\ldots,0,a_{k+3},\ldots)$
$\bsa_{[k,k+2]} \seteq (0,\ldots,0,a_{k},a_{k+1},a_{k+2},0,\ldots,0)$
and $\bsa_{<k} \seteq (\ldots,a_{k-1},0,\ldots,0)$. 
\end{proof}
\begin{remark}
The formula~\eqref{eq: braid mutation rule for heads} is well-known for a reduced sequence $\uw$ of $w \in \weyl$ and is given in \cite[Chapter 42]{LusztigBook}.  
\end{remark} 

By Lemma~\ref{lem:hwccomp}, 
Lemma~\ref{lem: untri gamma} and Lemma~\ref{lem: exsitence of bsa for ujj},
we obtain  the following proposition.

\Prop\label{cor:basis}
An arbitrary sequence $\uii=\st{\im_k}_{k\in K}$ \condi
satisfies condition \eqref{eq: complete}.
\enprop
\Proof
By Lemma~\ref{lem:hwccomp}, Lemma~\ref{lem: untri gamma} and
Lemma~\ref{lem: exsitence of bsa for ujj},
$\uii$ satisfies \eqref{eq: complete} if $\ttb^\uii=\Updelta^n$ for some $n\ge0$ (see \S\;\ref{subsec:Gar}).
Hence it is enough to show that
\eqn
\text{if $\uii=\st{\im_k}_{k\in K}$ satisfies \eqref{eq: complete},
  then so does $\uii_{\le k}$ for any $k\in K$.}\label{eq:cut}
\eneqn
Let $M$ be a simple module in $\scrC_\g^{\bbD,\uii_{\le k}}$.
Since $M\in \scrC_\g^{\bbD,\uii}$, there exists $\bsa\in\Z_{\ge0}^{\oplus K}$.
such that $M\simeq\Vi(\bsa)$.
On the other hand, \eqref{eq:subP} says that
$M$ appears as a simple subquotient of
$\Pii(\bsb)$ for some $\bsb=(b_s)_{s\in K}$ such that $b_s=0$ for $s>k$.
Then Lemma~\ref{lem:subq} says that $\bsa\preceq \bsb$, which implies $\bsa\in\PBix[{[l,k]}]$.
\QED

\begin{corollary}
Let $\ttb \in \ttB^+$. Then $\scrC_\g^{\bbD,\uii}$ does not depend on the choice of $\uii \in \Seq(\ttb)$. We denote it by
$\scrC_\g^\bbD(\ttb)$.
\end{corollary}

The following corollary is an immediate consequence of Theorem~\ref{thm: unitri beta}. 
\begin{corollary} \label{cor: K simeq oA}
Let $\ttb \in \ttB^+$. Then, we have
\bnum
\item $K(\scrC_\g^{\bbD}(\ttb)) \simeq \obbA(\ttb)$.
\item $\scrC_\g^{\bbD}(\ttb)$ coincides with the full subcategory of $\Cg$ consisting of modules $M\in\Cgz$ such that $[M]\in \bPhi_\bbD(\hAz(\ttb))$.
\item \label{it: poly} $K(\scrC_\g^{\bbD}(\ttb))$ is the polynomial ring generated by  $\{ [C^{\uii}_s] \}_{1 \le s \le r}$ for any $\uii \in \Seq(\ttb)$. 
\ee
\end{corollary}

\section{Quantum Grothendieck rings and Bosonic extensions}\label{Sec: quantizability}
In this section, we develop an application of $T$-systems among affine determinantial modules associated with
a complete duality datum $\bbD$ and an arbitrary sequence $\uii$ and investigate the relationship with the $(q,t)$-characters of simple modules in $\Cg$ and
$\bbD$-quantizability. For this goal, we first review the quantum Grothendieck rings and their related subjects by following \cite{Nak04,VV03,Her04}. 

\subsection{Quantum Grothendieck rings} \label{subsec: quantum Grothendieck ring}
{\em In this subsection, we assume that  the quantum affine algebra
  $U_q'(\g)$ is of untwisted affine type}
and 
we fix a $\rmQ$-datum $\calQ$ of $\g$.
In \cite{FR99}, Frenkel-Reshetikhin constructed an injective ring homomorphism 
$$
\chi_q \colon K(\Cgz) \hookrightarrow \calY \seteq \Z[Y_{\oim,p}^{ \pm1 } \ | \ (\im,p) \in \hDynkin^\sigma_0].
$$
which is known as the \emph{$q$-character homomorphism}. 

Let $\calM \subset\calY$ be the set of all Laurent monomials. 
We write $m \in \calM$ as 
\begin{align*}
m = \prod_{(\im,p) \in \hDynkin^\sigma_0} Y_{\oim,p}^{u_{\oim,p}(m)}.
\end{align*}
We say an element $m \in \calM$ \emph{dominant} if $u_{\oim,p}(m) \ge 0$ for all $(\im,p) \in \hDynkin^\sigma_0$ and set $\calM^+ \subset \calM$ the set of all dominant monomials.

Recall that the isomorphism classes of simple modules 
in $\Cg$ are parameterized by the set $(1+z\bfk[z])^{I_0}$ of $I_0$-tuples of monic polynomials, called \emph{Drinfeld polynomials} \cite{CP95,CP95A}. 

For each $m \in \calM^+$, we have a simple module $L(m) \in\Cgz$ corresponding to the Drinfeld polynomial $(\prod_p (1-q^pz)^{u_{i,p}(m)})_{i\in I_0}$. Note that 
the fundamental module $L( \im,p)$ in~\S\ref{subsec: HL category} corresponds to $L(Y_{\oim,p})$,
and the trivial module $\one$ corresponds to $L(1)$.

For an indeterminate $t$ with a formal square root $t^{1/2}$, let $\calY_t$ be the quantum torus associated with $U_q'(\g)$, which
is a $\Z[t^{\pm 1/2}]$-algebra generated by $\{\tY^{\pm1}_{\oim,p} \ | \ (\im,p) \in \hDynkin^\sigma_0 \}$ with the following  relations:
\begin{align*}
\tY_{\oim,p}\tY^{-1}_{\oim,p}  = \tY^{-1}_{\oim,p}\tY_{\oim,p} =1 \qtq
\tY_{\oim,p}\tY_{\ojm,s}  = t^{\calN(\im,p;\jm,s)  }  \tY_{\ojm,s}\tY_{\oim,p}.
\end{align*}
Here 
$$
\calN(\im,p;\jm,s) \seteq (-1)^{k+l+\delta(p \ge s)} \delta\bl(\im,p)\ne(\jm,s)\br
\cdot(\al,\be) \in\Z, 
$$
where
$\phi_\calQ(\im,p)=(\al,k)$ and  $\phi_\calQ(\im,p)=(\be,l)$.  For monomials $m,m'$ in $\calY $, we define 
\begin{align} \label{eq: N(m m')}
\calN(m,m') \seteq \sum_{(\im,p),(\jm,s) \in\hDynkin^\sigma_0} u_{\oim,p}(m)  u_{\ojm,s}(m') \calN(\im,p;\jm,s).
\end{align}
For simple modules $X,Y$ in $\Cgz$, we set 
$$ \calN(X,Y) \seteq \calN(m,m') \quad \text{ where } X \simeq L(m) \text{ and } Y \simeq L(m').$$

Note that
\bnum
\item $\calY_t$ is a $t$-deformation of $\calY$
since there exists a $\Z$-algebra homomorphism 
$$\text{$\ev_{t=1}:\calY_t \twoheadrightarrow \calY$} \quad \text{given by $\ev_{t=1}(t^{1/2})=1$ and 
  $\ev_{t=1}( \tY_{\oim,p})=Y_{\oim,p}$,}$$
\item
there exists the \emph{bar-involution} $\overline{(\cdot)}$ on $\calY_t$ which is the $\Z$-algebra anti-involution fixing $\tY_{\oim,p}$ and sending $t^{1/2}$ to $t^{-1/2}$, and 
\item there exists a $\Z[t^{\pm1/2}]$-algebra automorphism $\ocalD$ (resp. $\ocalD_t$) 
of $\calY$ (resp. $\calY_t$) defined by 
\begin{align} \label{eq: frakD_t on calY}
\ocalD(Y_{\oim,p}) = Y_{\overline{\im^*},p+|\sigma| h^\vee} \quad \text{(resp. }  \ocalD_t(\tY_{\oim,p}) = \tY_{\overline{\im^*},p+|\sigma| h^\vee} \text{).}  
\end{align}
Here $|\sigma|$ is the order of $\sigma$ and $h^\vee$ is the dual Coxeter number of $\g_0$.
\ee
\smallskip

For each simple module $L(m) \in \Cgz$, there exists a unique bar-invariant element $L_t(m) \in \calY_{t}$, called the \emph{$(q,t)$-character}
of $L(m)$ and constructed by Kazhdan-Lusztig type algorithm. It was established by Nakajima~\cite{Nak01,Nak04} based on the geometry of quiver varieties
for simply-laced untwisted affine types, and then extended to all untwisted affine types by Hernandez~\cite{Her04} in an algebraic setting. 
 
For a simple module $M \in \Cgz$, we also use $[M]_t$ to denote the $(q,t)$-character of $M$. 

\smallskip 

The \emph{quantum Grothendieck ring} $\calK_{\g;t}$ is defined to be the $\Z[t^{\pm 1/2}]$-subalgebra of $\calY_t$
generated by $[M]_t$'s for all simple modules $M$'s in $\Cgz$.  

Note that $\calK_{\g;t}$ is stable under the bar-involution $\overline{(\cdot)}$ and 
\begin{align} \label{eq: Kt to K}
\ev_{t=1}(\calK_{\g;t})=\chi_q(K(\Cgz)) \simeq K(\Cgz).    
\end{align}

It is known that  
\begin{subequations} \label{eq: basis Ls}
\begin{align}
&\text{$\bfL_t \seteq \{ L_t(m) \ |  \ m \in \calM^+ \}$ forms a $\Z[t^{\pm1/2}]$-basis of
$\calK_{\g;t}$,} \label{eq: basis Lt}   \\
&\text{$\ev_{t=1}(\bfL_t) \seteq \{ \ev_{t=1}(L_t(m)) \ |  \ m \in \calM^+ \}$ forms a $\Z$-basis of
$\chi_q(K(\scrC_\g^0))\simeq K(\scrC_\g^0)$}   \label{eq: basis L}
\end{align} 
\end{subequations}
(see~\cite{FHOO,FHOO2} for~\eqref{eq: basis L} in non-simply laced types).

\begin{proposition} [{\cite{Nak04,VV02} and \cite{FHOO,FHOO2}}] \label{pro: positivity pf qt character}
For each $m \in \calM^+$, 
$$L_t(m) \in\Z_{\ge0}[t^{\pm1}]\;[\tY_{\oim,p}^{\pm1} \ | \ (\im,p) \in \hDynkin^\sigma_0].$$
Moreover, for $m_1,m_2 \in \calM^+$, if we write
$$
L_t(m_1)L_t(m_2) = \sum_{m \in \calM^+} c^m_{m_1,m_2}(t) L_t(m),  
$$
then we have
\bna
\item
 $c^m_{m_1,m_2}(t) \in \Z_{\ge 0}[t^{\pm1/2}]$.
\item  $c^m_{m_1,m_2}(t)=0$ unless $m\preceq m_1\,m_2$.
Here $\preceq$ is the Nakajima order on $\calM^+$ \ro\cite{Nak01,FM01}\rf.
\item If $m=m_1m_2$, then $c^m_{m_1,m_2}(1)=1$,
  i.e., $c^m_{m_1,m_2}(t)=t^a$ for some $a\in\Z/2$.
\ee 
\end{proposition}

\begin{theorem} [\cite{HL15,FHOO,KKOP24}] \label{thm: auto of calK}
  Recall that $\g$ is assumed to be of untwisted type.
  Let $\calQ=\Qdatum$ be a $\rmQ$-datum of $\g$. 
Then there exists a unique $\Z$-algebra isomorphism 
\begin{align} \label{eq: PhiQ}
\Psi_\DQ\col \hcalA_\bfA \isoto \calK_{\g;t}    
\end{align}
such that  
$\Psi_\DQ(f_{\im,m}) = \ocalD_t^m([L^\calQ_\im]_t)$  and  $\Psi_\DQ(q^{\pm 1/2}) = t^{\mp 1/2}.$
Moreover, it satisfies the following properties:
\bnum
\item $\Psi_\DQ \circ \bar{ \ }  = \bar{ \ } \circ \Psi_\DQ$ and $\Psi_\DQ \circ \ocalD_q  = \ocalD_t \circ \Psi_\DQ$. 
\item $\ev_{t=1} \circ \Psi_\DQ = \bPhi_\DQ$; i.e, we have a commutative diagram
$$
\xymatrix@R=2ex@C=6ex{  
\hcalA_\bfA \ar@{->>}[dd]_{\ev_{q=1}} \ar[rr]^{\sim}_{\Psi_\DQ} \ar[ddrr]|-{\;\bPhi_\DQ\;} &&  \calK_{\g;t} \ar@{->>}[dd]^{\ev_{t=1}} \\ \\
\obbA \ar[rr]^{\sim}_{\obPhi_\DQ} && K(\Cgz).  } 
$$
\item \label{it: G to L} $\Psi_\DQ$ sends the $\bfA$-basis $\tbfG \seteq \{ q^{-(\wt(\bfb),\wt(\bfb))/4}G(\bfb) \ | \ \bfb \in \hB(\infty) \}$ of $\hcalA_\bfA$ to the $\Z[t^{\pm 1/2}]$-basis $\bfL_t$ of $\calK_{\g;t}$. 
\ee
\end{theorem}

We call $\tbfG$ the \emph{normalized global basis} of $\hcalA$. 
Note that each element in $\tbfG$ is $\overline{\phantom{a}}$-invariant (see \cite[(5.10)]{KKOP24}). 
The map $\Psi_\DQ$ in~\eqref{eq: PhiQ} can be understood as a \emph{quantization} of $\bPhi_\DQ$. 

Note that $\Q(q^{1/2})\tens\Psi_\DQ^{-1}$ is denoted by $\Omega_\calQ$ in \cite{KKOP24}. 

\begin{definition} \label{def: quantizable}
We say that a simple module $M$ is \emph{quantizable} if 
$$[M]_t|_{t=1}\seteq \ev_{t=1}([M]_t)=\chi_q(M).$$     
\end{definition}

\begin{conjecture} [{cf.~\cite[Conjecture 7.3]{Her04}}] \label{conj: positivity conjecture}
  Every simple module is quantizable.
  \end{conjecture}

\begin{remark} \label{rmk: chi quantizable}
Conjecture~\ref{conj: positivity conjecture} is proved in \cite{Nak04} for the affine types $A_n^{(1)},D_n^{(1)},E_{6,7,8}^{(1)}$,
and in \cite{FHOO} for the affine type $B_n^{(1)}$.    When the affine type is of $C_n^{(1)},F_4^{(1)},G_2^{(1)}$ and
the simple module $M$ is \emph{reachable},  i.e., a cluster monomial module  
or contained in the heart subcategory $\scrC_\calQ$, 
Conjecture~\ref{conj: positivity conjecture} is proved in~\cite{HO19,FHOO,FHOO2}. 
However, Conjecture~\ref{conj: positivity conjecture} for the affine types
$C_n^{(1)},F_4^{(1)},G_2^{(1)}$ is still open for general
simple modules $M$. Note also that
it is proved in \cite{Her07} that any fundamental representation
is quantizable.

However it is known that (\cite{FM01,Nak04,Her04})
\eq
\ev_{t=1}(L_t(m))\in [L(m)]+\sum_{m'\prec m}\Z\,[L(m')],
\label{eq:LtL}
\eneq
where $\prec$ is the Nakajima order on $\calM^+$.  

\end{remark}

The following corollary is an immediate consequence of Theorem~\ref{thm: auto of calK}
 above and Proposition~\ref{pro: positivity pf qt character}.
\Cor\label{cor:gbpositive} 
For any $\bfb_1,\bfb_2\in \tbfG$,
we have
\eqn
\bfb_1\bfb_2=\sum_{\bfb \in\tbfG}c_{\bfb_1,\,\bfb_2}^{\;\bfb}(q)\,
\bfb
\eneqn
where $c_{\bfb_1,\,\bfb_2}^{\;\bfb}(q)\in\Z_{\ge0}[q^{\pm1/2}]$.
\encor
 
\Lemma \label{lem: chi quantizable} Let $\calQ$ be a $\rmQ$-datum.
A simple module $M$ is quantizable if and only if it is $\bbD_\calQ$-quantizable \ro see {\rm Definition~\ref{def: strong real}}\rf.
\enlemma

\Proof  Set $\bbD=\bbD_\calQ$.
It is obvious that a quantizable $M$ is $\bbD$-quantizable.

Let us show that a $\bbD$-quantizable simple module $M$ is quantizable.
By the assumption, there exists $\bfb\in\tbfG$ such that
$\bPhi_\bbD(\bfb)=[M]$.
On the other hand, Theorem~\ref{thm: auto of calK} implies that
$\Psi_{\bbD}(\bfb)=L_t(m)$ for some $m\in\calM^+$.

Take $m'\in\calM^+$ such that $M\simeq L(m')$.
Then Theorem~\ref{thm: auto of calK} implies that $\ev_{t=1}(L_t(m))=[L(m')]$.
Hence we conclude that $m=m'$ by \eqref{eq:LtL}.
Thus $M$ is quantizable. 
\QED
\begin{lemma} \label{lem: L up to}
For quantizable simple modules $L(m_1)$ and $L(m_2)$ in $\Cgz$ such that one of them is real, 
if $\de(L(m_1),L(m_2))=0$ and $L(m_1m_2)$ is quantizable, we have  
$$\calN(L(m_1),L(m_2)) = \La(L(m_1),L(m_2)) $$
and
$$ t^{-\calN(m_1,m_2)/2}L_t(m_1)L_t(m_2)= L_t(m_1m_2) = t^{\calN(m_1,m_2)/2} L_t(m_2)L_t(m_1).$$ 
\end{lemma}

\begin{proof}
By Proposition~\ref{pro: positivity pf qt character}, we have
$$L_t(m_1)L_t(m_2) = \sum_{m \in \calM^+} c^m_{m_1,m_2}(t) L_t(m) \quad \text{ in $\calK_{\g;t}$} $$ 
with $c^m_{m_1,m_2}(t) \in \Z_{\ge0}[t^{\pm 1/2}]$ and $L_t(m) \in\Z_{\ge0}[t^{\pm1}]\;[\tY_{\oim,p}^{\pm1} \ | \ (\im,p) \in \hDynkin^\sigma_0]$. 
From the assumptions, taking $\ev_{t=1}$ yields 
\begin{align*}
\ev_{t=1}(L_t(m_1m_2)) = \chi_q(L(m_1m_2)) = \chi_q(L(m_1))\chi_q(L(m_2)) & =  \sum_{m \in \calM^+} c^m_{m_1,m_2}(1) \ev_{t=1}(L_t(m)).
\end{align*}
Thus $c^m_{m_1,m_2}(t)=  \delta(m = m_1m_2) \cdot t^a$ for some $a \in \Z/2$ by~\eqref{eq: basis Ls} and Proposition~\ref{pro: positivity pf qt character}.
Similarly,  $c^m_{m_2,m_1}(t) =  \delta(m = m_1m_2) \cdot t^b$ for some $b \in \Z/2$. Thus 
$L_t(m_1)$ and $L_t(m_2)$ commute up to a power of $t^{\pm 1/2}$. Then the assertion follows from the leading terms  
of $L_t(m_i)$ $(i=1,2)$ and~\cite[Corollary 6.15]{FO21}. 
\end{proof}

\begin{remark} 
In \cite[Lemma 9.9 and Lemma 11.5]{FHOO}, Lemma~\ref{lem: L up to} is proved for (i) $\Cgz$ in types $\g=ABDE^{(1)}$, and (ii) for $\scrC_\calQ$ in any 
affine type $\g$ and its $\rmQ$-datum $\calQ$. 
In these cases, any simple module is quantizable.
\end{remark}

\subsection{Canonical complete duality datum}
Let $\sfg$ be a simply-laced finite-dimensional simple Lie algebra and 
let $\hcalA$ be the corresponding bosonic extension, and let $\tbfG$ be the normalized global basis of $\hcalA$.

Let $\g$ be of affine untwisted type $\sfg^{(1)}$ 
and $\calQ$  a $\rmQ$-datum of $\g$. We set
\eq\Dc\seteq\bbD_\calQ \label{def:Dcan}
\eneq
and call it a \emph{canonical complete duality datum} associated with $\calQ$.
Then every simple module in $\calC_{\sfg^{(1)}}^0$ is $\calQ$-quantizable.  Let $\uii$ be an arbitrary sequence in $\sfI$.
Under these choices, by Remark~\ref{rmk: chi quantizable},
there exists a unique element $\bfb[a,b]^\uii \in \tbfG$ such that 
$$
\Psi_\Dc(\bfb[a,b]^\uii) = [M^{\calQ,\uii}[a,b]]_t \quad \text{for any $i$-box $[a,b]$}. 
$$
Then we have
$\bPhi_{\Dc}(\bfb[a,b]^\uii) = [M^{\calQ,\uii}[a,b]]$ in $K(\Csgz)$
 and
\begin{align} \label{eq: T-system in K}
\bPhi_{\Dc}(\bfb[a^+,b]^\uii \bfb[a,b^+]^\uii) = \bPhi_\Dc(\bfb[a,b]^\uii \bfb[a^+,b^-]^\uii) + \prod_{d(\im_a,\jm)=1} \bPhi_\Dc(\bfb[a(\jm)^+,b(\jm)^-]^\uii),   
\end{align}
by Theorem~\ref{thm: Tsystem}.  
From~\eqref{eq: T-system in K}, we have
\begin{align} \label{eq: oT-system in K}
{}^\circ\bfb[a^+,b]^\uii \cdot {}^\circ\bfb[a,b^+]^\uii
= {}^\circ\bfb[a,b]^\uii \cdot {}^\circ\bfb[a^+,b^-]^\uii + \prod_{d(\im_a,\jm)=1} {}^\circ\bfb[a(\jm)^+,b(\jm)^-]^\uii \quad  \text{in $\obbA$,}
\end{align}
where ${}^\circ\bfb[a,b]^\uii = \ev_{q=1}(\bfb[a,b]^\uii)$. 

\medskip
The following theorem says that the above characterization of $\bfb[a,b]^\uii$
holds for an arbitrary choice of
a complete duality datum $\bbD$. 

\begin{theorem} \label{thm: D-quantizable for M[a,b]}
  Let $\g$ be an arbitrary affine Lie algebra, and
  $\bbD$ a complete duality datum in $\Cgz$, and $\uii$ an arbitrary
  sequence in $\sfI$. Then, we have
$$
\bPhi_\bbD(\bfb[a,b]^\uii) = [M^{\bbD,\uii}[a,b]] \quad \text{for any $i$-box $[a,b]^\uii$.}
$$
In particular, every affine determinantial module $M^{\bbD,\uii}[a,b]$  is $\bbD$-quantizable.   
\end{theorem}

\begin{proof}
  Note that $\bPhi_\bbD(\bfb[a,b]^\uii) = [M^{\bbD,\uii}[a,b]]$ when $a=b$.
  For $b > a$, let us apply an induction on $b-a$. 
Applying the isomorphism $\obPhi_\bbD$ in~\eqref{eq: obPhi} to~\eqref{eq: oT-system in K}, we have
$$
\bPhi_\bbD(\bfb[a^+,b]^\uii \bfb[a,b^+]^\uii) = \bPhi_\bbD(\bfb[a,b]^\uii \bfb[a^+,b^-]^\uii) + \prod_{d(\im_a,\jm)=1} \bPhi_\bbD(\bfb[a(\jm)^+,b(\jm)^-]^\uii).
$$
Since $\bPhi_\bbD(\bfb[a(\jm)^+,b(\jm)^-]^\uii) = [M^{\bbD,\uii}[a(\jm)^+,b(\jm)^-]]$, etc., we can conclude that
$$
\bPhi_\bbD(\bfb[a,b]^\uii) = [M^{\bbD,\uii}[a,b]]
$$
as desired. 
\end{proof}

\section{Quantum cluster algebras} \label{Sec: QCA}
In this section, we briefly recall quantum cluster algebras and cluster algebras, introduced by
Berenstein-Fomin-Zelevinsky in~\cite{FZ02,BZ05}.

\subsection{Quantum cluster algebras}
Let $t$ be an invertible indeterminate with a formal square root $t^{1/2}$. Let $\sfJ$ be a set of indices which can be countably infinite and is decomposed into
the set of exchangeable indices $\sfJ_\ex$ and the set of frozen indices $\sfJ_\fr$; i.e., $\sfJ = \sfJ_\ex \sqcup \sfJ_\fr$.
For a $\Z$-valued skew-symmetric $\sfJ \times \sfJ$-matrix $L=(L_{ij})_{i,j \in \sfJ}$, we define the \emph{quantum torus} $\sfT(L)$ associated with $L$ to be the
$\Z[t^{\pm 1/2}]$-algebra generated by $\{ \tX^{\pm 1}_j \}_{j \in \sfJ}$ subject to following relations:
\begin{align*}
\tX_j\tX_j^{-1}=\tX_j^{-1}\tX_j=1 \qtq \tX_i\tX_j = t^{L_{ij}} \tX_j\tX_i \quad\text{for $i,j \in \sfJ$.}    
\end{align*}
Note that
$\sfT(L)$ is an Ore domain and hence is embedded into its skew-field of fractions $\bbF(\sfT(L))$.
 
The quantum torus $\sfT(L)$ is equipped with a $\Z$-algebra anti-involution $\overline{(\cdot)}$, called the \emph{bar-involution}, defined by
$\overline{t^{\pm 1/2}} = t^{\mp 1/2}$ and $\overline{\tX_j} = \tX_j$  for all $j \in \sfJ$. 

For $\bsa =(a_k)_{k \in \sfJ}\in \Z^{\oplus \sfJ}$, we define the element $\tX^\bsa$ in $\sfT(L)$ as 
\begin{align}\label{eq: commutative monomial}
\tX^\bsa \seteq t^{1/2\sum_{i>j}a_{i}a_j L_{ij}} \rprod_{k \in \sfJ} \tX_k^{a_k}.  
\end{align}
Here we take a total order on the set $\sfJ$. Note that the element $\tX^\bsa$ does not depend on the choice of a total order on $\sfJ$ and is invariant under the bar-involution.   
It is well-known that $\{\tX^\bsa\}_{ \bsa \in \Z^{\oplus J}}$ forms a free $\Z[t^{\pm 1/2}]$-basis of $\sfT(L)$.

\begin{definition} \label{def: ex mat}
A $\Z$-valued $\sfJ \times \sfJ_\ex$-matrix $\tB=(b_{ij})_{i\in \sfJ,j\in \sfJ_\ex}$  is called an \emph{exchange matrix} if it satisfies the following properties: 
\ben
\item for each $j \in \sfJ_\ex$, there exist finitely many $i \in \sfJ$ such that $b_{ij} \ne 0$,
\item the principal part $B \seteq (b_{ij})_{i,j \in \sfJ_\ex}$ is skew-symmetric.
\ee
\end{definition}

\begin{definition}
Let $(L,\tB)$ be a pair of matrices defined above and $\sfT(L)=\Z[t^{\pm1/2}][\tX^{\pm1}_{k}]_{k \in \sfJ}$ its quantum torus.    
\bnum
\item We say that a pair $(L,\tB)$ is \emph{compatible} if we have
$\sum_{k \in\sfJ} L_{ki} b_{kj}   = 2\delta_{ij}$. 
\item We call the triple $\scrS_t = (\{ \tX_k \}_{k \in \sfJ}, L , \tB )$ a \emph{quantum seed} in the quantum torus $\sfT(L)$ and $\{ \tX_k \}_{k \in \sfJ}$ a \emph{quantum cluster}. 
\ee 
\end{definition}
For $k \in \sfJ_\ex$, the \emph{mutation} $\mu_k(L,\tB) \seteq (\mu_k(L),\mu_k(B))$ of a compatible pair $(L,\tB)$ \emph{in a direction $k$} is defined in a combinatorial way as follows:
\begin{align}
\mu_k(L)_{ij} = & \bc
  -L_{ij} - \displaystyle\sum _{b_{sk}<0  }  b_{sk}  L_{is} & \text{if} \ i \neq k, \ j= k, \\
  -L_{ij} + \displaystyle \sum _{b_{sk}>0}  b_{sk} L_{sj} \quad \  & \text{if} \ i=k, \ j\neq k, \\
   L_{ij} & \text{otherwise,}
\ec    \label{eq: mu L} \allowdisplaybreaks\\
\mu_k(\tB)_{ij} =&
\bc
  -b_{ij} & \text{if}  \ i=k \ \text{or} \ j=k, \\
  b_{ij} + (-1)^{\delta(b_{ik} < 0)} \max(b_{ik} b_{kj}, 0) & \text{otherwise.}
\ec \label{eq: mu B}
\end{align}

Note that (i) the pair $(\mu_k(L),\mu_k(B))$ is also compatible and (ii) the operation $\mu_k$ is an involution; i.e.,
$ \mu_k(\mu_k(L,\tB)) = (L,\tB)$.
We define the mutation of a quantum cluster $\{ \tX_i \}_{i \in \sfJ}$ at $k \in \sfJ_\ex$ as follows:
\begin{align} \label{eq: mut of cluster}
\mu_k(\tX_j) \seteq \bc
\tX^{\bfa'} + \tX^{\bfa''} & \text{ if } j =k, \\
\tX_j  & \text{ if } j \ne k,
\ec
\end{align}
where
$$
\bfa'_i = \bc
-1 & \text{ if } i =k, \\
\max(0,b_{ik}) & \text{ if } i \ne k,
\ec
\text{ and } \ \
\bfa''_i = \bc
-1 & \text{ if } i =k, \\
\max(0,-b_{ik}) & \text{ if } i \ne k.
\ec
$$
Then the \emph{mutation $\mu_k(\scrS_t)$ of the quantum seed $\scrS_t$ in a direction $k$} is defined to be the triple 
$\mu_k(\scrS_t) \seteq \bl\{ \tX_i\}_{i \ne k} \sqcup \{ \mu_k(\tX_k) \}, \mu_k(L), \mu_k(\tB)\br$.   

For a quantum seed $\scrS_t=( \{\tX_k\}_{k\in \sfJ},L,\tB)$,
an element in $\bbF(\sfT(L))$ is called a \emph{quantum cluster variable}
(resp.\ \emph{quantum cluster monomial}) if it is of the form
$$
\mu_{k_1}  \cdots \mu_{k_\ell} (\tX_j), \quad \text{$($resp.\ } \mu_{k_1}  \cdots \mu_{k_\ell} (\tX^\bfa) )
$$
for some finite sequence $(k_1,\ldots,k_\ell) \in \sfJ_\ex^\ell$ $(\ell \in \Z_{\ge 0})$ and $j \in \sfJ$ (resp.\ $\bfa \in \Z_{\ge 0}^\sfJ$). 
Note that each quantum cluster variable is bar-invariant.

For a quantum seed $\scrS_t = (\{ \tX_k \}_{k \in \sfJ}, L,\tB )$, the \emph{
  quantum cluster algebra}
$\scrA_t(\scrS_t)$ is the $\Z[t^{\pm 1/2}]$-subalgebra of $\bbF(\sfT(L))$ generated by all the quantum cluster variables. Note that   $\scrA_t(\scrS_t)\simeq\scrA_t(\bmu(\scrS_t))$ for any sequence $\bmu$ of mutations.

The \emph{quantum Laurent phenomenon}, proved by Berenstein-Zelevinsky in \cite{BZ05}, 
 says that the quantum cluster algebra $\scrA_t(\scrS_t)$ is indeed contained in $\sfT(L)$.

Let $\nu$ be an indeterminate with a formal square root $\nu^{1/2}$.
 We say that an  $\Z[\nu^{\pm 1/2}]$-algebra $R$ has a
{\em quantum cluster algebra structure}
if there exists a quantum seed $\scrS_t$ and a $\Z$-algebra isomorphism $\Omega: \scrA_t(\scrS_t) \overset{\sim}{\longrightarrow} R$ sending $t^{\pm 1/2}$ to $\nu^{\pm 1/2}$ or $\nu^{\mp 1/2}$.
In the case, a {\em quantum seed of $R$} refers to the image of a quantum seed in $\scrA_t(\scrS_t)$, which is obtained by a sequence of mutations.

\subsection{Cluster algebras} Let $\tB$ be an exchange matrix in Definition~\ref{def: ex mat}. 
Let us consider 
the (commutative) Laurent polynomials $\Z[X^{\pm1}_k\mid{k \in \sfJ}]$  
and $\Q(X_k \mid k \in \sfJ)$ the field of fraction of
$\Z[X^{\pm1}_k\mid{k \in \sfJ}]$, which can be understood
as specializations of $\sfT(L)$ and $\bbF(\sfT(L))$ at $t^{1/2}=1$, respectively. Then one can define 
(i) $X^\bfa$ for $\bfa \in \Z^{\oplus \sfJ}$ and (ii) $\mu_k(X_j)$ for $(j,k)\in \sfJ \times \sfJ_\ex$ by specializing at $t^{1/2}=1$ in the formulas in~\eqref{eq: commutative monomial}
and~\eqref{eq: mut of cluster}. 

We call the pair $\scrS = (\{ X_k \}_{k \in \sfJ},  \tB )$ a \emph{seed} in $\Z[X^{\pm1}_k\mid{k \in \sfJ}]$ and $\{ X_k \}_{k \in \sfJ}$ a \emph{cluster}. 
An element in $\Q(X_k \mid  k \in \sfJ)$ is called a \emph{cluster variable} (resp. \emph{cluster monomial}) if it is written as
$$
\mu_{k_1}  \cdots \mu_{k_\ell} (X_j), \quad \text{$($resp.\ } \mu_{k_1}  \cdots \mu_{k_\ell} (X^\bfa) )
$$
for some finite sequence $(k_1,\ldots,k_\ell) \in \sfJ_\ex^\ell$ $(\ell \in \Z_{\ge 0})$ and $j \in \sfJ$ (resp.\ $\bfa \in \Z_{\ge 0}^\sfJ$).

The \emph{ cluster algebra}
$\scrA(\scrS)$ is the $\Z$-subalgebra of $\Q(X_k \mid k \in \sfJ)$ generated by all the cluster variables. As in the quantum cluster algebra,
it is proved that $\scrA(\scrS)$ is contained in $\Z[X^{\pm1}_k\mid{k \in \sfJ}]$, which is referred to as the Laurent phenomenon \cite{FZ02}.   

Specializing at $t^{1/2}=1$, we obtain a surjective ring homomorphism 
$\ev_{t=1}\col \sfT(L) \twoheadrightarrow \Z[X^{\pm1}_k\mid{k \in \sfJ}]$. The $\ev_{t=1}$ induces the surjection
$ \scrA_t(\scrS_t) \twoheadrightarrow \scrA(\scrS)$,    
given by $\ev_{t=1}(t^{r}\tX_i)=X_i$ for all $i \in \sfJ$ and $r \in \Z/2$. 
This surjection maps the quantum cluster monomials of $\scrA_t(\scrS_t)$ to the
cluster monomials of $\scrA(\scrS)$ bijectively (see~\cite[Lemma A.4]{FHOO2} for more details). 
We sometimes write $\scrA(\tB)$ for $\scrA(\scrS)$ to emphasize $\tB$.

\section{Monoidal seeds and their mutations} \label{Sec: MS}

In this section, we first recall the definition and properties of monoidal seeds and monoidal categorification, mainly studied in~\cite{KKKO18,KKOP20,KKOP24A}.
Then we construct monoidal seeds associated with arbitrary sequences and investigate their properties. 
Throughout this section, we fix a complete duality datum $\bbD$ providing an isomorphism $\obPhi_\bbD\col\obbA \isoto K(\Cgz)$ in~\eqref{eq: obPhi},
and we frequently skip $_\bbD$ and $^\bbD$ in notations for simplicity.

\subsection{Monoidal seeds}
Let $\scrC$ be a full subcategory of $\Cgz$ containing the trivial module $\mathbf{1}$ and stable under taking tensor products, subquotients
and extensions. We denote by $K(\scrC)$ the Grothendieck ring of $\scrC$. 

\begin{definition}  \hfill 
\bnum
\item A \emph{monoidal seed in $\scrC$} is
a  quadruple $\sfS = (\{ \sfM_i\}_{i\in \sfJ },\tB; \sfJ,\sfJ_\ex)$
consisting of
 an index set $\sfJ$, an index set
 $\sfJ_\ex\subset \sfJ$ of exchangeable vertices
 , a  commuting family $\{ \sfM_i\}_{i\in\sfJ}$ of
real simple modules in $\scrC$, and an  integer-valued
$\sfJ\times\sfJ_\ex$-matrix $\tB =
(b_{ij})_{(i,j)\in\sfJ\times\sfJ_\ex}$ satisfying the conditions
in Definition~\ref{def: ex mat}.
\item We call $\{ \sfM_i\}_{i\in \sfJ }$ in a monoidal seed $\sfS$ in $\scrC$ \emph{a monoidal cluster}. 
\item For $i\in\sfJ$, we call $\sfM_i$ the $i$-th {\em cluster variable module} of $\sfS$.
\ee
\end{definition}
 
For a monoidal seed  $\sfS=(\{\sfM_i\}_{i\in\sfJ}, \tB;\sfJ,\sfJ_\ex)$,
let $\La^\sfS=(\La^\sfS_{ij})_{i,j\in\sfJ}$ be the skew-symmetric
matrix defined by $\La^\sfS_{ij}\seteq\Lambda(\sfM_i,\sfM_j)$.

\begin{definition} \label{def:admissible}
  We say that a monoidal seed $\sfS=\mseed$  in $\scrC$
  {\em admits a mutation in direction $k\in\sfJ_\ex$}
  if there exists a simple object $\sfM'_k$ of $\scrC$  such that
  \bna
  \item
  there is an exact sequence in $\scrC$
\begin{align} \label{eq: Mk'} 
0 \to  \dtens_{b_{ik} >0} \sfM_i^{\tens  b_{ik}} \to \sfM_k \otimes \sfM_k'
\to \dtens_{b_{ik} <0} \sfM_i^{\tens  (-b_{ik})} \to 0,
\end{align}
\item
$\sfM_k'$
commutes with $\sfM_i$ for any $i\in
\sfJ\setminus\{k\}$.
\ee
We say that $\sfS$ is {\em admissible} if it admits a mutation
in direction $k$ for every $k\in\sfJ_\ex$. 
\end{definition}
Note that $\sfM'_k$ is unique up to an isomorphism if it exists, since
$\sfM_k\hconv \sfM'_k\simeq \dtens_{b_{ik} <0} \sfM_i^{\tens  (-b_{ik})}$ (see~\cite[Corollary 3.7]{KKKO15}).

\begin{lemma} [{\cite[Lemma 7.4]{KKOP24A}}] 
  If $\sfS=\mseed$ 
  admits a mutation in direction $k\in\sfJ_\ex$, 
  then the simple module $\sfM'_k$ in~\eqref{eq: Mk'} is real and the quadruple 
\begin{align} \label{eq: mutation of mseed}
\text{$\mu_k(\sfS)
\seteq  ( \{\sfM_i\}_{i\neq k}\cup\{\sfM_k'\},\mu_k(\tB);\sfJ,\sfJ_\ex)$ is a monoidal seed in $\scrC$.} 
\end{align}
\end{lemma}
We call $\mu_k(\sfS)$ in~\eqref{eq: mutation of mseed} the \emph{mutation} of $\sfS$ in direction $k$.

\begin{proposition}[{\cite[Proposition 6.4]{KKOP20}}] \label{prop:condition simplified}
Let $\Mseed$ be an admissible
monoidal seed in $\scrC$. Let  $k\in\sfJ_\ex$ and let $\sfM'_k$ be as in
{\rm Definition~\ref{def:admissible}}. Then we have the following
properties. 
\begin{enumerate}[{\rm (i)}]
\item 
  For any $j \in \sfJ$, we have $(\La^{\mathsf{S}} \tB)_{j,k}=-2\delta_{j,k}\,  \de(\sfM_k,\sfM'_k)$.
\item For any $j\in\sfJ$, we have
\begin{equation} \label{eq: mu Lambda}
\begin{aligned} 
\La(\sfM_j,\sfM'_k)& =-\La(\sfM_j,\sfM_k)-\displaystyle\sum_{b_{ik}<0}\La(\sfM_j,\sfM_i)b_{ik},\\
\La(\sfM'_k,\sfM_j)& =-\La(\sfM_k,\sfM_j)+\displaystyle\sum_{b_{ik}>0}\La(\sfM_i,\sfM_j)b_{ik}.
\end{aligned}
\end{equation}
\ee
\end{proposition}

\begin{definition}
[{\cite[Definition 6.5]{KKOP20}}] Let  $\Mseed$ be a monoidal seed.
\bnum
\item
  Assume that $\Mseed$ admits a mutation in direction $k\in\sfJ_\ex$.
  We say that the mutation
   $\mu_k(\sfS)$ of $\sfS$ at $k\in \sfJ_\ex$ is a {\em $\Uplambda$-mutation}
   if  $\sfM'_k$ in~\eqref{eq: Mk'} satisfies
   $\de(\sfM_k,\sfM'_k)=1$.
   In this case, we say that $\Mseed$ admits
   a $\Uplambda$-mutation in direction $k\in\sfJ_\ex$.
\item We say that $\sfS$ is \emph{$\Uplambda$-admissible}
  if  $\sfS$ admits a $\Uplambda$-mutation in every direction $k\in\sfJ_\ex$,
\item
We say that a monoidal seed $\sfS$ is \emph{completely $\Uplambda$-admissible} if $\sfS$ admits
successive $\Uplambda$-mutations in all possible directions. 
\ee
\edf

\subsection{Monoidal categorification}  Let $\scrC$ be a full subcategory of $\Cgz$ containing the trivial module $\mathbf{1}$ and stable under taking tensor products, subquotients
and extensions.

\begin{definition} [Monoidal categorification]  
$\scrC$ is called a \emph{monoidal categorification} of a cluster algebra $\scrA$ if
\bnum
\item the Grothendieck ring $K(\scrC)$ is  isomorphic to $\scrA$,
\item there exists a completely $\Uplambda$-admissible monoidal seed $\Mseed$ in $\scrC$ such that
$ [\sfS] \seteq  (\{ [\sfM_i] \}_{i \in \sfJ}, \tB)$
is  a seed of $\scrA$.
\ee
\end{definition}

\begin{theorem} [{\cite[Theorem 6.10]{KKOP20}}] \label{nthm: main KKOP2}
  
Let $\Mseed$ be a $\Uplambda$-admissible monoidal seed in $\scrC$, and set $ [\sfS] \seteq  (\{ [\sfM_i] \}_{i \in \sfJ}, \tB)$. We assume that
the algebra $K(\scrC)$ is isomorphic to $\scrA([\sfS])$. Then we have
\ben
\item $\sfS$ is completely $\Uplambda$-admissible, and
\item $\scrC$ gives a monoidal categorification of $\scrA([\sfS])$. 
\ee
\end{theorem}

\medskip
 
A family of real simple modules $\{\sfM_i\}_{i\in \sfJ}$ in $\scrC$ is
called a {\em real commuting family in $\scrC$} if it
satisfies:
\ben
\item $\{\sfM_i\}_{i\in \sfJ}$ is mutually  commuting.
\ee
It is called  a {\em maximal real commuting family in $\scrC$} if it
satisfies further : \ben
\item[{\rm(2)}]
if a simple  module $X$ commutes with all the $M_i$'s, then
$X$ is isomorphic to $\bigotimes_{i \in \sfJ}M_i^{\tens a_i}$ for some
$\mathbf{a}=\{a_i\}_{i\in \sfJ} \in \Z_{\ge 0}^{\oplus \sfJ}$. 
\ee

\begin{corollary}[{\cite[Corollary 6.11]{KKOP20}}] 
Let  $\Mseed$ be a
$\Uplambda$-admissible  monoidal
seed in $\scrC$ and assume  that the algebra $K(\scrC)$ is isomorphic
to $\scrA([\sfS])$. Then  the following statements hold{\rm:} 
\bnum
  \item Any cluster monomial in $K(\scrC)$  is the isomorphism class of a  real simple object in $\scrC$.
  \item   
    The isomorphism class of an arbitrary simple module in
    $\scrC$ is a Laurent polynomial of the initial cluster variables with coefficient in $\Z_{\ge0}$.
\item  Any monoidal cluster $\{\widetilde{M}_i\}_{i\in\sfJ}$ is a maximal real commuting family.
\ee
\end{corollary}

We call the real simple module corresponding to a cluster monomial of $\scrA([\sfS])$ in Theorem~\ref{nthm: main KKOP2} a \emph{cluster monomial module}.

\subsection{Monoidal seeds and admissible chains of $i$-boxes}
In this subsection, we review the properties of monoidal seeds related to weights and admissible chains of $i$-boxes, which are mainly investigated in~\cite{KKOP22,KKOP24A}. 

\begin{proposition} 
[{\cite[Proposition 7.13]{KKOP24A}}]
Let  $\Mseed$ be
a  $\Uplambda$-admissible monoidal seed  in
$\Cgz$  and let $k
\in \sfJ_\ex$. Assume that 
\bnum
\item
$\sfJ$ is a finite set and $\dim\big(\sum_{i\in\sfJ}\Q\wt_\calQ(\sf M_i)\big)\ge
|\sfJ_\fr|$,
\item there exist a real simple module $X \in \Cgz$
and an exact sequence
\begin{align*}
0 \to A \to \sfM_k \tens X \to B \to 0,
\end{align*}
such that \bna
\item $\de(X,\sfM_j)=0$ for all $j \in \sfJ \setminus \{k\}$ and $\de(X,\sfM_k)=1$,
\item $A=\dtens_{i \in \sfJ} \sfM_i^{\tens m_i}$,
$B=\dtens_{i \in \sfJ} M_i^{\tens n_i}$ for some
$m_i,n_i\in\Z_{\ge0}$. \ee \ee Then we have $b_{ik}=m_i-n_i$.

If we have furthermore $m_in_i=0$ for all $i \in \sfJ$, then  we have
$$X \simeq \sfM'_k,$$
where $\sfM'_k$ is given in {\rm Definition~\ref{def:admissible}}. 
\end{proposition}

\begin{proposition} [{\cite[Proposition 7.14]{KKOP24A}}] 
Let 
$$\Mpseed \qtq \Mseed$$ 
be two
$\Uplambda$-admissible  monoidal seeds in $\Cgz$ such
that $\sfJ^*\subset \sfJ$ and $\sfJ_\ex^*\subset\sfJ_\ex$. Assume that $\sfJ$ is a
finite set and $\dim\big(\sum_{i\in\sfJ}\Q\wt_\calQ(\sfM_i)\big)\ge |\sfJ_\fr|$. Then
$$\tB\vert_{\sfJ^*\times\sfJ_\ex^*}=\tB^* \quad \text{ and }\quad
\tB\vert_{(\sfJ\setminus\sfJ^*)\times\sfJ_\ex^*}=\mathbf{0}.$$
\end{proposition}

\begin{lemma} [{\cite[Lemma 7.15]{KKOP24A}}] \label{lem:restrseed} Let $\Mseed$
be a monoidal seed in $\Cgz$. Let $\sfJ^*$ be a subset of $\sfJ$ with a
decomposition $\sfJ^*=\sfJ_\ex^* \sqcup\sfJ^*_\fr$ such that $\sfJ_\ex^* \subset
\sfJ_\ex$. Set
$$\sfS \vert_{(\sfJ^*,\,\sfJ_\ex^*)}\seteq
\big( \{\sfM_i\}_{i\in\sfJ^*},\,\tB\vert_{(\sfJ^*)\times\sfJ_\ex^*};\sfJ^*,\,\sfJ^*_\ex \big).$$
Assume that 
\begin{align*} 
\text{$b_{ij} =0$ if $i\in\sfJ\setminus\sfJ^*$ and $j\in\sfJ_\ex^*$.}     
\end{align*}
 Then, we have \bnum
\item
$\big(\mu_s(\tB)\big)_{ij }=0$ if $s\in\sfJ_\ex^*$, $i\in\sfJ\setminus \sfJ^*$
and $j\in\sfJ_\ex^*$,
\item if $\Mseed$ is $\Uplambda$-admissible, then we have
\begin{align*}
\big(\mu_s\sfS \big)\vert_{(\sfJ^*,\sfJ^*_\ex)} =\bc
\mu_s(\sfS \vert_{(\sfJ^*,\,\sfJ^*_\ex)})&\text{if $s\in\sfJ^*_\ex$,}\\
\sfS\vert_{(\sfJ^*,\,\sfJ^*_\ex)}&\text{if $s\in\sfJ\setminus \sfJ^*$.}
\ec
\end{align*} 
\ee In particular, if $\Mseed$ is a completely $\Uplambda$-admissible monoidal seed in
$\Cgz$, then so is $\sfS\vert_{(\sfJ^*,\,\sfJ^*_\ex)}$. 
\end{lemma}

In the rest of this section, we  take
\bnum \item an arbitrary sequence $\uii=
\st{\im_k}_{k\in K}$ in $\sfI$,  where $K$ is an interval in $\Z$
such that $K\cap\st{0,1}\not=\emptyset$, and
\item a complete duality datum $\bbD$ in $\Cgz$.
  \ee

Let $\frakC=(\frakc_k)_{1\le k\le r}$ be an admissible chain of $i$-boxes associated with $\uii$ with range $[a,b]\subset K$, $r \ge 1$. Hence $b-a+1=r$.
We define
\begin{equation} \label{eq: frakC seed}
\begin{aligned}
\sfJ(\frakC) &\seteq [1,r], \\
\sfJ(\frakC)_\fr &\seteq \{ s \in \sfJ(\frakC) \ | \ \frakc_s = [a(\im)^+,b(\im)^-] \text{ for some } \im \in \sfI \}, \\
\sfJ(\frakC)_\ex &\seteq \sfJ(\frakC) \setminus \sfJ(\frakC)_\fr, \\
\sfM^\bbD(\frakC)  & \seteq \{ M^{\bbD,\,\uii}(\frakc_k) \}_{k \in \sfJ(\frakC)}.  
\end{aligned}
\end{equation}
Here $M^{\bbD,\,\uii}(\frakc_k)  \seteq M^{\bbD,\uii}[u_k,v_k]$ in~\eqref{eq: MiD} where $\frakc_k = [u_k,v_k]$.
Note that $\sfM^\bbD(\frakC)$ is a commuting family of real simple modules by
Theorem~\ref{thm: every sequence good}.
When we need to emphasize the range of $\frakC$ and the sequence $\uii$, we write $\frakC^{[a,b],\uii}$ for $\frakC$.
We sometimes drop $\bbD$ if there is no afraid of confusion.

The following lemma is an $\uii$-analogue of \cite[Lemma 7.17]{KKOP24A}, which tells that box-moves corresponds to mutations.
Since the proof is similar with the help of Theorem~\ref{thm: every sequence good},
we omit it.
\begin{lemma} \label{lem: b is mu}
Let $\mathfrak{C}=(\frakc_k)_{1\le k\le r}$ be an admissible chain of
$i$-boxes associated with $\uii$ and a finite range
such that
$S(\frakC) \seteq\big( \sfM(\frakC),\tB;\sfJ(\frakC),\sfJ(\frakC)_\ex\big)$ is a
 $\Uplambda$-admissible monoidal seed in
 $\Cgz$ for some exchange matrix $\tB$.
If $k_0\in \sfJ(\frakC)_\ex$
and $\frakc_{k_0}$ is a movable $i$-box such that
$\tfrakc_{k_0+1}=\frakc_{k_0+1}=[u,v]$, then we have
\eqn
\mu_{k_0}\bl(S(\frakC))=S\bl\bbB_{k_0}(\frakC))\br
&& =\big( \sfM(\bbB_{k_0}(\frakC)),\mu_{k_0}(\tB);
\sfJ(\bbB_{k_0}(\frakC)),\sfJ(\bbB_{k_0}(\frakC))_\ex\big) \\
&&=  \big( \{ \sfM_i\}_{i\in \sfJ \setminus   \{ k_0 \}   } \sqcup\{ \sfM_{k_0}'  \},
\mu_{k_0}(\tB) ; \sfJ(\frakC),\sfJ(\frakC)_\ex \big),
\eneqn
where
$$ \sfM'_{k_0} \seteq \begin{cases}
\sfM^\uii[u,v^-]  & \text{if $\frakc_k=[u^+,v]$,} \\
\sfM^\uii[u^+,v]  & \text{if $\frakc_k=[u,v^-]$,}
\end{cases}
\qtq \text{$\sfM_k \seteq M^\uii(\frakc_k)$ for $k \in \sfJ$.}$$
Thus the box move $\bbB_{k_0}$ at $k_0$ in {\rm Definition~\ref{def: movable and B-move}~\eqref{it: B-move}} 
corresponds to the mutation $\mu_{k_0}$ at $k_0$ and T-system in~\eqref{eq: Tsystem2}. 
\end{lemma}

\begin{corollary}[{\cite[Corollary 7.18]{KKOP24A}}] \label{cor: same range}
For a finite interval $[a,b] \subset K$, let $\frakC$ and $\frakC'$ be admissible chains of $i$-boxes associated with $\uii$
and the same range $[a,b]$. Assume the monoidal seed $\sfS(\frakC)= \big( \sfM(\frakC),\tB;  \sfJ(\frakC), \sfJ(\frakC)_\ex \big)$
is a completely $\Uplambda$-admissible in $\Cgz$
for some exchange matrix $\tB$. 
Then, the monoidal seed $\sfS(\frakC')= \big( \sfM(\frakC'),\tB';  \sfJ(\frakC'), \sfJ(\frakC')_\ex \big)$
is also a completely $\Uplambda$-admissible in $\Cgz$
for some exchange matrix $\tB'$.
\end{corollary}
 
\subsection{A construction of $\Uplambda$-admissible monoidal seeds}
\label{subsec: Construction}
Take a finite interval $\sfJ=[1,r] \subset K$
and we denote by $\frakC^{\sfJ,\uii}_+$ the admissible chain of $i$-boxes associated with $(1,(\calR,\calR,\ldots,\calR))$, i.e., $\frakC^{\sfJ,\uii}_+=\st{\frakc_k}_{k\in\sfJ]}$
with $\frakc_k=\{1,k]$ for $k\in \sfJ$.

Take $\sfJ_\fr = \{k \in \sfJ \mid k^+ >r \}$ and $\sfJ_\ex \seteq \sfJ \setminus \sfJ_\fr$.  
Let $\tB^{\sfJ,\uii}\seteq\tB(\frakC^{\sfJ,\uii}_+)
=(b^{\sfJ,\uii}_{s\,t})_{s \in \sfJ,t \in \sfJ_\ex}$ be an exchange matrix defined as follows:
\begin{align} \label{eq: tBii}
b^{\sfJ,\uii}_{s\,t} \seteq \bc
1 &\text{ {\rm (i)} if } s  < t  < s^+ < t^+ \text{ and }  d(\im_s,\im_t) = 1, \text{ or {\rm (ii)} } s=t^+,\\
-1 &\text{ {\rm (i$'$)} if }  t  < s < t^+ < s^+ \text{ and }  d(\im_s,\im_t) = 1, \text{ or {\rm (ii$'$)} } t=s^+, \\
0 &\text{ otherwise}. 
\ec
\end{align}

We set
\eqn
&& \La^{\sfJ,\uii}_{s,t}\seteq \La(M^\uii\{ a,s],M^\uii\{ a,t])
\eneqn
which satisfies
$$ \La^{\sfJ,\uii}_{s,t}=  -(\varpi_{\im_{s}}- w^\uii_{\le s}\varpi_{\im_{s}},\varpi_{\im_{t}}+ w^\uii_{\le t}\varpi_{\im_{t}}) \quad \text{for $s, t\in\sfJ $ such that $s\le t$.}$$
We frequently drop $^{\sfJ}$  in notations for simplicity.  

\smallskip

\begin{proposition} [{\cite[\S 1,2]{FHOO2}}] \label{prop: mutation and compatible}
Let $\uii$ be any sequence in $\sfI$. 
\bnum
\item The pair $(\La^\uii,\tB^\uii)$ is compatible. 

\item For a sequence $\ujj$ such that $\upga_k\uii=\ujj$, we have 
$$\tB^\ujj = \upsigma_k \tB^\uii \qtq \La^\ujj =  \upsigma_k \La^\uii.$$
\item For a sequence $\ujj$ such that $\upbe_k\uii=\ujj$, we have 
$$\tB^\ujj = \upsigma_{k+1} \mu_k \tB^\uii \qtq \La^\ujj =  \upsigma_{k+1} \mu_k\La^\uii.$$ 
\ee
\end{proposition}

The following theorem is a main result of this subsection and can be understood as a vast generalization of \cite[Theorem 7.20]{KKOP24A} to \emph{arbitrary sequences}. 

\begin{theorem} \label{thm: admissible seed}
For an arbitrary sequence $\uii$ in $\sfI$, the monoidal seed in $\Cgz$
\begin{align}
     ( \{ M^\uii\{  a,s] \}_{s \in \sfJ}, \tB^\uii; \sfJ,\sfJ_\ex )  \text{ is $\Uplambda$-admissible.}
\end{align}

\end{theorem}

Note that $\{ M^\uii\{  a,s] \}_{s \in [a,b]}= \sfM(\frakC^\uii_+)$. 
Since the proof of the theorem above is similar to the one of
  \cite[Theorem 7.20]{KKOP24A}, we omit the proof.

\section{Monoidal categorification and quantum cluster algebra structure} \label{Sec: MC for qca}
In this section, we will prove our theorems on monoidal categorification. We begin by showing that the category 
$\Cg(\ttb)$ provides a monoidal categorification of a cluster algebra. 
Then we will show that the algebra $\hcalA_\Zq(\ttb)$ has a \emph{quantum} cluster algebra structure by using the monoidal categorification.

\subsection{Monoidal categorification of a cluster algebra}  Let $\scrC$ be a full subcategory of $\Cg$ containing trivial module $\bone$
and stable under taking tensor products, subquotients and extensions.

Recall the definition of $\Cg^{[a,b],\bbD,\uii}$, etc.\ in Definition~\ref{def:affdet} for a complete duality datum $\bbD$. 

\begin{theorem} [{\cite[Theorem 8.1]{KKOP24A}}] \label{thm: main KKOP2}
Let $(\bbD_\calQ,\hwc)$ be a \emph{PBW-pair} of a $\rmQ$-datum $\calQ$ of $\g$ and let $\frakC$ be an admissible chain of $i$-boxes with range $[a,b]$  for $-\infty \le a \le b \le \infty$.  
Then we have
\bna
\item $\SC= \big( \sfM(\frakC),\tB;  \sfJ(\frakC), \sfJ(\frakC)_\ex \big)$
  is a completely $\Uplambda$-admissible monoidal seed in the category $\Cg^{[a,b],\bbD_\calQ,\hwc}$ for some exchange matrix $\tB$.
\item $\scrA(\SC) \simeq K(\scrC_\g^{[a,b],\bbD_\calQ,\hwc})$,
  where $\scrA(\SC)$ is the cluster algebra associated
  with the seed $[\SC]\seteq\big(\st{X_j}_{j\in\sfJ(\frakC)},\tB;  \sfJ(\frakC), \sfJ(\frakC)_\ex \big)$. 
\ee
Namely, the category $\Cg^{[a,b],\bbD_\calQ,\hwc}$ provides a monoidal categorification of the cluster algebra $K(\Cg^{[a,b],\bbD_\calQ,\hwc})$
with the initial monoidal seed $\sfS(\frakC)$. 
\end{theorem}

\smallskip
We fix a complete duality datum $\bbD = \{ L_\im \}_{\iinI}$ throughout this subsection. 
For simplicity of notation, let us take $\uii \in \sfI^{\Z_{>0}}$, and set $\sfM_m^\uii \seteq M^\uii\{ 1, m]$ for all $m \in \Z_{> 0}$. 
Recall the operations $\upga_k$
and $\upbe_k$ in Definition~\ref{def: moves}, which are
defined on the sequences in $\sfI$.

\begin{proposition} \label{prop: mutation}
Let $\ujj = (\jm_1,\jm_2,\ldots) \in \sfI^{\Z_{> 0}}$.
\bnum
\item \label{it: comm mutation} If $\ujj=\upga_k(\uii)$, then we have
$\sfM_m^\ujj \simeq \sfM_{\upsigma_k(m)}^\uii$ for all $m \in \Z_{> 0}$. 
\item \label{it: braid mutation} Let $\ujj=\upbe_k(\uii)$.
Assume that the monoidal seed
$\sfS(\frakC^{\bbD,\uii}_{+})$  
is a completely $\Uplambda$-admissible seed in the $\Cg^{[1,\infty],\bbD,\uii}$
Then  we have
\begin{align} \label{eq: mutation of C}
\text{
$\sfM(\frakC^\ujj_+) = \upsigma_{k+1}\mu_k \big(\sfM(\frakC^\uii_+) \big)$.} \quad \text{ Namely, 
$\sfM_m^\ujj \simeq \bc 
(\sfM^\uii_k)' & \text{ if } m=k, \\
\sfM^\uii_{k+2} & \text{ if } m=k+1, \\
\sfM^\uii_{k+1} & \text{ if } m=k+2, \\
\sfM_m^\uii & \text{ otherwise}.
\ec$}
\end{align}
Here $(\sfM^\uii_k)'$ denotes the mutation of $\sfM^\uii_k$ at $k$ described in~\eqref{eq: Mk'}.  
\ee
\end{proposition}

\begin{proof}
\eqref{it: comm mutation} Since $\{ \TT_\im \}_{\im\in \sfI}$ satisfies the relations in the braid group, we have $C^\uii_k \tens C^\uii_{k+1} \simeq  C^\uii_{k+1} \tens C^\uii_{k}$, 
$C^\uii_k \simeq C^\ujj_{k+1}$,
$C^\uii_{k+1} \simeq C^\ujj_k$ and $C^\uii_{m} \simeq C^\ujj_m$ for $m \not\in \{ k,k+1\}$. Then the assertion follows from the definition of 
$\sfM_m^\ujj \seteq M^{ \ujj }\{1, m]$. 

\mnoi
\eqref{it: braid mutation}
By Lemma~\ref{lem:lneg}, we can assume that $k \ge 0$ without loss of generality. 
Note that 
\begin{align*}
C_{k+1}^\ujj \simeq  C^\uii_{k+2} \hconv C^\uii_{k} & 
\simeq  \TT_{\im_1}\TT_{\im_2}\cdots \TT_{\im_{k-1}}( \TT_{\im_k}\TT_{\im_{k+1}} L_{\im_{k+2}} \hconv  L_{\im_{k}}) \\
 & \simeq  \TT_{\im_1}\TT_{\im_2}\cdots \TT_{\im_{k-1}}(L_{\im_{k+1}} \hconv  L_{\im_{k}})  \simeq \TT_{\jm_1}\TT_{\jm_2}\cdots \TT_{\jm_{k}}(L_{\jm_{k+1}}) = C_{k+1}^\ujj, 
\end{align*}
by Proposition~\ref{prop: ell de},   
$C^\uii_{m} \simeq C^\ujj_{m}$ for $m \not\in [k,k+2]$ and $C^\uii_{a} \simeq C^\ujj_{b}$
for $\{a,b\}=\{k,k+2\}$.
 
From Proposition~\ref{prop: still good sequence}, the sequence 
$\uC^\ujj = ( C^\ujj_r,\ldots,C^\ujj_1  )$ satisfies the same properties in Condition~\ref{assu: the assumption}. Then we have $\sfM^\ujj\{ 1,m]$ from 
$\uC^\ujj$ via an $i$-box $\{ 1,m ]$ in the usual way. Then by definition of $\sfM_{m}^\uii \simeq \sfM^\ujj\{ 1,m]$ for $m<k$, 
$\sfM^\uii_{k+2} \simeq \sfM^\ujj\{ 1,k+1]$ and $\sfM^\uii_{k+1} \simeq \sfM^{\ujj}\{ 1,k+2]$.  

Let us prove $(\sfM^\uii_k)' \simeq \sfM^{ \ujj }\{ 1,k]$. Note that Theorem~\ref{thm: admissible seed} says that
$$
(\sfM^\uii_k)' = C^\uii_{k+2} \hconv (\sfM^\uii_k)^{{\rm Vi}},  
$$
where $$(\sfM^\uii_k)^{{\rm Vi}} \seteq \dtens_{ \{ t   \ | \ d(\im,\im_t)=1 \text{ and } t < k < t^+ < k^+ \}  }  \sfM^\uii_t.$$

Since $\im_k=\im_{k+2}$ and $d(\im_k,\im_{k+1})=1$, $(\sfM^\uii_k)^{{\rm Vi}} \simeq \sfM^\uii_{k(\jm_{k})^-} = \sfM^{\uii}\{ 1,k(\jm_{k})^-] \simeq \sfM^\ujj\{1, k(\jm_{k})^-]$.
Hence 
$$ 
(\sfM^\uii_k)' \simeq C^\ujj_{k}\hconv \sfM^\ujj\{1, k(\jm_{k})^-] \simeq \sfM^\ujj\{1, k]. 
$$
Then the set of cluster variable modules of $\mu_k(\sfS(\frakC^\uii_+))$ coincides with $\{ \sfM^{\ujj}\{1,m ]\}_{m \in \Z_{> 0}}$. 
Thus the assertion follows.
\end{proof}

The following proposition is proved in \cite[Proposition 7.19]{KKOP24A}
when $\uii$ is a locally reduced sequence, but the same proof works for an arbitrary $\uii$.  
\begin{proposition} [{\cite[Proposition 7.19]{KKOP24A}}]
Let $\Mseed$ be a monoidal seed in $\Cg(\ttb)$. If  $\Mseed$ is $($completely$)$ $\Uplambda$-admissible in $\Cgz$,
then it is $($completely$)$ $\Uplambda$-admissible in $\Cg(\ttb)$.
\end{proposition}

Recall the admissible chain of $i$-boxes $\frakC^{\uii}_+ $ which is associated with $(1,(\calR,\calR,\ldots))$  
(\S~\ref{subsec: Construction}).  

\medskip

For  $\ttb\in \ttB^+$ and a complete duality datum 
$\bbD$, recall the subcategory $\Cg(\ttb)$ whose $K(\Cg(\ttb))$
is isomorphic to the commutative $\Z$-algebra $\obbA(\ttb)$. According to Corollary \ref{cor: K simeq oA} (\ref{it: poly}), this algebra is
the polynomial ring generated by 
$\{ [C_s^\uii]\}_{s\in[1,r]}$. The following theorem states that for any complete duality datum 
$\bbD$, the category $\Cg(\ttb)$ provides a monoidal categorification of the cluster algebra that is isomorphic to 
$\obbA(\ttb)$, thereby confirming \cite[Conjecture 8.13]{KKOP24A}.

\begin{theorem} \label{thm: mCat of cluster}
For any complete duality datum $\bbD$
and $\uii=(\im_1,\ldots,\im_r) \in \Seq(\ttb)$, let $\frakC^\uii$ be an admissible chain of $i$-boxes associated with $\bbD$, $\uii$ and  
 a range $[1,r]$. Then we have the followings:
\bna
\item \label{it: CLable}  $\sfS(\frakC^\uii)$ is a completely $\Uplambda$-admissible seed in $\Cg^0(\ttb)$.
\item \label{it: isoClsuter} $\scrA( \frakC^\uii ) \simeq  \obbA(\ttb)  \simeq   K(\Cg(\ttb))$. 
\ee
Namely, the category $\Cg(\ttb)$ provides a monoidal categorification of the cluster algebra $K(\Cg(\ttb))$ with the initial monoidal 
seed $\sfS(\frakC^\uii)$, which is isomorphic to $\obbA(\ttb)$. 
\end{theorem}

\begin{proof}
Theorem~\ref{thm: main KKOP2} says that we have an isomorphism 
\begin{align} \label{eq: obbAge0 cl}
 \Upsilon_{\ge 0} :   \scrA(\tB(\frakC^{[1,\infty],\hwc}_+)) \isoto K(\Cg^{[1,\infty],\bbD_\can, \hwc})\isoto  \obbA_{\ge 0}    
\end{align}
as rings, such that $ \Upsilon_{\ge 0}(X_k)=\obPhi_{\bbD_\can}^{-1}([M^{\bbD_\can, \hwc} \{ 1,k]]) = {}^\circ \bfb  \{1,k]^{\hwc}$
for all $1\le k$. 
Hence the composition $\obPhi_\bbD \circ  \Upsilon_{\ge 0} \colon  \scrA(\tB(\frakC^{[1,\infty],\hwc}_+)) \isoto K(\Cg^{[1,\infty],\bbD, \hwc})$ is an isomorphism of rings sending $X_k$ to $[M^{\bbD, \hwc} \{ 1,k]]$ for all $1\le k$ by Theorem~\ref{thm: D-quantizable for M[a,b]}.  Hence   
by Theorem~\ref{thm: admissible seed} and  Theorem~\ref{thm: main KKOP2},  $\sfS^\bbD(\frakC^{[1,\infty],\hwc}_+)$ is a completely $\Uplambda$-admissible seed in $\Cgz$.

Let us take $\tuii \in \Seq(\Updelta^m)$ such that $\uii=\tuii_{[1,r]}$
as in Remark~\ref{rem: conlcusion loc} and set $\tuii' \seteq  \tuii * \hwci{[m\ell+1,\infty]} \in I^{\Z_{> 0}}$, where $*$ denotes a concatenation of sequences. 
Then $\tuii'$ can be obtained from $\hwci{[1,\infty]}$ by applying finite commutation moves and braid moves. Since the monoidal seed $\sfS^\bbD(\frakC^{[1,\infty],\hwc}_+)$
is a completely $\Uplambda$-admissible seed, $\sfS^\bbD(\frakC_+^{[1,\infty],\tuii'})$ is a completely $\Uplambda$-admissible seed
in $\Cgz$ by Proposition~\ref{prop: mutation}. By setting $\sfJ^*=[1,r]$, Lemma~\ref{lem:restrseed} says that $\sfS^\bbD(\frakC_+^{[1,r],\uii}) \seteq \sfS^\bbD(\frakC_+^{[1,\infty],\tuii'})|_{(\sfJ^* \times \sfJ^*_\ex)}$ is a completely $\Uplambda$-admissible seed in $\Cgz$. 
Hence $\sfS^\bbD(\frakC_+^{[1,r],\uii})$ is a completely $\Uplambda$-admissible
seed in $\Cg(\ttb)$ by \cite[Proposition 7.19]{KKOP24A}.  
Then Corollary~\ref{cor: same range} implies that,
for any admissible chain $\frakC^{[1,r],\uii}$ of $i$-boxes associated with $\uii$ and a range $[1,r]$, 
there exists an exchange matrix $\tB$ such that the monoidal seed
$\big( \sfM(\frakC^{[1,r],\uii}), \tB;  \sfJ(\frakC^{[1,r],\uii}) , \sfJ(\frakC^{[1,r],\uii})_\ex \big)$
is a completely $\Uplambda$-admissible seed in $\Cg(\ttb)$.

Since $\sfS^\bbD(\frakC^{[1,r],\uii})$ is a completely $\Uplambda$-admissible seed in $\Cg(\ttb)$, the image of  each cluster monomial
of $\scrA(\tB(\frakC^{[1,r],\uii}))$  under $\Upsilon_{\ge 0}$   is contained in $K(\CgD(\ttb)) \simeq \obbA(\ttb)$; i.e, we have 
$$
\scrA(\tB(\frakC^{[1,r],\uii}))\hookrightarrow K(\CgD(\ttb)) \simeq \obbA(\ttb).
$$
For any $s\in[1,r]$, after successive box moves, the moved $\frakC^{[1,r],\uii}$
contains $\{[s]\}$ by Lemma~\ref{lem: finite sequence are T-equi}. Hence, the image of  $\scrA(\frakC^{[1,r],\uii})$ contains
$[C_s^\uii]$. Since $K(\CgD(\ttb))$ is the polynomial ring
with the system of generators $\{ [C_s^\uii]\}_{s\in[1,r]}$,
 we have $\scrA(\tB(\frakC^{[1,r],\uii})) \to K(\CgD(\ttb))$ is surjective.  Thus our assertion is completed. 
\end{proof}

Recall that
there exists a unique normalized global basis element $\bfb[a,b]^\uii \in \tbfG$ such that $\Phi_{\bbD}(\bfb[a,b]^\uii)=[M^{\bbD,\uii}[a,b]]$ for \emph{any} complete duality datum $\bbD$ (Theorem~\ref{thm: D-quantizable for M[a,b]}).

\subsection{Exchange matrix associated with an admissible chain of $i$-boxes}
In this subsection, we shall give explicitly the exchange matrix of the seed associated with an admissible chain of $i$-boxes following \cite{KK24}.

\begin{definition} [{\cite[\S 3.2]{KK24}}] \label{def: mon_seed_ibox}
Let $\frakC=\frakC^{[a,b],\uii}$ be an admissible chain  of $i$-boxes associated with $\uii$ and range $[a,b]$.
\bnum
\item
  For an $i$-box $[x,y] = \frakc_k \in \frakC$,  there exists a unique $z\in \{x,y\}$
  such that $\{z\} = \tfrakc_k \setminus \tfrakc_{k-1} $.
  We call $z$ the \emph{effective end of $[x,y]$}. 

\item
Let
$\bbB(\frakC)=(b_{\frakc_p,\frakc_q})_{p \in \sfJ(\frakC), q\in \sfJ(\frakC)}$ 
be the  skew-symmetric matrix 
 whose positive entries
are given as follows:
\begin{align*}
& b_{[x,y],[x',y']}=   \\
& \hspace{3ex} \begin{cases} 
1  & \text{if ($x=x'$ and $y'=y_-$) or ($y=y'$ and $x'=x_-$),}  \\
1  
& \text{if  $(\al_{i_x},\al_{i_x'})=-1$ and one of the following conditions $(a)$--$(d)$ is satisfied:} \nonumber\\
\end{cases}
\end{align*}
\begin{enumerate}[(a)]
\item  $[x,y_+] \in \frakC $,  \ $x$ is the effective end of $[x,y]$,   \ $x'_- <x<x'$,  \ $y'< y_+<y'_+$,
\item $[x,y_+] \in \frakC $,   \  $y'$ is the effective end of $[x',y']$,   \ $x'_- <x$,  \ $y<y'< y_+<y'_+$,
\item $[x'_-,y'] \in \frakC $, \   $y'$ is the effective end of $[x',y']$,   \ $x_-<x'_- <x$,   \ $y<y'< y_+$,
\item $[x'_-,y'] \in \frakC$,  \   $x$ is the effective end of $[x,y]$,   \ $x_-<x'_- <x<x'$,  \ $y'< y_+$.
\end{enumerate}
We denote by $\tB(\frakC)\seteq \bbB(\frakC)\vert_{(\sfJ(\frakC) \times \sfJ(\frakC)_\ex)}$.
\item
  Let us define the monoidal seed associated with the admissible chain
$ \frakC^{[a,b],\uii} $ as follows:
\begin{equation*}
\sfS^\bbD(\frakC^{[a,b],\uii}) \seteq ( \sfM(\frakC), \tB(\frakC);  \sfJ(\frakC), \sfJ(\frakC)_\ex ),    
\end{equation*}

\item We denote by $\scrA(\frakC^{[a,b],\uii})$ the cluster algebra $\scrA(\sfS(\frakC)) \seteq \scrA([\sfS(\frakC)])$.  
\ee
\end{definition}

Since the following proposition can proved in a similar way as in \cite{KK24}
with the help of our results, we omit the proof.

\Th[{\cite[Theorem 5.20]{KK24}, see also \cite{CQW25}}] \label{thm: box and mutation}
Let $\frakC^{[a,b],\uii}$ be an admissible chain of $i$-boxes.
Then the monoidal seed $\sfS^\bbD(\frakC^{[a,b],\uii})$ is completely
$\Uplambda$-admissible.
\enth

\subsection{Quantum cluster algebra structure on $\hcalA(\ttb)$}
We fix a simply-laced simple Lie algebra $\sfg$ and the index set set $\sfI$ of
its simple roots.
Let $\bbD$ be a complete duality datum in $\Cgz$
such that $\sfg$ is the  simply-laced Lie algebra associated with $\g$.
For a $\bbD$-quantizable simple module 
$S$, we denote by
$\tch_{\bbD}(S)$
the normalized global basis element corresponding to
$S$ (see Definition~\ref{def: strong real}).
Hence we have $\tch_{\bbD}(S)=q^{-(\wt(S),\wt(S)/4}\ch_{\bbD}(S)$.

Let us take $\ttb\in\ttB^+$ and $\uii=(\im_1,\ldots,\im_r)\in\Seq(\ttb)$. 
Let $\frakC^{[1,r],\uii} = (\frakc_k)_{1 \le k \le r}$ be an admissible chain of $i$-boxes. 
Then $\bl L(\frakC^{[1,r],\uii}) ,\tB(\frakC^{[1,r],\uii})\br$ is compatible, where
$L(\frakC^{[1,r],\uii})=(L_{a,b})$ is the $[1,r] \times [1,r]$-matrix given by
$$L_{a,b}= \La( \sfM^\bbD(\frakc_a), \sfM^\bbD(\frakc_b)) \qquad \text{for $a,b \in [1,r]$.}$$
Note that $\bl L(\frakC^{[1,r],\uii}) ,\tB(\frakC^{[1,r],\uii})\br$ does not depend on the choice of $\bbD$, i.e.,
$\La( \sfM^\bbD(\frakc_a), \sfM^\bbD(\frakc_b))=
\La( \sfM^\Dcan(\frakc_a), \sfM^\Dcan(\frakc_b))$
(Proposition~\ref{prop:detLambda}).  
Here $\Dcan$ is a canonical complete duality datum (see \eqref{def:Dcan}).

Let $\scrS_t(\frakC^{[1,r],\uii})$ be the quantum seed 
$\bl \{ \tZ_j  \}_{j \in [1,r]}, L(\frakC^{[1,r],\uii}),
\tB(\frakC^{[1,r],\uii});\sfJ(\frakC^{[1,r],\uii}), \sfJ(\frakC^{[1,r],\uii})_\ex \br.$
Let us denote by $\scrA_t(\frakC^{[1,r],\uii})$ the quantum cluster algebra whose initial quantum seed is  $\scrS_t(\frakC^{[1,r],\uii})$.
Let $\scrS^\bbD(\frakC^{[1,r],\uii})=\bl\st{\sfM^\bbD(\frakc_j)}_{ j \in [1,r]},
 \tB(\frakC^{[1,r],\uii});\sfJ(\frakC^{[1,r],\uii}), \sfJ(\frakC^{[1,r],\uii})_\ex\br$ be a monoidal seed in
$\CgD(\ttb)$.

\medskip

Let $\bbT$ be the set of sequences $\sfs=(k_1,\ldots, k_m)$ ($m\in\Z_{\ge0})$
in the set $\sfJ(\frakC^{[1,r],\uii})_\ex$ of exchangeable indices.
We say that $m$ is the length of $\sfs$ and denote it by $\ell(\sfs)$.
For $\sfs=(k_1,\ldots, k_m)\in\bbT$ and an exchangeable index $k$, we set $\mu_k\sfs=(k,k_1,\ldots, k_m)$.

Let $\sfs_0\in\bbT$ be the empty sequence and
$\scrS_t^{\sfs_0} \seteq \scrS_t(\frakC^{[1,r],\uii})$.
Set
$$\scrS_t^{\sfs}\seteq \mu_{k_1}\cdots\mu_{k_m}(\scrS_t^{\sfs_0})
=( \{  \tZ_j^\sfs \}_{j \in [1,r]}, L^\sfs, \tB^\sfs )
\qt{for $\sfs\in \bbT$.} $$
It is a quantum seed in $\scrA_t(\frakC^{[1,r],\uii})$.
We have
$\mu_k\scrS_t^{\sfs}=\scrS_t^{\mu_k\sfs}$.
Let 
$\scrS^{\sfs}\seteq ( \{  Z_j^\sfs \}_{j \in [1,r]},  \tB^\sfs  )$
be its image in the cluster algebra $\scrA(\frakC^{[1,r],\uii})$ by
$\ev_{t=1}$.

Similarly, let 
$\scrS^{\bbD,\sfs_0} \seteq \scrS^\bbD(\frakC^{[1,r],\uii})$
and
$$\scrS^{\bbD,\sfs} \seteq \mu_{k_1}\cdots\mu_{k_m}\scrS^{\bbD,\,\sfs_0}
=\bl \st{M_j}_{j\in[1,r]}, \tB^\sfs\br$$
be the monoidal seeds in $\CgD(\ttb)$. Note that
the exchange matrix $\tB^\sfs$ is same in $\scrS_t^{\sfs}$ and $\scrS^{\bbD,\sfs}$.
For $\bsa=(a_j)_{j \in [1,r]}$, let
 $(\tZ^\sfs)^\bsa$ be the bar-invariant product of $(\tZ_j^{\sfs})^{a_j}$'s
 as in \eqref{eq: commutative monomial}. It is a cluster monomial in $\scrA_t(\frakC^{[1,r],\uii})$.
 Similarly let
$M^{\bbD,\sfs}(\bsa)\seteq\sotimes_{j\in[1,r]}(M_j^{\bbD,\sfs})^{\otimes\,a_j}\in\CgD(\ttb)$ be the
cluster monomial module.

\begin{theorem} \label{thm: q cluster structure}
There exists a unique $\Z$-algebra isomorphism 
$$
f_\bbD\col\scrA_t(\frakC^{[1,r],\uii})  \isoto \hcalA_\bfA(\ttb)  
$$
sending $t^{\pm1/2}$ to $q^{\mp1/2}$ and $\tZ_j$ to
$\tch_\bbD\bl M^\bbD(\frakc_j)\br$. 
Moreover, we have 
\bnum
\item $f_\bbD$ does not depend on $\bbD$,
\item
  any cluster monomial in $\scrA_t(\frakC^{[1,r],\uii})$
corresponds to a member of
the normalized global basis of $\hcalA_\bfA(\ttb)$.
More precisely, for any $\sfs\in\bbT$ and $\bsa\in\Z_{\ge0}^{[1,r]}$,
  the cluster monomial module $M^{\bbD,\sfs}(\bsa)$ is $\bbD$-quantizable and
$\tch_\bbD\bl M^{\bbD,\sfs}(\bsa)\br=f_\bbD\bl(\tZ^\sfs)^\bsa\br$.

\ee

\end{theorem}

\begin{proof}
Note that $\ch_\bbD(M^\bbD(\frakc_j))$ does not depend on $\bbD$, i.e., $\ch_\bbD(M^\bbD(\frakc_j))=\ch_\Dcan(M^\Dcan(\frakc_j))$.
Hence if $f_\bbD$ exists, then it does not depend on $\bbD$.
  
Let $\sfT(L(\frakC^{[1,r],\uii}))$ be the quantum torus associated with the matrix $L(\frakC^{[1,r],\uii})$. Note that 
$\hcalA_\bfA(\ttb)$ is a Noetherian domain by Proposition~\ref{prop: Noetherian} and hence it is an Ore domain 
(\cite[(10.23)]{Lam12}). Let $\bbF(\hcalA_\bfA(\ttb))$ be the skew-field of the fractions of $\hcalA_\bfA(\ttb)$. Then $\overline{ \phantom{a}}$ on $\hcalA_\bfA(\ttb)$ is extended to $\bbF(\hcalA_\bfA(\ttb))$. 

By Lemma~\ref{lem: L up to}, we have
$$
q^{L_{a,b}}
\tch_\bbD(M^\bbD(\frakc_a))\tch_\bbD(M^\bbD(\frakc_b)) 
= \tch_\bbD(M^\bbD(\frakc_b))\tch_\bbD(M^\bbD(\frakc_a)).
$$
Hence there is a $\Z$-algebra homomorphism 
$$
\Theta: \mathsf{T}(L(\frakC^{[1,r],\uii})) \to  \bbF(\hcalA_\bfA(\ttb))
$$
sending 
$$
t^{1/2} \longmapsto q^{-1/2} \qtq 
\tZ_j \longmapsto \tch_\bbD\bl M^\bbD(\frakc_j) \br\quad (j \in [1,r]). 
$$
Note that $\Theta$ does not depend on the choice of $\bbD$.

Let $\overline{\phantom{a}}: \sfT(L(\frakC^{[1,r],\uii})) \to \sfT(L(\frakC^{[1,r],\uii}))$ be the $\Z$-algebra anti-automorphism such that $\overline{t^{1/2}}=t^{-1/2}$ and 
$\overline{\tZ_j} = \tZ_j$ for $j \in [1,r]$. Then we have
$$
\Theta \circ \overline{\phantom{a}}
= \overline{\phantom{a}} \circ \Theta. 
$$

First, we claim that $\Theta$ is injective. Indeed,  
$$
\left\{ \Theta(\tZ^\bsa) = \tch_\bbD\bl \stens_{j} \sfM^\bbD(\frakc_j)^{\tens a_j} \br \ | \ \bsa \in \Z_{\ge 0}^{[1,r]} \right\}
$$
is linearly independent over $\bfA$. It follows that
$$
\left\{ \Theta(\tZ^\bsa) \ | \ \bsa \in \Z^{[1,r]} \right\}
\subset  \bbF(\hcalA_\bfA(\ttb))$$
is  linearly independent over $\bfA$. Since
$\left\{ \tZ^\bsa \ | \ \bsa \in \Z^{[1,r]} \right\}$
is a $\bfA$-basis of $\sfT(L(\frakC^{[1,r],\uii}))$, $\Theta$
is injective. 
 
Now, let us show
\eq&&
\parbox{73ex}{for any $\sfs\in\bbT$ and any $\bsa\in\Z_{\ge0}^{[1,r]}$,
  the cluster monomial module $M^{\bbD,\,\sfs}(\bsa)$ is $\bbD$-quantizable and
  $$\tch_\bbD\bl M^{\bbD,\,\sfs}(\bsa)\br=\tch_\Dcan\bl M^{\Dcan,\,\sfs}(\bsa)\br=\Theta(\tZ^\sfs)^\bsa \qquad \qquad\qquad$$
  }\label{eq:mainquant}
  \eneq

  by induction on $\ell(\sfs)$.

  Assuming \eqref{eq:mainquant} for $\sfs$, let us show
  \eqref{eq:mainquant} for $\mu_k\sfs$.
  
The mutation $\tZ_k^{\mu_k \sfs}$ of $\tZ_k^\sfs$
satisfies 
\begin{align} \label{eq: tX com}
\tZ_k^\sfs \;\tZ_k^{\mu_k \sfs} = t^a (\tZ^\sfs)^{\bsb'}
+t^b (\tZ^\sfs)^{\bsb''}     
\end{align}
for some $a,b \in \Z/2$ and $\bsb',\bsb''\in\Z_{\ge0}^{[1,r]}$.

On the other hand, let
$$
0 \to A^\bbD \to M_k^{\bbD,\,\sfs} \tens M_k^{\bbD,\;\mu_k \sfs } \to B^\bbD \to 0
$$
be the short exact sequence in $\CgD(\ttb)$ which yields the exchange relation between the cluster variables $Z_k^\sfs$ and $Z_k^{\mu_k\sfs}$.
Hence we have
\eqn
\ev_{q=1}\tch_\bbD( A^\bbD)&&=\ev_{q=1}\Theta((\tZ^\sfs)^{\bsb'})\br,\\
\ev_{q=1}\tch_\bbD( B^\bbD)&&=\ev_{q=1}\Theta\bl(\tZ^\sfs)^{\bsb''}\br.\eneqn
Since they are normalized global basis members by the induction hypothesis on $\ell(\sfs)$, we conclude that
\eqn
\tch_\bbD( A^\bbD)&&=\Theta((\tZ^\sfs)^{\bsb'})\br\qtq
\tch_\bbD( B^\bbD)=\Theta\bl(\tZ^\sfs)^{\bsb''}\br.\eneqn

Let us apply it to $\Dcan$. Since any simple module in $\scrC_{\sfg^{(1)}}^0$  is $\Dcan$-quantizable, we obtain
the equality
in $\hcalA_\bfA(\ttb)$:
$$
\tchDcan(M_k^{\Dcan,\sfs})\tchDcan(M_k^{\Dcan,\;\mu_k\sfs}) = \sum_{S}
a_S(q)\tchDcan(S)
$$
with $a_S(q) \in \Z_{\ge 0}[q^{\pm 1/2}]$.
Here, $S$ ranges over the set of the isomorphism classes
of simple modules in $\scrC^0_{\sfg^{(1)}}$
(see Corollary~\ref{cor:gbpositive}). 
Since the application of $\obPhi_\Dcan$ should yield 
$$
[M_k^{\Dcan,\sfs}]\;[M_k^{\Dcan,\;\mu_k\sfs}] = [A^\Dcan]+[B^\Dcan],
$$
we can conclude that 
$$
\tchDcan(M_k^{\Dcan,\;\sfs})\tchDcan(M_k^{\Dcan,\;\mu_k\sfs})  =
q^c \tchDcan(A^\Dcan)+ q^d \tchDcan(B^\Dcan) \quad \text{for some $c,d \in \Z/2$.}
$$
Thus, by applying $\Theta$ to~\eqref{eq: tX com}, we obtain
$$
\tchDcan(M_k^{\Dcan,\;\sfs}) \Theta(M_k^{ \Dcan,\;\mu_k\sfs}) = 
q^{-a}\tchDcan(A^\Dcan)+ q^{-b}\tchDcan(B^\Dcan). 
$$

In the skew-filed $\bbF(\hcalA_\bfA(\ttb))$, the two elements
$$
 \tchDcan(M_k^{\Dcan,\;\mu_k\sfs}) = 
q^{c} \tchDcan(M_k^{\Dcan,\;\sfs})^{-1} \tchDcan(A^\Dcan)+
q^{d} \tchDcan(M_k^{\Dcan,\,\sfs})^{-1} \tchDcan(B^\Dcan).
$$
and 
$$
\Theta(\tZ_k^{\mu_k \sfs}) = q^{-a}\tchDcan(M_k^{\Dcan,\;\sfs})^{-1}\tchDcan(A^\Dcan)+
 q^{-b} \tchDcan(M_k^{\Dcan,\,\sfs})^{-1} \tchDcan(B^\Dcan)
$$
are both $\overline{\phantom{a}}$-invariant. 
Because $\tchDcan(M_k^{\Dcan,\;\sfs})^{-1}\; \tchDcan(A^\Dcan)$
and $\tchDcan(M_k^{\Dcan,\,\sfs})^{-1}\tchDcan(B^\Dcan)$ are linearly independent over $\bfA$ in $\bbF(\hcalA_\bfA(\ttb))$, we can conclude that $c=-a$ and $d=-b$. Hence
$$
\Theta(\tZ_k^{\mu_k \sfs})  = \tchDcan(M_k^{\Dcan,\;\mu_k\sfs}). 
$$

Now, we have
\eqn
[M^{\bbD,\, \sfs}_k]\cdot[M_k^{\bbD,\;\mu_k\sfs}]=[A^\bbD]+[B^\bbD]
=[A^\Dcan]+[B^\Dcan]=Z_k^\sfs \;Z^{\mu_k\sfs}_k.
\eneqn
Here we identify $\scrA( \frakC^\uii )$ and $\obbA(\ttb)$ with $K(\Cg^{\bbD}(\ttb))$ via the isomorphisms in Theorem \ref{thm: mCat of cluster}~\eqref{it: isoClsuter}.
Since $[M^{\bbD,\,\sfs}_k]=Z_k^\sfs$, we obtain
$[M_k^{\bbD,\;\mu_k\sfs}]=Z^{\mu_k\sfs}_k=\ev_{q=1}\bl\Theta(\tZ_k^{\mu_k\sfs})\br$.
Since $Z^{\mu_k\sfs}_k$ belongs to $\ev_{q=1}(\tbfG)$,
we conclude that
$M^{\bbD,\;\mu_k\sfs}_k$ is $\bbD$-quantizable and $\tch_\bbD(M_k^{\bbD,\;\mu_k\sfs})
=\tZ_k^{\mu\sfs}$.
Thus the induction proceeds and we obtain \eqref{eq:mainquant}.

\medskip
The assertion \eqref{eq:mainquant} implies that
the image of $\scrA_t(\frakC^{[1,r],\uii})$ by $\Theta$ is contained in
$\hcalA_\bfA(\ttb)$ and hence $\Theta$
induces  an injective $\bfA$-algebra homomorphism
$$f\col\scrA_t(\frakC^{[1,r],\uii})  \monoto \hcalA_\bfA(\ttb).$$
Since $\sfM^\bbD(\frakc)$ is a cluster monomial module of
$\CgD(\ttb)$ for any $i$-box $\frakc$, the image of $f$ contains
$\tch_\bbD( \sfM^\bbD(\frakc))$.
Since  $\hcalA_\bfA(\ttb)$ is generated by the $\tch_\bbD( \sfM^\bbD(\frakc))$'s
as a $\Z[q^{\pm1/2}]$-algebra,
we conclude that $f$ is surjective.
\end{proof}

\Cor\label{cor:D-q}
  Let $\bbD$ be a complete duality datum.
Any cluster monomial module in $\CgD(\ttb)$
is $\bbD$-quantizable. 
\encor

For the rest of this section, we consider $U_q'(\g)$ of untwisted affine type.

\begin{definition}
We say that monoidal seed $\sfS= (\{ \sfM_i\}_{i\in \sfJ },\tB; \sfJ,\sfJ_\ex)$ is \emph{quantizable} if 
each $\sfM_k$ is quantizable.
\end{definition}

For a quantizable monoidal seed $\sfS=(\{ \sfM_i\}_{i\in \sfJ },\tB; \sfJ,\sfJ_\ex)$, Lemma~\ref{lem: L up to} says that 
\begin{align} \label{eq: Lam ij}
\sfM_{i;t}\sfM_{j;t} =t^{\La(\sfM_i,\sfM_j)} \sfM_{j;t}\sfM_{i;t} \ \ \text{ in } \calK_{\g;t},  \quad \text{where $\sfM_{k;t} \seteq [\sfM_k]_t$ for $k \in \sfJ$.}    
\end{align}

Lemma~\ref{lem: chi quantizable} and Corollary~\ref{cor:D-q} imply the following corollary:

\begin{corollary} \label{cor: q-seed} For any $\rmQ$-datum $\calQ$ of $\g$,
every monoidal seed in $\scrC^{\DQ}_\g(\ttb)$, obtained from $\sfS^{\DQ}(\frakC^{\uii})$,
is a completely $\Uplambda$-admissible and quantizable monoidal seed. 
\end{corollary}

\begin{definition} \label{def: mc for q clsuter}
Let $\scrC$ be a monoidal subcategory of $\scrC_\g^0$.
\bnum
\item $\calK_t(\scrC)$ denotes the subalgebra of $\calK_{\g;t}$
generated by $[L]_t$ for all simple modules $L$ in $\scrC$.
\item $\scrC$ is called a \emph{monoidal categorification} of a \emph{quantum} cluster algebra $\scrA_t$ if
\bna
\item the ring $\calK_t(\scrC)$ is isomorphic to $\scrA_t$,
\item there exists a completely $\Uplambda$-admissible
  and quantizable monoidal seed \\
$\Mseed$ in $\scrC$ such that
$[\sfS]_t \seteq (\{ \sfM_{i;t} \}_{i \in \sfK}, \La^{\sfS}, \tB)$
is  an initial  quantum seed of $\scrA_t$.
\ee
\ee
\end{definition}

\begin{theorem} \label{Thm: MC for Kt}
For $\bbD = \DQ$ of untwisted affine type, the category $\CgD(\ttb)$ provides a monoidal categorification of the quantum cluster algebra $\calK_t(\CgD(\ttb)) \simeq \scrA_t(\frakC^{[1,r],\uii})$.     
\end{theorem}

\begin{proof}
It is enough to prove that $\calK_t(\scrC_\g^{\DQ}(\ttb)) \simeq \hcalA_\Zq(\ttb)$ by Corollary~\ref{cor:D-q} and Theorem~\ref{thm: q cluster structure}.
Theorem~\ref{thm: auto of calK} tells that there is an algebra isomorphism
$$
\Psi_\DQ : \hcalA_\Zq \isoto \calK_{\g;t}.
$$
Since every cuspidal module $C^\uii_s$ is quantizable, we have
$$
\Psi_\DQ(\ttP_s^\uii) = [C^\uii_s]_t \quad \text{ for any }s \in [1,r].
 $$
As $\calK_t(\scrC_\g^{\DQ}(\ttb))$ (resp. $\hcalA_\Zq(\ttb)$) is generated by $\ttP_s^\uii$ (resp. $[C^\uii_s]_t$) for $t \in [1,r]$, the restriction
$\Psi_\DQ$ to  $\hcalA_\Zq(\ttb)$ gives an isomorphism 
\[
\Psi_\DQ |_{\hcalA_\Zq(\ttb)} : \hcalA_\Zq(\ttb) \isoto \calK_t(\scrC_\g^{\DQ}(\ttb)). \qedhere 
\]
\end{proof}

\providecommand{\bysame}{\leavevmode\hbox to3em{\hrulefill}\thinspace}
\providecommand{\MR}{\relax\ifhmode\unskip\space\fi MR }
\providecommand{\MRhref}[2]{%
  \href{http://www.ams.org/mathscinet-getitem?mr=#1}{#2}
}
\providecommand{\href}[2]{#2}

\end{document}